\documentclass[11pt]{article}

\usepackage{graphicx,,psfrag,mathtools,leftindex,bm}

\usepackage{amssymb,amsmath,amsthm,enumerate}
\usepackage{fullpage}
\usepackage{bbm}
\newcommand{\url}{}

\usepackage{tikz}
\usepackage{authblk}
\numberwithin{equation}{section}

\usepackage[margin=1.1in]{geometry}

\makeatletter

\def\@settitle{\begin{center}%
    \bfseries
 \normalfont\LARGE\@title
  \end{center}%
}
\def\@setauthors{\begin{center}%
 \normalsize\@author
  \end{center}%
}

\makeatother

\RequirePackage[colorlinks,citecolor=blue,urlcolor=blue]{hyperref}

\theoremstyle{plain}
\newtheorem{theorem}{Theorem}[section]
\newtheorem{lemma}[theorem]{Lemma}
\newtheorem{remark}[theorem]{Remark}
\newtheorem{proposition}[theorem]{Proposition}
\newtheorem{definition}[theorem]{Definition}

\newtheorem{corollary}[theorem]{Corollary}

\newcommand{\DKL}{D_{\mathrm{KL}}}

\newcommand{\tr}{{\rm tr}}
\newcommand{\Tr}{{\rm Tr}}
\newcommand{\dE}{\mathbb {E}}
\newcommand{\dN}{\mathbb {N}}
\newcommand{\dP}{\mathbb{P}}
\newcommand{\dZ}{\mathbb {Z}}
\newcommand{\dR}{\mathbb {R}}
\newcommand{\dC}{\mathbb {C}}

\newcommand{\cG}{\mathcal {G}}
\newcommand{\cX}{\mathcal {X}}

\newcommand{\cP}{\mathcal {P}}
\newcommand{\cN}{\mathcal {N}}
\newcommand{\cC}{\mathcal {C}}
\newcommand{\cS}{\mathcal {S}}
\newcommand{\cD}{\mathcal {D}}

\newcommand{\cZ}{\mathcal {Z}}

\newcommand{\cGr}{\mathcal {G}^\bullet}
\newcommand{\cGe}{\vec{\mathcal {G}}^\bullet}

\newcommand{\cNr}{{\mathcal {N}}^\bullet}

\newcommand{\cNe}{\vec{\mathcal {N}}^\bullet}

\newcommand{\E}{\mathbb{E}}

\renewcommand{\P}{\mathbb{P}}
\newcommand{\cB}{\mathcal{B}}

\newcommand{\DEG}{{\mathrm{Deg}}}
\newcommand{\MM}{{\mathfrak{m}}}
\newcommand{\UGW}{{\mathrm{UGW}}}
\newcommand{\PWIT}{{\mathrm{PWIT}}}

\newcommand{\Utau}{\mathcal T}

\newcommand{\TRAF}{{\mathrm{Traf\langle J \rangle}}}
\newcommand{\TRAFB}{{\mathrm{Traf^\bullet\langle J \rangle}}}

\newcommand{\TV}{ \mathrm{TV}}

\newcommand{\Poi}{ \mathrm{Poi}}

\newcommand{\LOC}{{\mathrm{loc}}}

\newcommand{\DLOC}{{ \mathrm{d}_{\LOC}}}

\newcommand{\DD}{\mathfrak{d}}

\newcommand{\AND}{\quad \mathrm{and} \quad}

\newcommand{\1}{1\!\!{\sf I}}
\newcommand{\IND}{\1}
\newcommand{\veps}{\varepsilon}
\newcommand{\eps}{\varepsilon}

\newcommand{\cH}{\mathcal{H}}

\newcommand{\BEAS}{\begin{eqnarray*}}
\newcommand{\EEAS}{\end{eqnarray*}}
\newcommand{\BEA}{\begin{eqnarray}}
\newcommand{\EEA}{\end{eqnarray}}
\newcommand{\BEQ}{\begin{equation}}
\newcommand{\EEQ}{\end{equation}}
\newcommand{\BIT}{\begin{itemize}}
\newcommand{\EIT}{\end{itemize}}
\newcommand{\BNUM}{\begin{enumerate}}
\newcommand{\ENUM}{\end{enumerate}}

\newcommand{\mbf}{\mathbf}

\renewcommand{\leq}{\leqslant}
\renewcommand{\geq}{\geqslant}

\newcommand{\beq}{\begin{equation}}
\newcommand{\eeq}{\end{equation}}

\title{Large deviations for  macroscopic observables of heavy-tailed matrices}

\author{Charles~Bordenave\thanks{Aix-Marseille Univ, CNRS, I2M, Marseille, France. Email: charles.bordenave@univ-amu.fr}\,,  \; Alice Guionnet\thanks{ENS Lyon, CNRS, UMPA, Lyon, France. Email: aguionne@ens-lyon.fr} \; and \;  Camille~Male\thanks{Univ. Bordeaux, CNRS, IMB, Bordeaux, France. Email: camille.male@gmail.com}}

\date{}

\begin{document}

\maketitle

\begin{abstract} We consider a finite collection of independent Hermitian heavy-tailed random matrices of  growing dimension. Our model includes the L\'evy matrices proposed by Bouchaud and Cizeau, as well as sparse random matrices with  $O(1)$ non-zero entries per row. By representing these matrices as weighted graphs, we  derive a large deviations principle for key macroscopic observables. Specifically, we focus on the empirical distribution of eigenvalues, the joint neighborhood distribution, and the joint traffic distribution. As an application, we define a notion of microstates entropy for traffic distributions which is additive for the free traffic convolution.  
\end{abstract}

\section{Introduction}

\subsection{Large deviations for empirical spectral distribution of random matrices}

\paragraph{Background} Consider a self-adjoint matrix $Y \in M_{n} (\dC)$ of dimension $n$. We classically define its  {\em empirical spectral distribution} (ESD) as the probability measure on $\dR$ which puts an equal mass to all its eigenvalues counted with multiplicities : 
\begin{equation}\label{def:ESD}
L_Y = \frac 1 n \sum_{k=1}^n \delta_{\lambda_k(Y)} \in \cP(\dR). 
\end{equation}
Given a sequence of matrices $Y = Y_n$ of growing dimensions, a classical question 
is to study  the convergence, say  for the weak topology, of $L_Y$ as the dimension grows toward a limiting probability measure $L_{\star}$. When $Y$ is a random matrix, such convergence could be in expectation or in probability. This convergence of the ESD is well-understood for essentially all classical models of self-adjoint random matrices, see the monographs \cite{BaiSil,MR2760897,zbMATH06684673}. 

Much less is known however on the probability of  large deviations of the ESD away form  its typical behavior $L_{\star}$ beyond a short list of probabilistic ensembles. Before going through this list, let us recall briefly the definition of a {\em large deviation  principle } (LDP) for a sequence of random variables which was introduced by Varadhan \cite{Varadhan}. Let $(\cX,d)$ be a metric space. A {\em rate function} $I : \cX \to [0,\infty]$ is a lower semi-continuous function. A rate function is {\em good} if for all $t \in [0,\infty)$, $I^{-1} ([0,t])$ is a compact subset of $\cX$. Then, a sequence of random variables $X_n \in (\cX,d)$ (here $X_n = L_{Y}$ and $\cX = \cP(\dR)$) satisfies a LDP at rate $v_n$ with $\lim_n v_n = \infty$ and rate function $I$, if for every measurable set $A \subset \cX$, 
$$
-\inf_{x \in \dot A} I(x) \leq \liminf \frac{1}{v_n} \ln \dP ( X_n \in A) \leq \limsup \frac{1}{v_n} \ln \dP ( X_n \in A) \leq - \inf_{ x \in \bar{A} } I(x),
$$
where $\dot A$ and $\bar A$ are the interior and closure of $A$ respectively. For references on this classical topic of probability theory, see \cite{Deu-St,De-Ze,DuEl}.

Sanov's theorem asserts that the rate of the LDP for 
 the empirical measure of i.i.d random variables is always $n$ and therefore does not depend on their distribution. The situation is very different for the LDP for ESD of  random matrices  where it depends on the tail of the entries.
We now list the known LDP  for ESD of random matrices, ordered  by decreasing rates. The LDP for the ESD at rate $n^2$ of the Gaussian unitary ensemble and more general orthogonal or unitary invariant matrix ensembles was established in \cite{BAG97}. Recently \cite{BHG,NaSh24} established the LDP with rate $n^2$ for the ESD of sums of matrices in typical positions : $A + U B U^*$ where $A,B$ are self-adjoint and $U$ is random unitary sampled according to the Haar measure on the unitary group. In \cite{BordCap}, the LDP with rate $n^{1+ \alpha/2}$ was proved for Wigner matrices $Y =  X/ \sqrt n$ where the entries of $X$ have an exponential tail of the form $\dP ( |X_{ij}| \geq t ) \sim_{t \to \infty} c_1 \exp ( -c_2 t^{\alpha})$ for some $0 < \alpha < 2$.  In the recent work  \cite{augeri2024}, Augeri
 proves the LDP for the ESD for sparse Wigner matrices $Y_{ij} = A_{ij} X_{ij} / \sqrt{np} $ where $X$ is self-adjoint with iid bounded centered entries above the diagonal and $A$ is self-adjoint with iid Bernoulli $\{0,1\}$-entries above the diagonal, $\dP (A_{ij} = 1) = p$. In the regime $\log n \ll  p \ll 1$,  Augeri proves  that $L_{Y}$ satisfies a  LDP with  rate  $n^2p$. Moreover, from the references \cite{BoCa15} and \cite[pp 103-107]{MR3791802} , a LDP at rate $n$ for the ESD of the adjacency matrix can be extracted for the above matrix $A$ in the regime $p = d/n$ with $d$ fixed or for adjacency matrices of random graphs uniform given their sequence of degrees, provided that the degree sequence converges. 

A first goal of the present work is to present new large deviation  principles for the ESD of random matrices at rate $n$.

\paragraph{Weighted graphs with uniformly bounded average degrees}

We start with sparse Wigner matrices with order $1$ non-zeroes entries per row. More precisely, for every  integer $n \geq 1$, let 
$ (A_{ij})_{1 \leq i < j \leq n}$ be iid independent Bernoulli variables, $\dP(A_{ij} = 1) = 1 - \dP( A_{ij} = 0) = \min(d/n,1) $ for some fixed $d > 0$. We set $A_{ji} = A_{ij}$ and $A_{ii} = 0$. Let $\gamma$ be a probability measure on $\dC$ which is invariant by complex conjugation (that is  $\gamma (B) = \gamma (\bar B)$ for all Borel sets $B\subset \dC$).  We consider the self-adjoint matrix $X = (X_{ij})$ with zero diagonal entries $X_{ii} = 0$ and for all $i > j$, $(X_{ij})$ iid with law $\gamma$. Finally, as above we set 
\begin{equation}\label{eq:SW}
Y_{ij} = A_{ij} X_{ij},
\end{equation}
(in other words, for $n\geq d$, $Y_{ij}$ has law $(1-d/n) \delta_0 + (d/n) \gamma$ outside the diagonal). By construction $A$ is the adjacency  matrix of an Erd\H{o}s-R\'enyi random graph on $n$ vertices and edge probability $p/n$.

Note that by definition, the rate function $I$ of a LDP  is non-negative and $\inf I=0$. If $I$ is  a good rate function, this infimum is achieved.
Here and below, we will say that a rate function $I$ on $\cX$ has a {\em unique minimizer} if there is a unique $x \in \cX$ such that $I(x) = 0$.  This property 
is important as it implies the almost sure convergence of the random variable satisfying this LDP towards the unique minimizer $x$.
\begin{theorem}\label{thA}
Let $d > 0$, $\gamma \in \cP(\dC)$ be invariant by complex conjugation  and $Y$ the Wigner matrix with entries defined in \eqref{eq:SW}. The ESD $L_{Y}$ satisfies a LDP in $\cP(\dR)$ with speed $n$ and good rate function $J_{\gamma,d}$  with a unique minimizer $\mu_{\gamma,d}\in\cP(\dR)$.
\end{theorem}
 Interestingly, there is no tail assumption on the  law $\gamma$ for the the conclusion of Theorem \ref{thA} to hold.  
As explained below, the rate function $J_{\gamma,d}$ will be obtained by a contraction principle from an explicit rate function on a larger space of marked graphs (see Section \ref{secproofAB}). Therefore, the minimizer 
$\mu_{\gamma,d}$ will be described as the push-forward of the unique minimizer of a rate function on this larger space, namely a random marked tree $ \UGW(\gamma,d)$ (see Lemma \ref{le:minimizer}). This minimizer also characterizes the almost sure limit of 
the ESD $L_{Y}$ already described in \cite{BGM,BoLe}

The proof of Theorem \ref{thA} is very robust, we may for example consider Laplacian instead of adjacency operators for the conclusion to hold. We may also consider the case where $A$ is the adjacency matrix of  other random graphs than  the Erd\H{o}s-R\'enyi random graph. More precisely, let $(D_1(1),\ldots,D_n(n))$ be an integer sequence such that 
$\sum_v D_n(k) $ is even and $1 \leq D_n(v) \leq \theta$  for some fixed $\theta >0$. Assume further that any integer $k \geq 0$,
$$
\lim_{n \to \infty} \frac{1}{n} \sum_{v=1}^n \IND(D_n(v) = k) = \pi(k),
$$
for some probability measure $\pi$ in $\cP(\dZ_+)$. Let $G_n$ be a random graph uniformly sampled on the set of simple graphs on the vertex set $V_n = \{ 1,\ldots,n\}$ such that for all $v \in V_n$, the degree of $v$ is $D_n(v)$. This set is non-empty for  $n$ large enough. Let $A$ be the adjacency matrix of $G_n$ and define the self-adjoint matrix $Z = (Z_{ij})$ with $Z_{ij} = A_{ij} X_{ij}$ and $X$ independent of $A$ as above.

\begin{theorem}\label{thB}
Let $\pi \in \cP(\dZ_+)$ be as above and $\gamma \in \cP(\dC)$ be invariant by complex conjugation. The ESD $L_{Z}$ satisfies a LDP in $\cP(\dR)$ with speed $n$ and good rate function $J_{\gamma,\pi}$  with a unique minimizer $\mu_{\gamma,\pi}$.
\end{theorem}
This result is proved in Section \ref{secproofAB}.

\paragraph{Heavy-tailed random matrices}  We consider a random Hermitian matrix $Y_n$ such that $(Y_{ij})_{i > j}$ are iid with law $\gamma_n$,{{ the diagonal entries $(Y_{ii})_i$ are independent variables, and they are  independent of $(Y_{ij})_{i > j}$. }}We assume that $\gamma_n$ is invariant by complex conjugation and that for every Borel sets $A \subset \dC$, $0 \notin A$, 
\begin{equation}\label{eq:gamman2N}
\lim_{n \to \infty} n \gamma_{n}(A) = \Lambda (A),
\end{equation}
where $\Lambda$ is a non-trivial Radon measure on $\dR^k \backslash \{0\}$ which we call the intensity measure of the model. We assume that $\Lambda$ is finite at infinity, that is, if $B_1$ is the unit ball of $\dR^k$, we assume that
\begin{equation}\label{eq:Lambdabdd}
\Lambda(\dR^k \backslash B_1) < \infty.
\end{equation}

Sparse Wigner matrices $Y_{ij}=A_{ij} X_{ij}$ defined in the previous paragraph is contained in this model with $\Lambda = d \gamma'$ where $\gamma' = \gamma - \gamma(\{0\}) \delta_0$.  Importantly, the L\'evy matrices fit also into this framework. They are defined as follows. Consider an infinite triangular array $(X_{ij})_{i \geq j}$ which are iid copies of a random variable $X$ on $\dR$ and for some $\alpha\in (0,2)$,  some $(p,q)\in [0,1]^{2}$ so that $p+q = 1$ and $c > 0$, $ \dP ( X \geq t ) \sim_{t \to \infty}  p c t^{-\alpha}$ and $\dP ( X \leq - t ) \sim_{t \to \infty}  q c t^{-\alpha}$. 
We set $X_{ji} = X_{ij}$ for $i \leq j$ and define the symmetric matrix $Y = (Y_{ij})_{1 \leq i,j \leq n}$ with $$Y_{ij}=\frac{X_{ij}}{ (cn)^{1/\alpha}}.$$
It is immediate to check that the law $\gamma_n$ of $Y_{ij}$ satisfies \eqref{eq:gamman2N} with $\Lambda$ given by 
\begin{equation}\label{eq:LambdaHT}
\Lambda (dt) =  ( \IND( t > 0 ) p + \IND( t < 0 ) q) \alpha |t|^{-\alpha -1}  dt.
\end{equation}
This measure satisfies \eqref{eq:Lambdabdd} for any $\alpha > 0$.

Back to the general setup, we assume the tail assumption: for some $c_0 > 0$ and $0 < \alpha < 2$, for all $n \geq 1$, $i,j$,
\begin{equation}\label{eq:tailalpha2}
n \dP ( |Y_{ij}| \geq t ) \leq c_0 t^{-\alpha} \quad \hbox{ for all $0 < t  \leq 1$}.
\end{equation}

 \begin{theorem}\label{LDPHT} Let $\gamma_n \in \cP(\dC)$ be invariant by complex conjugation. Assume  that \eqref{eq:gamman2N}-\eqref{eq:Lambdabdd} and \eqref{eq:tailalpha2} hold for some $0 < \alpha < 2$. The ESD $L_{Y}$ satisfies a LDP in $\cP(\dR)$ with speed $n$ and good rate function  $J_{\Lambda}$ with a unique minimizer $\mu_{\Lambda}$.
  \end{theorem} 
This theorem is also obtained by the contraction principle thanks to a large deviation principle on a larger space, namely networks. It is proved in Section \ref{sec:LDPHT}.
The minimizer $\mu_{\Lambda}$ is thus also described as the push-forward of the unique minimizer of this rate function, namely Aldous's PWIT (see Lemma \ref{le:minimizerPWIT}). This minimizer was shown to be the limit of the ESD \cite{BCC}. In the special case of L\'evy matrices, a different characterization of this limit was provided in \cite{BAG1}.

\paragraph{The rate function for regular weighted graphs} 
In this paragraph, we give an expression for the rate function $J_{\gamma,\pi}$ which appears in  Theorem \ref{thB} in the simplest example where  $\pi_d = \delta_d$ is a Dirac mass at some integer $d \geq 2$ and the support of $\gamma$ is a compact subset of $\dR$. The general expression of the rate functions in the above theorems requires to introduce more notions and we postpone its definition to Section \ref{sec:LDPHT}.

Let $T = (V,E)$ be an infinite $d$-regular tree. We equip the  set of probability measures on $\dR^E$, $\cP(\dR^E)$, with the product topology. A real random process indexed by $E$, say $A = (A_e)_{e \in E}$, is invariant if its law is left invariant by all automorphisms of $T$ (that is bijections  $\varphi : V \to V$ such that $\varphi(T)= T$). We denote by $\cP_{\gamma} (\dR^E) \subset \cP ( \dR^E)$, the set of laws of invariant processes such that for any finite subsets $F \subset E$ of edges, the restriction of $A$ to $F$ has a finite relative entropy with respect to the product measure $\gamma^{\otimes F}$.  Recall that the relative entropy between probability measures $p,q$ is given by $\DKL(p|q) = \int d p(z) \ln (dp(z)/dq(z))$.

Let $o \in V$ be a distinguished vertex and $\rho = \{a,b\} \in E$ be a distinguished edge. For $h \geq 1$, let $B_h$ be the ball of radius $h$ around $o$ and $\vec B_h$ be the ball of radius $h-1$ around $\rho$ (the union of the two ball of radius $h-1$ around $a$ and $b$). We let $\mu_h$ and $\vec \mu_{h}$ be the laws of the restriction of $A$ to the balls of radius $h$ around the roots. For $\mu \in \cP(\dR^{E})$ we define its entropy as 
$$
\Sigma_{\gamma,\pi_d}( \mu) = \left\{\begin{array}{ll}
\lim_{h \to \infty}\left( \DKL(\mu_h|\gamma^{\otimes B_h})- \frac d 2  \DKL(\vec \mu_h|\gamma^{\otimes \vec B_h}) \right) & \hbox{if $\mu \in \cP_{\gamma}(\dR^{E})$}\\
 \infty & \hbox{\hbox{otherwise.}}
\end{array}\right. 
$$ 
We shall prove that the above limit is well-defined: in fact, for $\mu \in \cP_\gamma(\dR^E)$, $\DKL(\mu_h|\gamma^{\otimes B_h})- \frac d 2  \DKL(\vec \mu_h|\gamma^{\otimes \vec B_h})$ is non-decreasing in $h$. The entropy $\Sigma_{\gamma,\pi_d}$ will appear as the good rate function of an LDP on random weighted regular graphs. The unique global minimizer of $\Sigma_{\gamma,\pi_d}$ is the product measure $\gamma^{\otimes E}$.

Now, to each element  $A = (A_e)_{e \in E} \in \dR^{E}$ with uniformly bounded coordinates, $\sup_e |A_e| < \infty$, we can associate a bounded self-adjoint operator on $\ell^2 (V)$, which we also denote by $A$, by the following  formula that  for every $\psi \in \ell^2(V)$ and $u \in V$,
$$
(A \psi )(u) = \sum_{ v : \{u,v\} \in E} A_{\{u,v\}} \psi (v).  
$$
From the spectral theorem, we can then define its spectral measure at vector $\psi$ as the unique probability measure $L_A^{\psi} \in \cP(\dR)$ such that for every bounded continuous functions $f$,
$$
\langle \psi, f(A)\psi \rangle = \int_{\dR} f(\lambda) d L_A^{\psi} (\lambda).
$$

Since $\gamma$ has bounded support, the above condition $\sup_e |A_e| < \infty$ holds a.s. if $\mu \in \cP_{\gamma}(\dR^{E})$ and $A$ has law $\mu$. If $(e_v)_{v \in V}$ is the canonical basis of $\ell^{2}(V)$, we set 
$$
L_\mu = \dE_{\mu} [ L^{e_o}_A] = \int  L^{e_o}_A d \mu(A).
$$

We are finally ready to give an expression for the rate function $J_{\gamma,\pi_d}$ which appears in Theorem \ref{thB}: for $p \in \cP(\dR)$,
\begin{equation}\label{def:Jgamma}
J_{\gamma,\pi_d} (p) = \inf \left\{ \Sigma_{\gamma,\pi_d}( \mu) : \mu \in \cP_\gamma(\dR^{E}) \hbox{ such that } L_\mu = p \right\},
\end{equation}
with the usual convention that the inf over an empty set is $\infty$.
This expression is already rather delicate to understand: neither the map $\mu \to L_\mu$ nor the map $\mu \to  \Sigma_{\gamma,\pi_d}( \mu)$ are  straightforward to compute. In general, the formula for the rate functions $J_{\gamma,d}$, $J_{\gamma,\pi}$ and $J_{\Lambda}$ in Theorem \ref{thA}, Theorem \ref{thB} and Theorem \ref{LDPHT} are of the same nature but will require   more care to be properly defined.

\subsection{A microstates entropy for heavy-tailed traffic distribution} \label{traffic:int}

\paragraph{Background}
The notion of traffics \cite{Camille} generalizes the notion of non-commutative variables introduced by Voiculescu in the framework of free probability, see the monographs \cite{MR1217253,MR2760897,zbMATH06684673}.  Voiculescu defined several notions of entropy in this framework, mainly the microstates free entropy $\chi$ and the free entropy $\chi^*$, see \cite{zbMATH00703862,zbMATH01186160}. Even though these two notions were expected  to match, it could only be proved so far that one is bounded by  the other \cite{BCG03,dabrowski,jekel}. The microstates entropy is defined as the rate function of a large deviation  principle  for the non-commutative distribution of Gaussian matrices, namely as the volume of matrices whose non-commutative distribution approximates a given non-commutative law. One of the main difficulty is that a full large deviation  principle  has  not yet been proven, in particular that the limsup of these volumes can be replaced by a liminf. As  a consequence, one of the main expected property that  the entropy  sums under freeness is still an open question in general. 

The introduction of traffic distributions was motivated by study of the non-commutative distribution of heavy-tailed matrices and  sparse matrices \cite{Camille2}. While they allow to generalize the later, they also give a more appropriate framework to deal with such matrices which have no unitary invariance, but rather invariance under permutation multiplied entry-wise by bounded matrices. This distribution symmetry is relevant for matrices that are associated to random marked graphs. On the other hand, the data of the traffic distribution is equivalent to the neighborhood distribution of these graphs. 

In this section, we show how the large deviation  principle  for the law of a Erd\"os Renyi marked graphs implies a large deviation  principle  for their traffic distribution, by a contraction principle. The rate function for these large deviations defines an entropy in the traffic framework that we call the entropy of sparse traffics. We show that it sums under independence of traffics \cite{Camille2, Camille}. 

\paragraph{Traffic distribution} Let us recall the notion of {\em traffic distribution} of a collection $\mbf Y$ of matrices in $M_n(\dC)$. It is encoded in the linear form defined below, which contains more information than the usual non commutative distribution of $\mbf Y$ seen  in the $*$-probability space $M_n(\dC)$ equipped with the tracial state $\frac 1 n \tr (\cdot)$. 

Let us fix a label set $J$.  We call \emph{test graph} labeled in $J$ a finite connected graphs $H=(V,E,\ell, \varepsilon)$ with possibly multiple edges, where $V$ is a nonempty set and each edge $e \in E$  has labels $\ell(e)\in J$ and $\varepsilon(e)\in \{1,*\}$. We use this terminology since we use them as elementary functions in order to understand $\mbf Y = (Y_j)_{j\in J}$, for which an edge $e\in E$ is associated to the matrix $Y_{\ell(e)}^{\varepsilon(e)}$ ($Y_j^*$ denotes the conjugate transpose of $Y_j$). In the definitions of this subsection, the graphs are directed, and they may have self-loops and multiple edges, each with its own labels.  An \emph{isomorphism} between two test graphs $H$ and $H'$ is a bijection between the sets of vertices of $H$ and $H'$ which  preserves the adjacency structure and the edge labels. We denote by $\cH \langle J\rangle$ the set of test graphs up to isomorphisms, by $\dC \cH \langle J\rangle$ the linear space it generates (the elements of the set form a basis of the space). Then an (algebraic) \emph{traffic distribution} is an element of the linear dual space $\TRAF:={\dC\cH}\langle J \rangle^*$ of linear forms on $\dC\cH\langle J \rangle$. If $\mbf Y=(Y_j)_{j\in J}$ is a collection of matrices in $M_n(\dC)$, then their canonical traffic distribution $\tau_{\mbf Y} \in \TRAF$ is defined by
	\begin{equation}\label{traf:def}\tau_{\mbf Y}[H]= \frac 1 n \sum_{  \phi:V\to \{1,\ldots,n\}}    \prod_{e = (v,w)\in E} Y_{\ell(e)}^{\varepsilon(e)}\big(\phi(w),\phi(v) \big).\end{equation}
We equip $\TRAF$ with the topology of pointwise convergence.  	
	Investigating LDPs for traffic distribution, a difficulty arises with the fact that the  canonical traffic distribution $\tau_{\mbf Y}[H] $ may go to infinity  or become ill-defined when $n$ 	goes to infinity.
This can be avoided if each of the matrices belong to a set 	
\begin{equation}\label{eq:defBntheta}
B_{n}(\theta)=\left\{ Y\in M_{n}(\mathbb C): \max_{1\le i\le n}\Big( \sum_{j=1}^{n}\IND_{Y_{ij}\neq 0}\Big) \le \theta, \max_{1\le j\le n}\Big(\sum_{i=1}^{n}\IND_{Y_{ij}\neq 0}\Big)\le \theta, \max_{1\le i,j\le n}|Y_{ij}|\le \theta\right\}\,,
\end{equation}
	for some deterministic real $\theta > 0$. 
		This is the case when $\mbf Y$ are independent matrices with law $Z$ as in Theorem \ref{thB} with $\pi$ and $\gamma$ with bounded support. We then can show that 
	
\begin{theorem}\label{LDPtraffic0} 
Let  $\mbf Y$ be $|J|$ independent matrices distributed as in Theorem \ref{thB}, with parameters, for $j \in J$,  $\pi_j \in \mathcal P(\dZ_+)$ finitely supported  and $\gamma_j \in \cP(\dC)$ invariant by complex conjugation and compactly supported. Set $\pi = (\pi_j)_{j \in J}$ and  $\gamma = (\gamma_j)_{j \in J}$.
	 Then, the traffic distribution of $\mbf Y$ satisfies a LDP with speed $n$ and good rate function  $\chi_{\pi,\gamma}$ with a unique minimizer $\tau_{\pi,\gamma}$.
\end{theorem}

As in Voiculescu's definition of micro-states entropy \cite{Vo93,Vo02}, the rate function $\tau_{\pi,\gamma}$ can be thought as a micro-states entropy on traffic distributions. However, in situations where the family of random matrices $\mbf Y$ do not belong to the sets $B_n(\theta)$ for all $n$ for some $\theta >0$ such an approach has to be modified. To do so, we have identified two possible strategies that we explain below. But, let us first describe the setting. We now consider independent sparse Wigner or heavy-tailed random matrices. More precisely, we consider a sequence $\mbf Y = (Y_j)_{j \in J}$ of iid random matrices in $M_n(\dC)$ for a finite set $J$ of labels. The entries $(Y_{j}(k,l))_{k \geq l}$ are assumed to be independent, diagonal entries are assumed to be $0$ for simplicity and outside the diagonal, the entries have distribution $\gamma_n \in \cP(\dC)$ invariant by complex conjugation. We assume that \eqref{eq:gamman2N} holds for some measure $\Lambda$. We restrict ourselves to the simplest case :
\begin{equation}\label{eq:Lambdabddt} 
0 < \Lambda(\dC\backslash \{0\}) < \infty
\end{equation}
This is a strengthening of \eqref{eq:Lambdabdd}. This case contains the sparse Wigner matrices defined in \eqref{eq:SW} with $\Lambda = d \gamma$, but not the $\alpha$-stable laws.

We now describe the first strategy to define the micro-states entropy, which is similar to Voiculescu's micro-state entropy. With the above assumptions, the probability that $  Y_j \in B_n (\theta)$ is of order $\exp( - n L(\theta) (1+o(1))) $ for some  $L( \theta) >0$ for every $\theta$ large enough with $\lim_{\theta \to \infty } L(\theta) = 0$ (see Subsection \ref{subsec:tildechi}). Nevertheless we can deduce from our work that for all $\theta >0$ large enough, any $\tau \in \TRAF$,
\begin{align*}
& \tilde \chi_{\Lambda,\theta}(\tau)  = \lim_{\epsilon\downarrow 0}\liminf_{n\rightarrow\infty}\frac{1}{n}\ln \mathbb P\left( \mbf Y\in  B_{n}(\theta) ^{ J } \; ;  d( \tau_{\mbf Y}, \tau) \le\epsilon\right) \\
 & \quad  = \; \lim_{\epsilon\downarrow 0}\limsup_{n\rightarrow\infty}\frac{1}{n}\ln \mathbb P\left( \mbf Y\in  B_{n}(\theta) ^{ J } \; ;  d( \tau_{\mbf Y}, \tau) \le\epsilon\right),
\end{align*}
where $d$ is any distance generating the topology of pointwise convergence (see Subsection \ref{subsec:tildechi}). The map $\tau \to \tilde \chi_{\Lambda,\theta}(\tau) $ is a good rate function. By monotonicity, we may then define the good rate function: 
\begin{equation}
\label{eq:deftildechi}
 \tilde \chi_{\Lambda}(\tau) = \lim_{\theta \to \infty} \tilde \chi_{\Lambda,\theta} ( \tau).
\end{equation}
The need to restrict ourselves to matrices in $B_{n}(\theta)$ is similar to Voiculescu's restriction to matrices with bounded norm in his definition of the micro-sates entropy: otherwise the traffic distribution  may be ill-defined. For example, in  \cite{fannymom}, the rate of the large deviations for traces of moments of GUE matrices is shown to depend on the power of these moments.

The second strategy to define a micro-state entropy and  circumvent this issue is to change the topology.  In the definition of Voiculescu's micro-states entropy, in \cite{BCG03, CDG2} polynomials test functions were replaced by products of resolvents. In \cite{BB03}, it was shown that these different notions give the same definition of Voiculescu's micro-states entropy.  

In our case where $\mbf Y$ are $|J|$ independent heavy-tailed random matrices satisfying \eqref{eq:Lambdabddt}, we may use that  with overwhelming probability only a few number of columns have a large number of non-zero (or very small) entries or large entries. To study the traffic distribution of these matrices, we found to be more convenient to study  the \emph{rooted} traffic distribution  $\tau_{\mbf Y,i}$ where the value of the function $\phi$ in Definition \ref{traf:def} is specified for a vertex of $H$: the root of $H$ has to be sent to a given $i\in\{1,\ldots, n\}$.  We then take $i$ at random and consider the randomly rooted-traffic distribution given by 
 $$  \Utau_{\mbf Y}=\frac{1}{n}\sum_{i=1}^{n}\delta_{  \tau_{\mbf Y,i} }.$$

Our main result is a LDP for the random  rooted traffic distribution of independent sparse Wigner or heavy-tailed random matrices, see subsection \ref{subsec:rootedtraff} for the precise definition of the topologies.

\begin{theorem}\label{LDPtraffic1} Assume that $\gamma_n \in \cP(\dC)$ is invariant by complex conjugation, and  that  \eqref{eq:gamman2N} and \eqref{eq:Lambdabddt}  hold for some $\Lambda$. Then $\Utau_{\mbf Y}$, the randomly rooted traffic distribution of $\mbf Y$, satisfies a LDP with speed $n$ and good rate function  $\chi_{\Lambda}$ with a unique minimizer $\tau_{\Lambda}$.
\end{theorem}

The rate function $\chi_{\Lambda}$ defines a notion of entropy for the traffic distributions: by definition as a rate function it is a micro-states entropy since  it measures the volume of small balls around a given traffic distribution. On the other hand, we identify in \eqref{defent}, Section  \ref{sec:traffic},  that the microstates entropy $\chi_{\Lambda} $ is given by a rate function $\chi_{\Lambda}^{*}$, thanks a contraction principle as the pushforward of an entropy for the law of a weighted Erd\"os-R\'enyi  graphs, see Theorem \ref{th:SIGMAER}. 
The minimizer $\tau_{\Lambda}$ is the product of the distributions $\Phi_{\Lambda}$ as described in \cite[Definition 2.11]{Camille}, with $\Phi_{\Lambda}$ the limiting traffic distribution of a single matrix (see \cite{Camille2}).  An advantage of $\chi_{\Lambda}$ over $\tilde \chi_{\Lambda}$ defined in \eqref{eq:deftildechi} is that it does not require a truncation step.

Interestingly, the existence of the LDP implies that this entropy is additive for the traffic distributions in situation of \emph{traffic independence}. The notion of independence for traffics was introduced in \cite{Camille}, and encodes in particular the other notions of non-commutative independence such as classical, boolean or free independence. As for the latter notions, it is formulated in terms of a product of their traffic distribution, called the free product of the traffic distributions where \emph{free} has be be understood as \emph{canonical}. 

To fit the framework of LDP, surprisingly, we must consider an adequate notion of \emph{random traffic distributions} drawing inspiration from the concept of amalgamation, which allows  to well define the free product $\tau_1*\tau_2 $ of two random traffic distributions, see Section \ref{trafind}.

\begin{corollary}\label{cor:trafficind} For any random traffic distributions $\tau_{1}$, $\tau_2$, with  $\chi_{\Lambda}$  as in Theorem \ref{LDPtraffic1} 
$$\chi_{\Lambda}(\tau_{1}*\tau_{2})=\chi_{\Lambda}(\tau_{1})+\chi_{\Lambda}(\tau_{2}).$$
\end{corollary}

This result is interesting notably because the analog statement is open in general for Voiculescu's microstates entropy. Voiculescu's free product $\Phi = \Phi_1*\Phi_2$ of two non-commutative tracial distribution corresponds to the free product $\tau_{\Phi} = \tau_{\Phi_1} * \tau_{\Phi_2}$ via the universal construction \cite{GAC}. However, the case considered in this article is somehow orthogonal to Voiculescu's original motivation, when $\Phi_1$ and $\Phi_2$ are the distribution of free semicircular variables the above additivity results is trivial as both sides are infinite. The techniques of this article may open perspectives for further investigations on Voiculescu's microstates entropy question.

\subsection{Organization of the paper}

In Section \ref{sec:graphs}, we define the metric space of unlabeled random rooted marked graphs and its local weak topology. Some of our matrices will be embedded into this space. On this metric space, we establish large deviation  principles.  In Section \ref{sec:networks}, we consider the enlarged metric space of unlabeled random rooted networks which is necessary to deal with general heavy-tailed matrices. In Section  \ref{sec:LDPHT}, we prove that the ESD extends to a continuous function on these metric spaces. As a byproduct, we will obtain Theorem \ref{thA}, Theorem \ref{thB} and Theorem \ref{LDPHT} by a contraction principle. In Section \ref{sec:traffic}, we prove Theorem \ref{LDPtraffic1} and Corollary \ref{cor:trafficind} again by a contraction principle. We also prove that  traffic distributions and neighborhood distributions of unlabeled rooted marked graphs are homeomorphic  in  Lemmas \ref{continuous} and \ref{Lem:Homomorphismfin} . Finally, in Section \ref{sec:discretization}, we establish a general convergence result for the relative entropy of discretization variables in $\dR^d$ which is of general interest. 

\bigskip

\noindent
{\bf Acknowledgments:} A.G. was partly supported  by the ERC Project LDRAM : ERC-2019-ADG Project 884584.

\section{Microstates entropy for locally finite marked graphs}
\label{sec:graphs}

\subsection{Marked graphs}\label{Sec:MarkG}

\paragraph{Basics} We start with a few definitions. Let $G = (V,E)$ be a graph on a countable set $V$ and edge set $E$. In the definitions of this section, the graphs are unoriented, they can have multiple edges and self-loops attached to a vertex. A graph is {\em locally finite} if for all $v \in V$, $\deg(v)  < \infty$, where $\deg(v)$ is the degree of $v \in V$ in $G$ (i.e. number of adjacent edges, a self-loop counts as two adjacent edges). Let $\cZ$ be a set. A {\em marked graph} $G = (V,E,\xi)$ on $\cZ$ is a  graph $(V,E)$
where each edge carries two marks in $\cZ$, one towards
each of its endpoints. For an edge $e = \{u,v\} \in E$ between vertices $u,v$ in $V$, we denote its mark towards the vertex $v$ by $\xi (u,v)$, and its mark towards the vertex $u$ by $\xi (v, u)$. The marked graphs considered in this work are {\em symmetric} marked graphs, namely we assume that $\cZ$ is equipped with an involution denoted by $*$ ($(z^*)^* = z$ for all $z \in \cZ)$ such that for all $\{u,v\} \in E$, 
\begin{equation}\label{eq:symxi}
\xi(u,v) = \xi(v,u)^*.
\end{equation}
By marked graph, we will always mean symmetric marked graph (note that there is no loss of generality here: up to replacing $\cZ$ by $ \cZ \times \cZ$ equipped with the involution $(z,w)^* = (w,z)$ we may always identify a marked graph with a symmetric marked graph). For functions or parameters of $G$, we might use sometimes a subscript $G$ to insist on its dependence on $G$, $V_G,E_G,\deg_G,\xi_G$ and so on.

For a marked graph $G=(V,E,\xi)$ and $v\in V$, we denote by $G(v)$ the connected component $v$ (the subgraph spanned by vertices connected to $v$).

A rooted marked graph $g = (G,o)$ is the pair formed by a connected marked graph and a distinguished vertex $o \in V(G)$ called the root. Let $(G,o)$ and $(G',o')$  be rooted marked graphs. We say that $(G, o)$ and $(G',o')$ are {\em isomorphic}, and write $(G, o) \equiv (G',o')$, if there exists a bijection between the sets of vertices of $G$ and $G'$ which maps $o$ to $o'$ while preserving  the adjacency structure of these connected components, and the edge marks.
Let $\cGr = \cGr(\cZ)$ denote the space of locally finite rooted marked graphs up to isomorphism. In combinatorial language, $\cGr(\cZ)$ is the set of unlabeled locally finite rooted marked graphs. From \cite{MR2354165,MR3933204}, if $\cZ$ is a Baire space, for each $g \in \cGr$, there is a canonical representative marked graph on a vertex set in $\dN$ with root $o = 0$. It is thus fine to write for $g \in \cGr$,  $g = (G,o) $ where $G = (V,E,\xi)$ is a marked graph as long as the properties we consider are class properties (that is do not depend on the choice of the representative of $g$). In the next subsection, we shall introduce a random labeling which will turn to be convenient.

If $G = (V,E, \xi)$ is a marked graph, its unmarked graph is denoted by $\bar G = (V,E)$. Similarly for rooted marked graphs, if $g = (G,o)$, we set  $\bar g = (\bar G,o)$.

\paragraph{Degree sequence}

Let $G = (V,E,\xi)$ be a locally finite marked graph. The {\em marked degree} of $v \in V$ is the counting measure  $\DEG (v)$ on  $\cZ$ defined by
$$
\DEG(v) = \sum_{u \sim v} \delta_{ \xi (v,u)},
$$
where the sum is over all $u \in V$ such that $\{u,v\} \in E$, that is, the neighbors of $v$. In other words, for any $z \in \cZ$, $\DEG(v ; z):=\int {\bf 1}_{\{z\}}\DEG(v)$ is the number of neighbors of $u$ such that $\xi(v,u) = z$. In particular, the total mass of the marked degree is the usual degree:
$$
\DEG(v ; \cZ) = \deg(v).
$$
The {\em marked degree sequence} of $G$ is the sequence $(\DEG(v))_{v \in V}$. If $G =(V,E,\xi)$ is finite then the following {\em edge counting measure} is well-defined:
$$
\MM = \sum_{v \in V} \DEG(v) =  \sum_{e = \{u,v\} \in E} \left( \delta_{\xi(u,v)} + \delta_{\xi(v,u)} \right) ,
$$ 
where we recall that $\{u,v\}$ and $E$  are non oriented edges. 
Namely, for any $z \in \cZ$, $\MM(z)=\int 1_{\{z\}}d\MM$ is the number of oriented edges $(u,v)$ with $\{u,v\} \in E$ such that $\xi(u,v) =z $.  From \eqref{eq:symxi} and the hand-shaking lemma, if $G$ is finite, the counting measure $\MM$ satisfies the following balance equations : for any subset $A \subset \cZ$:
\begin{align} \label{eq:handshake}
 \MM (A) = \MM( A^*) \AND  \MM (A) \in 2 \dN \;\hbox{ if $A = A^*$}.
\end{align}
For $V$ finite, the Erd\H{o}s-Gallai Theorem and its extensions give the necessary and sufficient conditions for a sequence of finite counting measures on $\cZ$, $D = (D(v))_{v \in V}$ to be the marked degree sequence of a simple marked graph $G = (V,E,\xi)$, see \cite{erdHos1960graphs,TRIPATHI2003417}. 
%For a given sequence $D = D(v))_{v \in V}$ with $V = \{1,\ldots,n\}$, we set 
%$$
%\cG_n(D) = \{ G = (V,E,\xi) : \DEG = D \}
%$$
%is the set of marked graphs on $V$ whose marked degree sequence is equal to $D$.

\subsection{Benjamini-Schramm convergence}

\paragraph{Local weak topology} The local topology on $\cGr$ is the product topology inherited from projections around the root. More precisely, for an  integer  number $r \geq 0$ and $g \in \cGr$, let $g_r$ be the intersection of the graph $G$ with the ball (for the graph metric) of center $o$ and radius $r$. Assume from now on that $\cZ$ is a complete separable metric space equipped with the distance $d$. Let $g = (G,o)$, $g' = (G',o') \in \cGr$, we say that the pair $(r,\delta)$ is good for $(g,g')$ if there exists an isomorphism $\psi$ between  $\bar g_{r}$ and $\bar g'_{r}$ such that 
$$\sup_{u,v\in V(G_{r})}d ( \xi(u,v), \xi'_{(\psi(u),\psi(v))}) \leq \delta.$$
We then endow $\cGr$ with the local topology, which is compatible with the distance
$$\DLOC(g,g')=\inf\left\{ \frac{1}{1+r} +\delta:(r,\delta)  \mbox{ is  good for  }(g,g')\right\}.$$

The set $\cGr$ equipped with $\DLOC$ is a complete separable metric space. We denote by $\cP( \cGr)$ the set of probability measures on $\cGr$ equipped with the topology of weak convergence. %It is compatible with the Lévy-Prohorov distance associated to $\DLOC$ which we denote by $\DWLOC$. 

\paragraph{Neighborhood distribution of finite graphs}
Let $G = (V,E,\xi)$ be a marked graph with $V$ finite. The {\em neighborhood distribution} $U(G) \in \cP( \cGr)$ of $G$ is the law of the equivalence class of the rooted graph $(g,o)$ where the root $o$ is sampled uniformly, namely
\begin{equation}\label{eq:defU}
U(G)=\frac{1}{|V|} \sum_{v\in V}\delta_{[G,v]},
\end{equation}
where $[G,v] \in \cGr$ stands for the equivalence class of $(G(v),v)$.

We shall say that a sequence of finite marked graphs $(G_n)$ converges in {\em Benjamini-Schramm} sense toward $\mu \in \cP(\cGr)$ if $U(G_n)$ converges weakly to $\mu$.

\paragraph{Unimodularity}
The neighborhood distribution $U(G)$ satisfies a reversibility assumption called {\em unimodularity}, see  \cite{MR2354165}. An edge-rooted marked graph $(G,e)$ is a connected marked graph $G$ with a distinguished oriented edge $e = (u,v) \in V^2$ (we do not necessarily assume $\{u,v\} \in E$). The above isomorphisms for rooted marked graphs extend to edge-rooted marked graphs and we may speak of unlabeled edge-rooted marked graphs. We denote by $\cGe$, the set of unlabeled edge locally finite marked graphs. The local topology extends to this setting as well. If $G = (V,E,\xi)$ is a connected marked graph and $e = (u,v) \in V^2$, let $[G,e]$ denotes its associated unlabeled edge-rooted marked graph. We then say that a probability measure $\mu \in \cP(\cGr)$ is unimodular if for every non-negative measurable functions $f$ on $\cGe$, we have  
\begin{equation}
\label{eq:defunimod}\int \sum_{v\in V}f([G,o,v])d\mu([G,o])=\int \sum_{v\in V}f([G,v,o])d\mu([G,o]).    
\end{equation}
It is in fact necessary and sufficient to check \eqref{eq:defunimod} for functions $f$ such that $f(G,u,v) = 0$ unless $\{u,v\} \in E_G$, see \cite[Proposition 2.2]{MR2354165}. This equivalent property is called involution invariance.

\paragraph{Edge-rooting}

If $\mu \in \cP(\cGr)$ and $\dE_\mu \deg (o) = \int \deg(o)d \mu([G,o]) \in (0,\infty)$, we define its edge-rooted version $\vec \mu \in \cP( \cGe)$ by the formula: for every measurable sets $A$ on $\cGe$, 
\begin{equation}\label{eq:defvecmu}
\vec \mu (A) = \frac{1}{ \dE_\mu \deg (o) } \dE_\mu \left[ \sum_{v \sim o} \IND ( [G,o,v] \in A )  \right],
\end{equation}
where the sum is over all neighbors of $o$ in $G$. Thanks to \cite[Proposition 2.2]{MR2354165}, the unimodularity  of $\mu$ \eqref{eq:defunimod} is equivalent to the invariance of $\vec \mu$ by inverting the root edge. %: that is, if $(G,o,o')$ has law $\vec \mu$ then $(G,o',o)$ has also law $\vec \mu$. We then say that $\vec \mu$ is {\em invariant}. 
 Observe that if $G = (V,E,\xi)$ is a finite graph, then by definition of $U(G)$,
\begin{equation*}%\label{eq:vecUG}
\vec U(G) = \frac{1}{2|E|} \sum_{e = \{u,v\} \in E} \left( \delta_{[G,u,v]} + \delta_{[G,v,u]} \right).
\end{equation*}
Hence, $\vec U(G)$ corresponds to the law of the edge-rooted graph $(G,\rho)$, where the edge is sampled uniformly at random.

\paragraph{Finite-depth neighborhood and invariance}

For integer $h \geq 0$, we denote by $\cGr_h \subset \cGr$, the subset of unlabeled rooted marked graph such that $g_r = g_h$ for all $r \geq h$: that is unlabeled rooted marked graphs which are contained in the ball of radius $h$ from the root. If $\mu \in \cP(\cGr)$, we denote by $\mu_h \in \cP ( \cGr_h)$ its image by the map $g \to g_h$. Similarly, for integer $h \geq 1$ for $(G,\rho) \in \cGe$, $\rho = (o,o')$, we define $(G,\rho)_h$ as its restriction to the $(h-1)$-neighborhood around $\{o,o'\}$ (in other words, $(G,\rho)_h$ is the restriction to the ball of radius $h$ centered at a fictitious vertex in the middle of $\{o,o'\}$). We define $\cP(\cGe_h)$ similarly.

We will say that $\mu  \in \cP ( \cGr_h)$ is {\em invariant} if \eqref{eq:defunimod} holds for all $f$ such that $f(G,u,v) = 0$ unless {{ $\{u,v\}$ is an  edge }}and  $f(G,u,v)$ is measurable with respect to $(G,u,v)_h$. For $\mu \in \cP(\cGr)$, being unimodular is equivalent to $\mu_h$ invariant for all $h \geq 1$ (in \eqref{eq:defunimod}, functions with bounded support are dense). Similarly, we will say that  $\nu \in \cP(\cGe_h)$ is invariant if $\nu = \vec \mu$ for some invariant $\mu \in \cP(\cGr_h)$. 

%if $(G,o,o')$ and $(G,o,o')$ have the same law where $(G,o,o')$ has law $\nu$. If $\mu \in \cP(\cGr_h)$ then $\vec \mu$ is defined as an element of $\cP(\cGe_h)$. With this convention if $\mu  \in \cP(\cGr_h)$ is invariant then $\vec \mu$ is invariant.

\paragraph{Random labeling} As we have seen, unlabeled rooted graphs give the proper setup for defining the Benjamini-Schramm convergence. Rooted labeled graphs are however more convenient than unlabeled rooted graphs to do computations. We now define a randomized canonical labeling of graphs which appeared notably in \cite{MR3945757,MR4492967}. We define the set of finite integer sequences as 
\begin{equation}\label{eq:defNf}
\dN^f = \cup_{k \geq 0} \dN^k 
\end{equation}
where $\dN^0 = \{ o \}$ and $\dN = \{1, 2,\ldots,\}$ by convention. Let $(G,o)$ be a random rooted marked graph. The law of $[G,o]$ is denoted by $\mu \in \cP(\cGr)$. We perform the breadth-first search tree on the unmarked graph $\bar G$ started at the root $o$, where ties between vertices are broken uniformly at random and independently of $(G,o)$.  This defines a random marked graph $(G',o)$ on a subset of $\dN^f$ whose law depends only the equivalence class of $(G,o)$: a vertex at distance $k$ from the root receives a label in $\dN^k$, if $(i_1,\ldots,i_{k-1})$ is the label of its parent in the search tree, it has the label $(i_1,\ldots,i_{k-1},j)$ if it is the $j$-th offspring of its parent in the random ordering. We call this random rooted marked graph, the {\em uniform labeling} of $\mu \in \cP(\cGr)$. %we shall say that a random labeled marked graph $(G,o)$ on a subset of $\dN^f$ is randomly labeled if its law is equal to the law of the randomly labeled rooted marked graph associated to its unlabeled rooted marked graph. 

The same construction can be performed for edge-rooted graphs $(G,\rho)$, $\rho = (o_1,o_2)$, with two copies $\dN^f \times \dN^f$. Indeed, we perform the breadth-first search tree by starting from two seeds $(o_1,o_2)$: we start with the neighbors of $o_1$ different from $o_2$, then the neighbors of $o_2$ which have not been seen before and so on. This defines the {\em uniform labeling} of $\nu \in \cP(\cGe)$.

In the sequel, when the choice of the root is clear, we will often omit the root and write $G$ in place of $(G,o)$ or $(G,\rho)$.

\subsection{Entropy for marked graphs}

\subsubsection{Graphs with colors and reals marks}\label{susbec:involution}
In this subsection, we define an entropy associated to the Benjamini-Schramm topology for rooted weighted graphs when the weights are taken in $\cZ = \cB \times \dR^k$ where $\cB$ is a finite set (which can be thought as the set of ''colors'', which will later allow us to include traffic distributions of several matrices, the set $\cB$ being then the set of indices of these matrices). This notion of entropy is obtained by using microstates, or rather in a probabilistic terminology, by establishing a large deviation  principle  for a sequence of random weighted graphs. When $\cZ$ is a finite set such results have been established in \cite{BoCa15,DeAn,MR4492967}.

We consider an involution $*$ on $\cB$. To avoid technicalities, we consider an involution on $\dR^k$ which comes from a signed involution on coordinates. More precisely, we fix some $\veps \in \{-1,1\}^k$ and a permutation $\tau \in S_k$ such that for all $i \in \{1,\ldots,k\}$, $\tau^2 (i) = i$ and $\veps_i \veps_{\tau(i)} = 1$. Then for $x =(x_1,\ldots,x_k)$ we set 
\begin{equation}\label{eq:definvolution}
(x_1,\cdots,x_k)^* = (\veps_1 x_{\tau(1)},\ldots,\veps_k x_{\tau(k)}).
\end{equation}
The most relevant examples are $x^* =x$, on $\dC \simeq \dR^2$ the complex conjugate $x^* = \bar x$ and on $\dR^{2}$, $(x_1,x_2)^*  = (x_2,x_1)$ (this choice allows to represent marks on vertices as marks on edges).  These involutions on $\cB$ and $\dR^k$ extend as an involution on $\cZ = \cB \times \dR^k$ defined as $(b,x)^*= (b^*,x^*)$.

Let $G = (V,E,\xi)$ be a (possibly random) marked graph on the mark set $\cZ = \cB \times \dR^k$. We denote by $G^0 = (V,E,\xi^0)$ the marked graph with mark set $\cB$ obtained from $G = (V,E,\xi)$ by setting $\xi^0(u,v) = b$ if $\xi(u,v) = (b,x)$. The same constructions extend to rooted marked graphs. By construction, if $\mu \in \cP(\cGr(\cZ))$, we denote by $\mu^0\in \cP(\cGr(\cB))$ the push-forward of $\mu$ by the map $(G,o) \mapsto (G^0,o)$. Similarly, for $\vec \mu \in \cP(\cGe(\cZ))$, its push-forward is denoted by $\vec \mu^0 \in   \cP(\cGe(\cB))$. By construction, we have
$
\dE_{\mu^0} \DEG(o) = \dE_\mu \DEG (o) (\cdot \times \dR^k).
$

% We will fix a finite partition $\cB$ of $\cZ$ where each $b \in \cB$ is a Borel set and $\cB$ is compatible with the involution: for all $b \in \cB$, there exists $c \in \cB$ such that $b^* =c$ and $c^* = b$. 

%We also consider a  reference finite  measure $\gamma$ on $\cZ$ which is involution-invariant: for all Borel sets $A \subseteq \cZ = \cB \times \dR^k$, $\gamma ( A ) = \gamma (A^*)$. Moreover, we assume that for all $b \in \cB$, $\gamma ( \{ b \} \times \dR^k)> 0  $ and consider the probability measure on $\dR^k$, 
%\begin{equation}\label{eq:defgammab}
%\gamma_b (\cdot) = \frac{\gamma ( \{b\} \times \cdot )}{\gamma ( \{ b \} \times \dR^k)}. 
%\end{equation}
%Note that $\gamma_b = \gamma_{b^*}$.

\subsubsection{Entropy associated to a partially given average marked degree} \label{sec:markedg}

\paragraph{Graphs with a partially given average marked degree}
We start by considering marked graphs such that their edge counting measure restricted to $\cB$ is given. To this end, we consider $d \in [0,\infty)^{\cB}$ and let $(\MM_n)$ be a sequence in $\dZ_+^\cB$  such that for all  $b \in \cB$, 
\begin{equation}\label{eq:limmnB}
\lim_{n \to \infty} \frac{\MM_n(b)}{n} = d(b),
\end{equation}
and which satisfies \eqref{eq:handshake} for every $b \in \cB$. We set 
$$
\bar d = \sum_{b} d(b).
$$
Let $V_n = \{1,\ldots,n\}$. The set of marked graphs on $V_n$ with mark set $\cZ$ whose edge counting measure is $\MM_n$ when restricted to $\cB$ is denoted by
\begin{equation}\label{eq:defGnmB}
 \cG_{n,\MM_n}  = \{ G = (V_G,E_G,\xi_G) : V_G = V_n , \MM_{G} (\{ b \} \times \dR^k)= \MM_n(b) \hbox{ for all $b \in \cB$} \},  
\end{equation}
We denote by $\cG^0_{n,\MM_n}$ the set of simple marked graphs on the mark set $\cB$  whose edge-counting measure is $\MM_n$. It is also the image of $\cG_{n,\MM_n}$ by the map $G \to G^0$ (recall the definition in section \ref{susbec:involution}). The dependency in $\cB$ is implicit. The sets $ \cG^0_{n,\MM_n}  $  and $\cG_{n,\MM_n}$ are non-empty iif $\bar \MM_n \leq n(n-1)/2$.

\paragraph{Randomly marked graph}  We next define a random graph $G_n = (V_n,E_n,\xi_n) $ taking value in $\cG_{n,\MM_n}$ as follows. Let $\gamma_n = (\gamma_{n,b})_{b \in \cB}$ be a collection of probability measures on $\dR^k$ which are $*$-invariant in  the sense that for every $b \in \cB$ and every Borel set $A \subset \dR^k$, 
\begin{equation}\label{eq:gamman}
\gamma_{n,b}(A) = \gamma_{n,b^*}(A^*).
\end{equation}

The graph $G^0_n = (V_n,E_n,\xi^0_n)$ is sampled uniformly on  $\cG^0_{n,\MM_n}$. Then,  independently for each edge $e = \{u,v\} \in E_n$ with $u < v$ and $\xi^0 (u,v) = b$, we sample a weight $\xi_n(u,v)$ in $\{b\} \times \dR^k$ according to the probability distribution $\gamma_{n,b}$. Finally, we set $\xi_n(v,u)= \xi_n(u,v)^*$. Note that since $\gamma_n$ is $*$-invariant, any deterministic choice of the ordering $(u,v)$ or $(v,u)$ leads to the same distribution for $G_n$. We assume that there exists a collection of probability measures $\gamma = (\gamma_b)_{b  \in \cB}$ such that for all $b \in \cB$, and all Borel sets $A \subset \dR^k$,
\begin{equation}\label{eq:gamman2}
\lim_{n \to \infty} \gamma_{n,b}(A) = \gamma_b(A).
\end{equation}

\paragraph{LDP for random graphs} We may now state our main result on the large deviation  principle  for the sequence of random graphs $(G_n)$.

\begin{theorem}[Entropy with a given average marked degree]\label{th:SIGMA1R}
For every integer number  $h \geq 1$, $U(G_n)_h$ and $\vec U(G_n)_h$ satisfy a LDP on $\cP(\cGr_h)$ and $\cP(\cGe_h)$ with rate $n$ and good rate functions $\Sigma_{\gamma,d}(\mu,h)$ and $\vec \Sigma_{\gamma,d}(\nu,h)$ respectively.  Moreover $U(G_n)$ and $\vec U(G_n)$ satisfy a LDP on $\cP(\cGr)$  and $\cP(\cGe)$ with rate $n$ and good rate function $\Sigma_{\gamma,d}(\mu) = \lim_{h} \Sigma_{\gamma,d}(\mu_h,h)$ and $\vec \Sigma_{\gamma,d}(\nu) = \lim_{h} \vec \Sigma_{\gamma,d}(\nu_h,h)$.
\end{theorem}

The precise expression for $\Sigma_d(\mu,h)$ and $\vec \Sigma_d(\nu,h)$ is defined below when $\mu$ and $\nu$ satisfy some admissibility conditions. We conjecture that the entropy is infinity otherwise. We shall call these functions the vertex and edge entropies associated to the pair $(\gamma,d)$. Note that their limits as $h$ goes to infinity are well defined by monotonicity.

\paragraph{Definition of the entropy}
If $G = (V,E,\xi)$ is a (possibly random) marked graph on the mark set $\cZ = \cB \times \dR^k$, we denote by $G_\gamma = (V,E,\xi_\gamma)$ the random marked graph obtained as follows. We have $G^0_\gamma  = G^0 = (V,E,\xi^0)$. Given $G^0$, $ e = \{u,v\} \in E$ and an independent choice of orientation $(u,v)$, for $\xi^0(u,v) = b$, we set $\xi_\gamma (u,v) = (b,y_e)$, where the variables $(y_e)_{e \in E}$ are independent and $y_e$ has distribution $\gamma_{b}$. Finally, we set $\xi_\gamma (v,u) = \xi_\gamma(u,v)^*$.  Note that since $\gamma$ is $*$-invariant, the choice of the orientation is irrelevant.

We denote by $\DKL(p|q) = \int d p(z) \ln (dp(z)/dq(z))$ the relative entropy between two  probability measures $p,q$, which we also write $\DKL(X|Y)$ if $X$ has law $p$ and $Y$ law $q$. 
\begin{definition}[Admissible neighborhood] For $h \geq 1$, a measure $\mu \in \cP(\cGr_h)$ is {\em admissible} if it satisfies the following conditions:
\begin{enumerate}[(C1)]
\item \label{admi} $\mu$ is invariant,
\item \label{admii} $\mu$ is supported on marked trees,
\item \label{admiii} $\dE_{\mu} \deg (o) \ln \deg(o) < \infty$,
\end{enumerate}
We further say that $\mu$ is {\em  $(\gamma,d)$-admissible} if it is admissible and 
\begin{enumerate}[(C4)]
\item $\dE_{\mu^0} \DEG(o) = d$,
\item[(C5)]  $\DKL  ( G | G_\gamma) <\infty$, where $G$ is the uniform labeling of $\mu$.
\end{enumerate}
We say that $\mu$ is $d$-admissible if it is admissible and satisfies (C4). Similarly, we say that $\nu \in \cP(\cGe_h)$ is {\em admissible} if there exists $\mu \in \cP(\cGr_h)$ admissible such that $\vec \mu = \nu$ and $(\gamma,d)$-admissible if we further have that 
\begin{enumerate}[(C4')]
\item $\xi^0(\rho)$ has law $d(\cdot)/\bar d$ and, if $h \geq 2$, $d_\nu = (\dE_{\nu}  \deg(o)^{-1} ) ^{-1} \geq \bar d$, 
\item[(C5')]  $\DKL  ( \vec G | \vec G_\gamma)  <\infty$ where $\vec G$ is the uniform labeling of $\nu$.
\end{enumerate}
We say that $\nu$ is $d$-admissible if it is is admissible and satisfies (C4').
\end{definition}

We note that if the support of $\gamma$ is a finite set, then condition (C5) is satisfied when assumptions (C1)-(C2)-(C3) hold, see  \cite[Lemma 5.10]{BoCa15} and \cite[Lemma 4]{DeAn}.

We are ready for the definition of the entropy. In general, the {\em vertex entropy} $\Sigma_{\gamma,d}(\mu,h) $ and the {\em edge-entropy} $\vec \Sigma_{\gamma,d}(\nu,h)$ will be defined as limits in the proof of Theorem \ref{th:SIGMA1R}. For $\mu \in \cP(\cGr_h)$ which is $(\gamma,d)$-admissible, we have an explicit finite expression: 
\begin{equation}\label{eq:defSigmad}
\Sigma_{\gamma,d}(\mu,h)= 
\Sigma^0_d(\mu^0,h) + \Sigma^1_{\gamma,d}(\mu,h), 
\end{equation}
where $\Sigma_d^0$ and $\Sigma_{\gamma,d}^1$ are defined as follows. If $\mu \in \cP(\cGr_h)$ is $(\gamma,d)$-admissible, we set
\begin{equation}\label{diff}
\Sigma_{\gamma,d}^1 (\mu,h)  =  \DKL ( G | G_\gamma ) - \frac {\bar d}{ 2} \DKL  ( \vec G | \vec G_\gamma ),
\end{equation}
where  $G$ and $\vec G$ are the randomly labeled rooted marked graphs with laws $\mu$ and $\vec \mu$ respectively.  We shall check that if  $\mu$ is $(\gamma,d)$-admissible,  $\vec \mu$ is $(\gamma,d)$-admissible and thus $\DKL  ( \vec G | \vec G_\gamma  ) <\infty$. Hence, the above difference \eqref{diff} is well defined  if  $\mu$ is $(\gamma,d)$-admissible.

Recall the definition of the Shannon entropy and conditional Shannon entropy for random variables $X,Y$ on a finite set: 
$$
H(X) = - \sum_{x} \dP (X = x) \ln \dP(X = x)\mbox{ and }H(X|Y) =  - \sum_{x,y}  \dP ( X= x, Y = y)\ln \dP(X=x | Y = y).
$$
If $\mu \in \cP(\cGr_h(\cB))$ is $d$-admissible,    we set
\begin{equation}\label{defS}
\Sigma^0_d(\mu,h) = - H (G |(G)_1) + \frac {\bar d}{ 2} H  ( \vec G  | (\vec G)_1) + \DKL(\DEG_G(o) | N_d),
\end{equation}
where $G$ and $\vec G$ are the uniform labeling of $\mu$ and $\vec \mu$ respectively and  $N_d$ is a multivariate Poisson variable $(N_d(b))_{b\in \cB}$ where $(N_d(b))_{b\in \cB}$ are independent and for every $b\in\cB$, $N_d(b)$ has distribution $\Poi(d(b))$. Otherwise, we set $\Sigma^0_d(\mu,h) = \infty$.
%
%\begin{equation}\label{eq:defsigmad}
%\sigma_d(\deg(o)) =  \dE \ln (\deg(o)!) - \frac{\bar d}{2} \ln (\bar d/e) - \sum_{b \in \cB} \frac{d(b)}{2} \ln (d(b)/e).
%\end{equation}
%Otherwise, we set $\Sigma^0_d(\mu,h) = \infty$. It is possible to rewrite $\sigma_d$ in terms of entropies. A simple calculation shows that if $N $ is a random variable on $\dZ_+$ with $\dE N = \lambda >0$, then
%$$
%\dE \ln N! - \lambda \ln (\lambda / e) = \DKL ( N |\Poi(\lambda) ) + H(N), 
%$$
%where $\Poi(\lambda)$ is a Poisson distribution with mean $\lambda$. After a few computation, it follows that
%$$
%\Sigma^0_d(\mu,h) = -H((G,o) | (G,o)_{1} ) + \frac{\bar d }{2} H( (\vec G,\rho)  | (\vec G,\rho)_1 ) + \DKL(\DEG | \Poi(d)).
%$$
%where $(G,o)$ and $(\vec G,o)$ are unlabeled rooted marked graphs with laws $\mu$ and $\vec \mu$. Here $H(X|Y)$ is the conditional entropy  and $\Poi(d)$ is multivariate Poisson variable $(N_b)_{b\in \cB}$ where $(N_b)$ are independent and $N_b$ has distribution $\Poi(d(b))$. 
For a very interesting alternative expression of $\Sigma_d^0$ with only non-negative terms involving relative entropies, see the recent work \cite{ramanan2023large} and Remark \ref{rk:kavita} below.

If $\mu \in \cP(\cGr_h)$ is not $(\gamma,d)$-admissible then we conjecture that $\Sigma_{\gamma,d}(\mu,h) = \infty$. We will prove this claim under a small extra condition on $\mu$, see Lemma \ref{le:casinfini} below.

Similarly, for $\nu \in \cP(\cGe_h)$ which is $(\gamma,d)$-admissible, we also define the   edge-entropy associated to $(\gamma,d)$ as
\begin{equation}\label{eq:defvecSigmad}
\vec \Sigma_{\gamma,d}(\nu,h) = 
\vec \Sigma^0_d(\nu^0,h) + \vec \Sigma^1_{\gamma,d}(\nu,h),
\end{equation}
where $\vec \Sigma^0$ and $\vec \Sigma^1$ are defined as follows. For integer number $h \geq 1$, if $\nu \in \cP (\cGe_h)$, we consider the element of $\hat \nu \in \cP(\cGr_{h-1})$ defined, for $A$ measurable set of $\cGr_{h-1}$, by
\begin{equation*}\label{eq:defhatnu}
\hat \nu ( A ) = d_\nu \dE_\nu  \left[ \frac{\IND_{[G,o]_{h-1} \in A }} {\deg(o)} \right],  
\end{equation*}
where $o$ is the origin vertex of the root edge $\rho$ and $d_\nu =  (\dE_{\nu}  \deg(o)^{-1} ) ^{-1}$. If $h=1$, $\hat  \nu$ is trivial since $\cGr_0$ has a unique element. Finally,  we set 
\begin{equation}\label{eq:defnupr}
\dot{\nu} = \dot{\nu}(\bar d) = (1-p)\delta_0 + p \delta_{\hat \nu},
\end{equation}
 with $p =\bar d / d_\nu $ and $0$ is the  rooted graph with a single vertex. For $h \geq 2$, this is a probability measure if and only if $d_\nu \geq \bar d $. 
 
For $\mu \in \cP(\cGr)$ with finite positive expected degree, we denote by $\check \mu \in \cP (\cGr)$, its sized-biased version defined for $A$ measurable set of $\cGr$ as
$$
\check \mu (A) = \frac{\dE [ \deg(o) \IND ( G \in A) ]}{\dE [ \deg(o) ]}
$$
 
 We set 
$$
\vec \Sigma^1_{\gamma,d} (\nu,h) = \frac {\bar d} 2 \DKL ( \vec G  | \vec G_\gamma ) - \bar d \DKL ( \check G   |  \check G_\gamma ) + \DKL ( G | G_\gamma ),
$$
where  $G$, $\check G$ and $\vec G$ are the uniform labeling of $\dot{\nu} , \check {\dot \nu} \in \cP (\cGr_{h-1})$ and $\nu$ respectively. Note that for $h=1$,
 $$\vec \Sigma^1_{\gamma,d} (\nu,1) = \frac {\bar d} 2 \DKL ( \vec G | \vec G_\gamma ).$$ If $\nu \in \cP(\cGe_1(\cB))$ is $d$-admissible, we set 
$$
\vec \Sigma_d^0 (\nu,h) = - \frac {\bar d} 2 H ( \vec G |(\vec G)_1 ) + \bar d H( \check G | (\check G)_{1} ) - H ( G|(G)_{1}) + \DKL(\DEG_G(o) | N_d),
$$
where  $G$, $\check G$ and $\vec G$ are the uniform labeling of $\dot{\nu} , \check {\dot \nu} \in \cP (\cGr_{h-1})$ and $\nu$ respectively. Finally $N_d$ is the multivariate Poisson random variable as above. Otherwise, $\nu$ is not $d$-admissible and  we set $\vec \Sigma_d^0 (\nu,h) = \infty$. 

Again, we conjecture that $\vec \Sigma_{\gamma,d}(\nu,h) = \infty$ if $\nu$ is not $(\gamma,d)$-admissible. Here is our partial result in this direction. We say that $\mu \in \cP(\cGr_h)$ is {\em well-behaved} if for any $1 \leq k \leq h-1$ such that $\DKL((G)_k|(G_\gamma)_k) < \infty$ we also have $\DKL((\check G)_k|(\check G_\gamma)_k) < \infty$. For example if the vertex degree $\deg(o)$ under $\mu$ has bounded support then $\mu$ is well-behaved. Similarly, $\nu \in \cP(\cGe_h)$ is well-behaved if for any $1 \leq k \leq h-1$ such that $\DKL((G)_k|(G_\gamma)_k) < \infty$ we also have $\DKL((\check G)_k|(\check G_\gamma)_k) < \infty$ where $G$, $\check G$ and $\vec G$ are the uniform labeling of $\dot{\nu} , \check {\dot \nu} \in \cP (\cGr_{h-1})$ and $\nu$ respectively.

\begin{lemma}[Infinite entropy]\label{le:casinfini}
For every  integer number $h \geq 1$, if $\mu \in \cP(\cGr_h)$ is either not $d$-admissible or well-behaved and not $(\gamma,d)$-admissible then $\Sigma_{\gamma,d}(\mu,h) = \infty$. Similarly, if $\nu \in \cP(\cGe_h)$ is either not $d$-admissible or well-behaved and not $(\gamma,d)$-admissible then $\vec \Sigma_{\gamma,d}(\nu,h) = \infty$.
\end{lemma}

Finally, we remark   that the entropy is lower-semi-continuous but not continuous on $\cP(\cGr)$, even when restricted to admissible measures, see \cite[Proposition 5.14]{BoCa15}.

%We will see in the proof of Theorem \ref{th:SIGMA2R} that if $\mu$ is $h$-admissible for $\gamma$ then $\DKL ((G,o)_{h} | (G_\gamma,o)_{h} )$, $\DKL  ( (\vec G,\rho) _{h} | (\vec G_\gamma,\rho) _{h} )$  and $\Sigma_d(\mu^0,h)$ are all finite. The two conditions (1) $\mu$ $h$-admissible for $\gamma$ and (2) $\dE_{\mu^0} \DEG(o) = d$ are thus equivalent to $\Sigma_d(\mu,h)$ finite.

\subsubsection{Entropy associated to a partially given marked degree distribution}

We now extend the results of the previous paragraph to marked graphs on $\cZ  = \cB \times \dR^k$ with a given marked degree sequence on $\cB$. Let $\pi$ be a probability measure on $\dZ_+^\cB$ and $D$ with distribution $\pi$. As above, we set $\dE D = d$ and $\bar d = \sum_{b} d(b)$. We will assume $d(b) = d(b^*)$ for all $b \in \cB$ and support of $\pi$ finite. We next consider a sequence $(D_n)_n$ with $D_n = (D_n(v))_{v \in V_n}$ and $D_n(v) \in \dZ_+^\cB$ such that for all $k \in \dZ_+^\cB$, 
\begin{equation}\label{eq:convDnB}
\lim_{n\to \infty} \frac{1}{n} \sum_{v =1}^n \IND( D_n(v) = k ) = \pi(k).
\end{equation}
Setting $\MM_n = \sum_v D_n(v)$ we assume that $\MM_n$ satisfies \eqref{eq:handshake} and that $D_n$ is uniformly bounded, that is for $\theta >0$, $\sum_b D_n(v;b) \leq \theta$ for all $v \in V_n$ and integer number $n \geq 1$. Under these assumptions, we have for all $b \in \cB$,
$$
\lim_{n \to \infty} \frac{\MM_n(b)}{n} = d(b).
$$

The set of $\cZ$-marked graphs on $V_n$ whose $\cB$-marked degree sequence is $D_n$ is
\begin{equation*}
 \cG_{n,D_n}  = \{ G = (V_G,E_G,\xi_G) : V_G = V_n , \DEG_{G}(v)(\cdot \times \dR^k) = D_n(v) \hbox{ for all } v \in V_n \}.  
\end{equation*}
This is a subset of $\cG_{n,\MM_n}$  defined in \eqref{eq:defGnmB} which is not empty for all $n$ large enough. 

We now define a random graph $H_n = (V_n,E_n,\xi_n)$ on $\cG_{n,D_n}$ as follows. The graph $H^0_n = (V_n,E_n,\xi_n)$ is sampled uniformly on $\cG^0_{n,D_n}$. Then, independently and independently for each edge $e = \{u,v\} \in E_n$ with $u < v$ and $\xi^0 (u,v) = b$, we sample a weight $\xi_n(u,v)$ in $\{b\} \times \dR^k$ according to the probability distribution $\gamma_{n,b}$ defined as in \eqref{eq:gamman}. Finally, we set $\xi_n(v,u)= \xi_n(u,v)^*$. Note that since $\gamma_n$ is $*$-invariant, any deterministic choice of the ordering $(u,v)$ or $(v,u)$ leads to the same distribution for $H_n$.

We say that $\mu \in \cP(\cGr_h)$ is $(\gamma,\pi)$ admissible if it is $(\gamma,d)$ admissible and $\DEG^0(o)$ has law $\pi$ under $\mu$. The {\em vertex-entropy} associated to $(\gamma,\pi)$ is
$$
\Sigma_{\gamma,\pi}(\mu,h) = \left\{\begin{array}{ll}
\Sigma^0_{\pi}(\mu^0,h) + \Sigma^1_{\gamma,d}(\mu,h)  & \hbox{if $\mu$ is $(\gamma,\pi)$-admissible}\\
 \infty & \hbox{\hbox{otherwise,}}
\end{array}\right. 
$$
where $\Sigma_{\gamma,d}^1$ is defined as above. For $\mu \in \cP(\cGr_h(\cB))$ satisfying $\mu^0_1 = \pi$ and conditions (C1)-(C2)-(C3),    we set
$$
\Sigma^0_{\pi}(\mu,h) = - H (G |(G)_1) + \frac {\bar d}{ 2} H  ( \vec G  | (\vec G)_1),
$$
where $G$ and $\vec G$ are the uniform labeling of $\mu$ and $\vec \mu$ respectively. Similarly,  the next theorem is the analog of Theorem \ref{th:SIGMA1R} for $U(H_n)$.

\begin{theorem}[Entropy with a partially given marked degree sequence]\label{th:SIGMADR}
Let $\pi$ be a probability measure on $\dZ_+^\cB$ with finite support and average $d \in (0,\infty)^{\cB}$. For any integer number $h \geq 1$, $U(H_n)_h$ satisfies a LDP on $\cP(\cGr_h)$  with rate $n$ and good rate function $\Sigma_{\gamma,\pi}(\mu,h)$.  Moreover $U(H_n)$ satisfies a LDP on $\cP(\cGr)$   with rate $n$ and good rate function $\Sigma_{\gamma,\pi}(\mu) = \lim_{h} \Sigma_{\gamma, \pi} (\mu_h,h)$.
\end{theorem}

There is also  a LDP for $\vec U(H_n)_h$ but we have omitted the explicit expression for the sake of concision (see the proof of Theorem \ref{th:SIGMADR}). When $\pi = \delta_d$ is a Dirac mass at $d \in \dZ_+^{\cB}$ with $\bar d = \sum_b d(b)$ and $d(b^*) = d(b)$. Then the entropy has a particularly nice expression. Indeed, there is a unique admissible $\mu^0 \in \cP(\cGr(\cB))$ such that $(\mu^0)_1 = \delta_d$: this the (Dirac mass at the) deterministic infinite $\bar d$-regular tree where each vertex $u$ has exactly $d(b)$ neighbors $v$ such that $\xi(u,v) = b$. In this case, we thus have $\Sigma^0(\mu^0,h) = 0$ for all $\mu$ which are $(\gamma,\pi)$-admissible.

\subsection{Entropy associated to marked Erd\H{o}s-R\'enyi random graphs}

\label{subsec:LDPER}

We finally consider a model of random graphs with independent edges. We take $\cZ  = \dR^k$ without loss of generality for this model. Let $d_n$ be a sequence of positive real numbers such that 
$$
\lim_{n\to \infty} d_n = d > 0 .
$$ 
We also consider  a sequence of $*$-invariant probability measures  $(\gamma_n)_{n\in\mathbb N}$ on $\dR^k$ such that for every Borel sets $A \subset \dR^k$, 
$$
\lim_{n \to \infty} \gamma_n( A) = \gamma(A), 
$$
where $\gamma$ is a probability measure on $\dR^k$.  We consider the random marked graph $G'_n = (V_n , E_n , \xi_n)$ on the vertex set $V_n = \{1,\ldots,n\}$ where each edge $\{u,v\}$ of the complete graph is present independently with probability $\min(d_n/n,1)$ and $\xi_n(u,v) = \xi_n(v,u)^*$ receives an independent mark with distribution $\gamma_n$. 

For $\delta \geq 0$, we set $j_{d}$ to be the entropy for the mean degree:
$$
j_d (\delta ) = \frac{1}{2} \left(  d  \ln \left( \frac d \delta \right) +  d  - \delta  \right). 
$$

We have the following corollary of Theorem \ref{th:SIGMA1R}. In the statement below $\Sigma_{\gamma,d}$ denotes the rate function which appears in Theorem \ref{th:SIGMA1R} (when $\cB$ is reduced to a singleton).

\begin{theorem}\label{th:SIGMAER}
For any integer number $h \geq 1$, $U(G'_n)_h$ satisfies a LDP on $\cP(\cGr_h)$  with rate $n$ and good rate function 
$\Sigma^{\tiny {\rm ER}}_{\gamma,d} (\mu,h) = \Sigma_{\gamma,\delta} (\mu,h) + j_d (\delta)$ where $\delta = \dE_{\mu} \deg(o)$ (if $\delta = \infty$ then $\Sigma^{\tiny {\rm ER}}_{\gamma,d} (\mu,h) = \infty$).  Moreover $U(G'_n)$ satisfies a LDP on $\cP(\cGr)$   with rate $n$ and good rate function $\Sigma^{\tiny {\rm ER}}_{\gamma,d}(\mu) = \lim_{h} \Sigma_{\gamma,d}^{\rm ER}(\mu_h,h)$. $\Sigma^{\tiny {\rm ER}}_{\gamma,d} (\mu) $ has a unique minimizer.
\end{theorem}
The minimizer of $\Sigma^{\tiny \rm{ ER}}_{\gamma,d} (\mu) $ is described in Lemma \ref{le:minimizer}.

\subsection{Global minimizer of the entropy}

For discrete marks,  \cite{BoCa15,DeAn,MR4492967,ramanan2023large} gave  a precise description of the unique minimizer $\mu(p)$  of the entropy   given a  $d$-admissible neighborhood  $p \in \cP( \cGr_h(\cB))$:
$$
\inf \{ \Sigma^0_{d} (\mu) : \mu \in \cP( \cGr(\cB) ), \mu_h = p\} = \Sigma^0_d(\mu(p)).
$$
That is, there is a unique and explicit measure  $\cP(\cGr(\cB))$ which minimizes the entropy given the law of its $h$-neighborhood. We will not extend these developments for our more general marked space $\dZ = \cB \times \dR^k$ here. We will simply prove that there is a unique global minimizer that we describe now. 

If $\pi$ is a probability measure on $\dZ_+^\cB$ with non-zero expectation $d = \dE D = \sum_k k \pi(k) \in \dZ_+^\cB$ such that $d(b) = d(b^*)$ for all $b \in \cB$. For $b \in \cB$ such that $d(b) >0$, we define the size-biased law $\hat \pi_b \in \dZ_+^\cB$ by the formula
$$
\hat \pi_b ( k ) =  \frac{(k_{b^*} + 1)\pi(k + 1_{b^*})}{d(b)},
$$
where $1_b \in \dZ_+^{\cB}$ is the vector $1_b(c) = \IND ( b =c)$. As in \cite[Section 4.4]{BoCa15}, the {\em unimodular Galton-Watson tree} with degree distribution $\pi$ is the random $\cB$-marked rooted tree $G^0$ where the root has marked degree $\DEG(o)$ with distribution $\pi$. Then, recursively and independently, any vertex $v \ne o$ such that  $\xi(u,v) = b$, where $u$ is the parent of $v$,  produces offspring according to the distribution $\hat \pi_b$. Finally, we denote by $\UGW(\gamma,\pi)$ the law of the $\cZ$-marked rooted tree $G = G^0_\gamma$ where $G^0$ is as above.

\begin{lemma}[Minimizer of the entropy]\label{le:minimizer}
For any integer number  $h \geq 1$ and $\mu \in \cP(\cGr_h)$, we have  $\Sigma_{\gamma,\pi} (\mu,h) = 0$ if and only if  $\mu = \UGW(\gamma,\pi)_h$. Similarly $\Sigma^{\tiny {\rm ER}}_{\gamma,d} (\mu,h) = 0$ if and only if $\mu = \UGW(\gamma,\Poi(d))_h$, where $\Poi(d)$ is the multivariate Poisson distribution with mean $d$.
\end{lemma}

%There is also  a LDP for $\vec U(F_n)_h$ but we have omitted the explicit expression for the sake of concision (a slight care needs to be taken with the event where there is no edge and $\vec U(F_n)$ is ill-defined).

\section{Proofs of Theorem \ref{th:SIGMA1R}, Theorem \ref{th:SIGMADR} and Theorem \ref{th:SIGMAER}}

\subsection{Discrete marks}

We start by establishing the LDP for $U(G^0_n)$ and $\vec U(G_n^0)$, that is the case of marked graphs where the mark set $\cB$ is finite.

\begin{proposition}\label{prop:SIGMA1}
For any integer number $h \geq 1$, $U(G^0_n)_h$ and $\vec U(G^0_n)_h$ satisfy a LDP on $\cP(\cGr_h(\cB))$ and $\cP(\cGe_h(\cB))$ with rate $n$ and good rate functions $\Sigma^0_d(\mu,h)$ and $\vec \Sigma^0_d(\nu,h)$. %Moreover, $U(G^0_n)$ and $\vec U(G^0_n)$ satisfy a LDP on $\cP(\cGr(\cB))$ and $\cP(\cGe(\cB))$ with rate $n$ and good rate functions $\Sigma^0_d(\mu) = \lim_{h\to \infty}\Sigma^0_d(\mu_h,h) $ and $\vec \Sigma^0_d(\nu) = \lim_{h\to \infty} \vec \Sigma^0_d(\nu_h,h)$.  
\end{proposition}

We start with a simple lemma on the law of the mark of the root edge under $\vec \mu$.
\begin{lemma}\label{le:rootmark}
Let $\mu \in \cP(\cGr)$ such that $\dE_\mu \DEG (o) = d$, $\bar d = \sum_b d(b) \in (0,\infty)$. Then, under $\vec \mu$, $\xi(\rho)$ has law $d(\cdot)/\bar d$. Moreover, for every  integer number $k \geq 1$, 
$$
\dP_{\vec \mu} (\deg(o) = k) = \frac{k  \dP_{\mu} (\deg(o) = k)}{\bar d} .
$$
\end{lemma}
\begin{proof}
Let $B$ be a measurable subset of the mark space $\cZ$. For the first claim, we apply \eqref{eq:defvecmu} to $A = \{ \xi(\rho) \in B \}$ and use $\dE_{\mu} \deg(o) = \bar d$. For the second claim, we apply \eqref{eq:defvecmu} to $A = \{ \deg(o) = k \}$.
\end{proof}

We split the proof of this proposition in two parts: the case of $U(G_n^0)$ and the case of $\vec U(G_n^0)$. The latter is a consequence of known results.

\begin{proof}[Proof of Proposition \ref{prop:SIGMA1} : the case $U(G_n^0)$]
The statement on $U(G_n^0)_h$ is contained in \cite[Theorem 4.11]{MR4492967} which follows itself from \cite[Theorem 3]{DeAn} together with a simplification of the formulas appearing in \cite{BoCa15,DeAn} thanks to the introduction of the random labeling of an unlabeled rooted graph, see \cite[Equation (4.8)]{MR4492967}. More precisely, it follows from these references that $U(G_n^0)_h$ satisfies a LDP with good rate function on $\cP(\cGr_h(\cB))$: 
\begin{equation}\label{eq:defJdmuh}
J_d(\mu,h) =  \left\{\begin{array}{ll}
 - H (G) + \frac {\bar d}{ 2} H  ( \vec G ) + \sigma_d(\mu),
& \hbox{if $\mu$ is admissible and $\dE_{\mu} \DEG(o) = d$}\\
 \infty & \hbox{\hbox{otherwise,}}
\end{array}\right. 
\end{equation}
where  $G$ and $\vec G$ are uniform labeling with law $\mu$ and $\vec \mu$ respectively and 
\begin{equation*}
\sigma_d(\mu) =  \dE_\mu \ln (\deg(o)!) - \frac{\bar d}{2} \ln \bar d - \sum_{b \in \cB} \frac{d(b)}{2} \ln (d(b)) + \bar d . 
\end{equation*}
We finally show that $J_d(\mu,h)=\Sigma_{\gamma,d}(\mu,h)$
Indeed, writing $H(G) = H(G|G_1) + H(G_1)$ and $H(\vec G) = H(\vec G | \vec G_1) + H(\vec G_1)$ we find that 
\begin{equation}\label{jd1}
J_d(\mu,h)=-H(G|G_{1})+ \frac {\bar d}{ 2} H  ( \vec G |\vec G_{1}) -H(G_{1})+ \frac {\bar d}{ 2} H  ( \vec G_{1})+ \sigma_d(\mu)\,.
\end{equation}
By Lemma \ref{le:rootmark}, we have 
$$\bar d H(\vec G_1) = - \sum_{b \in \cB} d(b) \ln d(b) + \bar d \ln \bar d. $$ 
By \cite[Lemma 2.4]{MR4492967}, we have
$$
H(G_1) = H(\DEG_G(o)) - \sum_{b \in \cB} \dE_{\mu} \ln ( \DEG(o ; b) ! ) + \dE_{\mu} \ln (\deg(o)!) .
$$ 
Finally, if $N_d$ is as in  the definition \eqref{defS} of $\Sigma^0_d(\mu,h)$ and $D$ is a random variable on $\dZ_+^{\cB}$ with $\dE D = d$, then we observe that
\begin{align}
\DKL( D| N_d) & = \sum_{\delta \in \dZ_+^\cB} \dP ( D= \delta ) \ln \left( \frac{ \dP ( D = \delta)}{ \prod_{b} (e^{-d(b)} d(b)^{\delta_b} / \delta(b) !)} \right) \nonumber \\
 & = - H(D) + \bar d  - \sum_{b \in \cB}d(b) \ln ( d(b)) + \sum_{b \in \cB} \dE \ln (D(b)!)\label{eq:defsigmad}.
\end{align}
Putting together these equations, we find that
$$-H(G_{1})+ \frac {\bar d}{ 2} H  ( \vec G_{1})+ \sigma_d(\mu)=\DKL( \DEG_G(o)| N_d) $$
and we 
 obtain the claimed formula: $J_d(\mu,h) = \Sigma_{\gamma,d}(\mu,h)$.
\end{proof}

\begin{remark}[Expression of the entropy with unlabeled graphs]
Following the same proof and using the above references \cite{DeAn,MR4492967}, it also possible to write the same expression for $\Sigma^0_d(\mu,h)$ with unlabeled marked graphs. More precisely if $\mu$ is admissible and $\dE_\mu \DEG(o) = d$, we have
$$
\Sigma^0_d(\mu,h) = -H(G| G_{1} ) + \frac{\bar d }{2} H( \vec G  | (\vec G)_1 ) + \DKL(\DEG_G (o)| N_d).
$$
where $G$ and $\vec G$ are unlabeled rooted marked graphs with laws $\mu$ and $\vec \mu$. In this case, we simply have $H((G)_1) = H(\DEG_G(o))$.
\end{remark}

We now turn to the proof of the LDP for $\vec U(G_n)_h$. We start with the existence of the LDP. 

\begin{lemma}\label{le:SIGMAe}
Let $h \geq 1$ be an integer. $\vec U(G^0_n)_h$ satisfies a LDP on $\cP(\cGe_h(\cB))$ with good rate function 
$$\vec J_d (\nu,h) =  \inf \{ \Sigma^0_d(\mu,h) : \mu \in \cP( \cGr(\cB)), \vec \mu = \nu \}.$$
\end{lemma}

\begin{proof}
If $G = (V,E,\xi)$ is a finite weighted graph with $|E| = \bar m$, $|V| = n$, for any bounded continuous function $f: \cGe_h \to \dR$,
$$
\vec U(G)_h (f) = \frac{1}{2\bar m} \sum_{u \in V} \sum_{v : \{u,v\} \in E } f(G,u,v)  =  \frac{n}{2 \bar m} U(G)_h  (\check f),
$$
with 
$$
\check f(G,o) =  \sum_{v \sim o } f(G,o,v).
$$
The  function $\mu \to \mu(\check f)$ is not necessarily continuous for the weak topology (as $\check f$ may be unbounded). However, for any $t > 0$, on $\cD_t = \{ \mu \in \cP (\cGr_h(\cB)) :  \hbox{for all $b$, } \dE_{\mu} \DEG(o;b) \ln (\DEG(o;b)) \leq t \}$, the function $\mu \to \mu(\check f)$ is continuous. 

We note also that $\Sigma^0_d ( \mu,h) \geq  \Sigma^0_d ( \mu,1) = \DKL( \DEG(o) | N_d)$. Since $n \ln n \leq \ln (n!) + Cn$ for all integers $n \geq 1$ and  $H(\DEG(o))$ is uniformly bounded over all $\mu$ such that $\dE_\mu \DEG(o) = d$ (see for example \cite[Lemma 5.5]{BoCa15}),  we deduce that $\inf_{\mu \notin \cD_t} \Sigma^0_d ( \mu,h) \geq \inf_{\mu \notin \cD_t} \DKL( \DEG(o) | N_d)$ diverges as $t \to \infty$ since the last term in the right hand side  of  \eqref{eq:defsigmad} does. From the contraction principle in its extended version, see  \cite[Theorem 4.2.23]{De-Ze}, we deduce that $\vec U(G_n)_h$ satisfies a LDP with good rate function 
$
 \vec J_d (\nu,h) = \inf \{ \Sigma^0_d(\mu,h) : \vec \mu =  \nu \}
$
as requested.
\end{proof}

In view of Lemma \ref{le:SIGMAe}, we need to prove that $\vec J_d = \vec \Sigma_d^0$. We start with a preliminary lemma. Recall the definition of $\dot \nu = \dot \nu (\bar d)$ in \eqref{eq:defnupr}.

\begin{lemma}\label{le:hatvec} For every integer number $h \geq 2$ and $\mu \in \cP(\cGr_h)$ with $\dE \deg (o) = \bar d$, the distribution $\dot{\vec\mu}$ is equal to $\mu_{h-1}$. If $\nu \in \cP(\cGr_h)$ is invariant and is such that $\xi(\rho)$ has law $d/\bar d$ under $\nu$ and $d_\nu\geq \bar d$, then $\dE_{\dot \nu} \DEG(o) = d$. 
\end{lemma}

\begin{proof}
We start with the first claim and set $\nu = \vec \mu$. We start by observing for any integer number $k \geq 1$, 
$$\dP_\nu (\deg(o) = k) = \frac{1}{\bar d} \dE_\mu \sum_{v \sim o} \IND (\deg(o) = k) = \frac{k \dP_\mu ( \deg(o) = k)}{\bar d}. 
$$
Therefore, 
\begin{equation}\label{eq:pdnu}
\frac{1}{d_\nu} = \dE_{\nu} \left( \frac{1}{\deg(o)}\right) = \sum_{k=1}^{\infty} \frac { \dP_{\nu} (\deg(o) = k)}{k} = \frac{\dP_\mu ( \deg(o) >0)}{\bar d}. 
\end{equation}
Recall the definition of $\hat \nu$ above \eqref{eq:defnupr}. For any measurable event $A \subset \cGr_{h-1}$, since $\nu$ is supported on graphs with $\deg(o) >0$, 
$$
\hat \nu (A) = d_\nu \dE_{\nu} \left( \frac{\IND_{[G,o]_{h-1} \in A , \deg(o) > 0}} {\deg(o)} \right) = \frac{d_\nu}{\bar d} \dE_{\mu} \left( \sum_{v \sim o}  \frac{\IND_{[G,o] \in A , \deg(o) >0}} {\deg(o)} \right) =  \frac{d_\nu}{\bar d}\mu_{h-1} (A \cap \{\deg(o) >0\} ).
$$
It follows that $\hat \nu$ is equal to the distribution of $\mu_{h-1}$ conditioned on $ \{\deg(o) >0\}$.  In particular,  $p = \bar d /d_\nu  = \dP_\mu ( \deg(o) > 0)$ and $\mu_{h-1} = (1-p)\delta_0 + p \delta_{\hat \nu} = \dot \nu$ as requested in the first claim.

For the second claim, since $\nu$ is invariant, $\nu = \vec \mu$ for some invariant $\mu \in \cP(\cGr_h)$. Let $\mu'$ be equal to $\mu$ conditioned on $\{ \deg (o)  > 0 \}$. Note that $\mu(q) = (1-q) \delta_0 + q \mu'$ with  $0 \leq q \leq 1$ satisfies $\vec \mu(q) = \nu$ and, from \eqref{eq:pdnu}, $\dE_{\mu(q)} \deg(o) = q  d_{\nu}$.   Since $d_\nu \geq \bar d$,  $\dE_{\mu(q)} \deg(o) = \bar d$ for $q = \bar d / d_\nu$. For such $q$, let $d' = \dE_{\mu(q)} \DEG(o)$. From the first claim, we have $\dE_{\dot \nu} \DEG(o) = d'$. However Lemma \ref{le:rootmark} implies that $d' = d$.
\end{proof}

A final lemma will be useful.
\begin{lemma}\label{le:rootmark2}
Let $h \geq 1$ and $k \geq 1$ be integers and  $\mu \in \cP(\cGr_h)$. Let $G,\vec G$ be uniform labeling with laws $\mu$ and $\vec \mu$ and $\vec G^{(o)}$ be the $(h-1)$-neighborhood of the origin vertex of the root-edge of $\vec G$.  Then $\vec G^{(o)}$ is the uniform labeling of $\check \mu_{h-1}$. In particular, conditioned on $\{ \deg(o) = k \}$,  the distributions of $(G)_{h-1}$ and $\vec G^{(o)}$ are equal. Similarly, conditioned on $\{ \deg(o) = k \}$, if $(1,\ldots,k)$ are the neighbors of the root-vertex $o$, the distributions of $(G,o,1)_h$ and $\vec G$ are equal. 
\end{lemma}

\begin{proof}
Let $B \subset \cGr_{h-1}$ and $k \geq 1$. We apply  \eqref{eq:defvecmu} to $A = \{[G,o] \in B\} \cap \{ \deg(o) = k \}$ and get $\dP (\vec G^{(o)} \in A , \deg(o) = k) = k \dP ( (G)_{h-1} \in A , \deg(o) = k) / \bar d $.
We finally  use Lemma \ref{le:rootmark}. For the proof of the second claim we take $A = \{[G,\rho]_h \in B\} \cap \{ \deg(o) = k \}$.
\end{proof}

We are now ready to prove the second claim of Proposition \ref{prop:SIGMA1}.

\begin{proof}[Proof of Proposition \ref{prop:SIGMA1}  : the case $\vec U(G_n^0)$]
By Lemma \ref{le:SIGMAe}, $\vec U(G^0_n)_h$ satisfies a LDP on $\cP( \cGe_h (\cB))$ with good rate function 
$$
\vec J_d (\nu,h) =  \inf \{ \Sigma^0_d(\mu,h) : \mu \in \cP(\cGr_h(\cB)) , \vec \mu =  \nu \}.
$$
Our goal is therefore to show that $\vec J_d (\nu,h) =\vec \Sigma^0_d(\nu,h)$.

If $h \geq 2$, from the definition of $\hat \nu$, $d_\nu = (\dE_{\nu} (\deg(o)^{-1}) )^{-1} = \dE_{\hat \nu} (\deg(o))$. Let $\mu  \in \cP(\cGr_h(\cB))$. By Lemma \ref{le:hatvec}, if $\vec \mu = \nu$ and $\dE_\mu [ \deg(o)] = \bar d$ then $\hat{\vec \mu}  = \hat  \nu$ is the law $\mu_{h-1}$ conditioned on $\{\deg(o) > 0\}$.  In particular, $d_\nu = \bar d / \dP_\mu (\deg (o) > 0)$. It follows that $d_\nu \geq \bar d$ otherwise $\vec J_d (\nu,h) = \infty$. If  $d_\nu \geq \bar d$, it follows from Lemma \ref{le:hatvec}, that all $\mu \in  \cP(\cGr_h(\cB))$ such that $\dE \deg(o) = \bar d$ and $\vec \mu  = \nu$ have the same marginal distribution on ball of radius $(h-1)$: $\mu_{h-1} = \dot \nu$. Also by Lemma \ref{le:rootmark}, under $\nu$, $\xi(\rho)$ has law $d/\bar d$.

We deduce that if $\nu$ is not admissible or if  $d_\nu < \bar d $ (for $h\geq 2$) or if $\xi^0(\rho)$ has not law $d/\bar d$ then $\vec J_d (\nu,h) = \infty$.
Hereafter, we may and shall restrict ourselves to $\nu$ admissible, so that $d_{\nu}=\bar d$ and $\xi^0(\rho)$ has  law $d/\bar d$.
 When $\nu$ is admissible and, by  \cite[Lemma 5.10]{BoCa15} and \cite[Lemma 4]{DeAn}, $\Sigma^0_d(\mu,h) < \infty$ for all $\mu$ admissible such that $\vec \mu = \nu$. 
We start with the lower bound  of $\vec J_{d}$ by $\vec\Sigma^{0}_{d}$ and prove an even stronger statement:  for every $d$-admissible $\mu$ such that $\vec \mu = \nu$,
\begin{equation}\label{eq:yena}
\Sigma^0_d(\mu,h) \geq \vec \Sigma^0_d( \nu,h).
\end{equation}
Let $G,\check G,\vec G$ be uniform labeling of $\mu,\check \mu$ and $\vec \mu = \nu$. Writing $H(X|f(X)) = H (X) - H(f(X))$, the lower bound \eqref{eq:yena} is equivalent to 
$$
H(G) - \frac {\bar d} 2 H(\vec G) \leq  \frac {\bar d} {2} H(\vec G) - \bar d H((\check G)_{h-1}) + H((G)_{h-1}),
$$
Let $H_k(\cdot)$ denote the conditional entropy given $\{\deg(o) = k\}$ and let $\pi(k) = \dP_\mu ( \deg(o) = k)$. We write $H(G) = H(\deg_G(o)) + H(G|\deg_G(o)) = H(\deg_G(o))  + \sum_{k\geq 1} \pi(k) H_k(G)$ and similarly for $H(\vec G) = H(\deg_{\vec G}(o)) + H(\vec G|\deg_{\vec G}(o))$. We deduce that \eqref{eq:yena} is equivalent to 
$$
\sum_{k \geq 1} \pi(k) \left( H_k(G) - k H_k (\vec G)  + (k-1) H_k ((G)_{h-1}) \right) \leq 0,
$$
where we have used Lemma \ref{le:rootmark} and Lemma \ref{le:rootmark2}. Let $(1,\ldots,k)$ be the neighbors of $o$ in $\vec G$ with $\rho = (o,1)$ and let $\vec G_i$ be the edge-rooted graph $(\vec G,o,i)$. Similarly, we set $G_i = (G,o,i)_h$. The  rooted marked graph $G$ is in bijection with the $k$-tuple $(G_1,\ldots,G_k)$. In particular 
\begin{eqnarray}
H_k(G)= H_k( (G)_{h-1} ) +  H_k(G | (G)_{h-1} ) &= & H_k( (G)_{h-1})   + H_k((G_1,\ldots,G_k ) | (G)_{h-1}) \nonumber \\
&   \leq & H_k( (G)_{h-1}) + \sum_{i=1}^k H(G_i | (G)_{h-1}).\label{eq:Hkin}
\end{eqnarray}
Also, by exchangeability and by Lemma \ref{le:rootmark2}, $\vec G_i$ and $G_i$ have the same distribution. Using  exchangeability and Lemma \ref{le:rootmark2} again, we write
$$
k H_k(\vec G) = \sum_{i=1}^k H(G_i) = k H_k( (G)_{h-1}) +   \sum_{i=1}^k H(G_i | (G)_{h-1}).
$$
It concludes the proof of \eqref{eq:yena}. For the upper bound, we observe that \eqref{eq:Hkin} is an equality if and only if $(G_1,\ldots,G_k)$ are independent conditioned $(G)_{h-1}$ and $\deg(o) = k$. This characterizes uniquely a measure $\mu$ and it is immediate to check that this measure is admissible.  It concludes the proof.
\end{proof}

\subsection{Proof of Theorem \ref{th:SIGMA1R} : existence of the LDP}

In this subsection, we establish that $U(G_n)$ satisfies a LDP. By a projective limit,  we will prove that it is a consequence of Proposition \ref{prop:SIGMA1} for finite mark sets.

\paragraph{Quantization of marks. } We first describe a natural quantization of $\dR^k$ into bins. Fix a mesh size $\delta >0$ and a window size $\kappa > 0$ with $\kappa/\delta$ integer.  The quantization is a map $\{ \cdot \}_\delta^\kappa : \dR^k \to  \left(\delta \dZ \cap [-\kappa,\kappa)\right)^k \cup \{\omega\}$ with $\omega \in \dR^k \backslash (\delta \dZ)^k$ is any default value such that $\omega^* =\omega$. Assume first that $k=1$. For $|x| \geq \kappa$, we set $\{x\}^\kappa_\delta = \omega$. For $0 \leq x < \kappa$, we set $\{x\}^\kappa_\delta = \lfloor  x / \delta \rfloor \delta$. If $-\kappa < x< 0$, we set $\{x\}^\kappa_\delta = - \{-x\}^\kappa_\delta$. On $\dR^k$, we set $\{(x_1,\ldots,x_d)\}^\kappa_\delta = (\{x_1\}^\kappa_\delta,\ldots,\{x_d\}^\kappa_\delta)$. Finally, if $\kappa/\delta$ is not an integer, we set $\{x\}^\kappa_\delta = \{x\}^{\kappa'}_{\delta}$ with $\kappa' = \delta \lceil \kappa/\delta \rceil$.

We fix a non-increasing function $\kappa : (0,1] \to (0,\infty)$ such that
$$
\lim_{\delta \to 0} \kappa(\delta) = \infty.
$$
We set $$L^\delta = \left(\delta \dZ \cap [-\kappa(\delta),\kappa(\delta))\right)^k \sqcup \{\omega\} \AND \cZ^\delta = \cB \times L^\delta.$$
Note that the involution $*$ defined in \eqref{eq:definvolution} acts on $L^\delta$. The marked space $\cZ^\delta$  is thus also equipped with a compatible involution.

If $G = (V,E,\xi)$ is a marked graph on the mark set $\cZ$, we define $ G^\delta= (V,E,\xi^\delta)$ as the marked graph on the mark set $\cZ^\delta$ obtained by setting for $e = \{u,v\} \in E$, $\xi^\delta(u,v) = \{ \xi(u,v) \}_{\delta}^\kappa$.

\paragraph{Projective system. } We set $\Delta = \{ 2^{-i} : i \in \dN\}$. The set $\Delta$ is totally ordered with the mesh order $\preceq$ defined as $\delta  \preceq \delta'$ if $\delta' \leq \delta$.  Moreover, if $\delta  \preceq \delta'$ there is a natural projection $p_{\delta \delta'}$ from marked graphs with mark set $\cZ^{\delta'}$ to marked graphs with mark set $\cZ^\delta$ by setting $p_{\delta \delta'} (G) = G^\delta$. These projections define a projective system in the sense that $p_{\delta_1 \delta_2} \circ p_{\delta_2 \delta_3} = p_{\delta_1 \delta_3}$ for $\delta_1 \preceq \delta_2 \preceq \delta_3$ (because $\delta_{2}/\delta_3$ and $\delta_1 / \delta_2$ are integers). These projections extend to $\cGr(\cZ)$ and $\cGr(\cZ^\delta)$ the locally finite unlabeled rooted marked graphs on the mark sets $\cZ$ and $\cZ^\delta$. For the local topology, we can identify $\cGr(\cZ)$ with the projective limit of the sets $(\cGr(\cZ^\delta))_{\delta \in \Delta}$, that is sequence $(g_\delta)_{\delta \in \Delta}$ such that $g_\delta \in \cGr(\cZ^\delta)$ and $g_{\delta} = p_{\delta \delta'} (g_\delta)$ for all $\delta \preceq \delta'$, see \cite[Section 4.6]{De-Ze}. Indeed, for any $z \in \cZ$, $|z| <  \delta \lceil \kappa(\delta)/\delta \rceil$ for all $\delta$ small enough (for the usual order) and thus 
$$
| z - \{z \}^{\kappa(\delta)}_{\delta} |_\infty \leq \delta,
$$ 
 where $| x |_\infty = \max_i |x_i|$ is the $L^\infty$-norm on $\dR^k$. Note that the projections $p_{\delta\delta'}$ are continuous on $\cGr(\cZ^{\delta'})$.

The projective system extends to $\cP(\cGr (\cZ))$ and $\cP(\cGr( \cZ^\delta))$ equipped with the local weak topology. For $\mu \in \cP(\cGr (\cZ))$, we denote by $\mu^\delta \in \cP(\cGr( \cZ^\delta))$ the push forward of  $\mu$ by the map $g \to g^\delta$, that is the law of $(G^\delta,o)$ where $(G,o)$ has law $\mu$.

\paragraph{Large deviation for the projections. }  Recall that $G_n  = ( V_n,E_n,\xi_n) \in \cG_{n,\MM_n}$ is our   random graph. We denote by $G_n^{\delta} = (V_n,E_n,\xi^\delta_n)$ its quantized version.  The edge counting measure of $G_n^\delta$ is given for $z  \in \cZ^\delta$ by 
$$
\MM_n^\delta(z) = \sum_{e = \{u,v\}  \in E_n} \IND ( \xi_n^\delta( u, v) = z ) + \IND ( \xi_n^\delta( v, u) = z ).
$$
By construction, for every $b \in \cB$,
$$
\sum_{l \in L^\delta} \MM_n^\delta(b,l) = \MM_n(b).
$$

For $l \in L^\delta \backslash \{\omega\}$, we set $B^\delta_l = [l_1+\delta)\times \cdots \times [l_k + \delta)$. Similarly, we set $B_\omega = B_\omega^\delta = \{ x \in \dR^k : |x|_\infty \geq \kappa(\delta)\}$. For $l \in L^\delta$, $b \in \cB$ and an edge $ e = \{u,v\} \in E_n$, the probability  that $\xi_n^\delta(u,v) = (b,l)$ given $\{ \xi_n^\delta(u,v) \in \{b\} \times \dR^k\}$ is equal to 
$
\gamma_{n,b}^\delta(l) =  
\gamma_{n,b}(B^\delta_l) 
$. 
We set 
$$\gamma^\delta_{b}(l) = \gamma_{b}(B^\delta_l).$$
By assumption \eqref{eq:gamman2},
$$
\lim_{n \to \infty}\gamma_{n,b}^\delta(l) =  \gamma^\delta_{b}(l).
$$
%$$
%\MM_n^\delta((b,l)) =  \sum_{e = \{u,v\}  \in E_n} \IND ( \xi_n( u, v) \in \{b\} \times B_l ) + \IND ( \xi_n( u, v) \in  \{b^*\} \times B^*_l  ).
%$$
%Similarly, with $B_\omega = \{ x \in \dR^k : |x|_\infty \geq \kappa\}$ and $| x |_\infty = \max_i |x_i|$ the $L^\infty$-norm on $\dR^k$, we have
%$$
%\MM_n^\delta((b,\omega)) \sum_{e = \{u,v\}  \in E_n} \IND ( \xi_n( u, v) \in \{b\} \times B_\omega ) + \IND ( \xi_n( u, v) \in  \{b^*\} \times B_\omega  ).
%$$

If $b \ne b^*$, the law of the vector $\MM^\delta_n(b,\cdot) = (\MM_n^\delta(b,l) )_{l \in L^\delta}$ is a multinomial distribution with parameters $(\MM_n(b), (\gamma^\delta_{n,b}(l))_{l \in L^\delta}))$. If $b = b^*$, the law of the vector $\MM^\delta_n(b,\cdot)/2$ is a multinomial distribution with parameters $(\MM_n(b)/2, (\gamma^\delta_{n,b}(l))_{l \in L^\delta}))$. Let $\cP ( L^\delta)$ be the set of probability measures on $L^\delta$ (equipped with the weak topology). It follows from Sanov' Theorem that, as $n \to \infty$,  $\MM^\delta_n(b,\cdot) / \MM_n(b)$ satisfies a large deviation  principle  in $\cP(L^\delta)$ with speed $\MM_n(b)/2^{\IND_{b = b^*}}$ and good rate function 
$$
\DKL(\nu |\gamma^\delta_b) = \sum_{l \in L^\delta} \nu (l) \ln \frac{\nu(l)}{\gamma^\delta_b (l)},
$$
see e.g. \cite[Theorem 2.10]{De-Ze}. From \eqref{eq:limmnB}, the vector $\MM_n^\delta(b,\cdot)/n$ satisfies a large deviation principle on $\dR_+^{\cZ}$ with speed $n$ and good rate function
$$I^\delta_{d,b} (\DD ) =   \left\{ \begin{array}{ll} \frac{1}{2^{\IND_{b = b^*}}}\sum_l \DD(l) \ln \frac{\DD(l) }{d(b) \gamma_b^{\delta} (l)}     &\hbox{if } \sum_l \DD(l) = d(b) \\
\infty &\hbox{otherwise}.
\end{array}\right.
$$
 Moreover, the vectors $(\MM_n^\delta(b,l) )_{l \in L^\delta}$ are independent for different values of $b \in \cB$ since the marks were chosen independently. 
Hence, the vector $(\MM^\delta_n(z)/n)_{z \in \cZ^\delta}$ satisfies a large deviation  principle s in $\dR_+^{\cZ^\delta}$ with speed $n$ and good rate function given by
$$
I_d^\delta ( \DD )  =  \left\{ \begin{array}{ll} \frac{1}{2} \sum_{z = (b,l) \in \cZ^\delta} \DD(z) \ln \frac{\DD(z) }{d(b) \gamma_b^{\delta} (l)}     &\hbox{if for all $b \in \cB$, $z \in \cZ^\delta$, } \sum_l \DD(b,l) = d(b) \hbox{ and } \DD(z^*) = \DD(z) \\
\infty &\hbox{otherwise}.
\end{array}\right.
$$

Given the random vector $(\MM_n^\delta(z))_{z \in \cZ^\delta}$, the random graph $G^\delta_n$ is uniform on $\cG_{n,\MM_n^\delta}$. In other words, $G^\delta_n$ is a mixture of uniform measures on $\cG_{n,\MM_n^\delta}$. Writing explicitly the dependence in the marked set, $U(G_n^\delta) $ is a random element in  $\cP( \cGr(\cZ^\delta))$. For integer number $h \geq 1$, it follows from Biggins \cite[Theorem 5(b)]{MR2081460} and Proposition \ref{prop:SIGMA1}, that $U(G_n^\delta)_h$ satisfies a large deviation  principle  in $\cP( \cGr(\cZ^\delta))$ with speed $n$ and good rate function given, for $\mu \in \cP( \cGr_h(\cZ^\delta))$, by
$$
\Sigma^\delta_d (\mu,h) = \inf \left\{  \Sigma^0_{\DD} (\mu,h) + I_d^\delta(\DD)\right\},
$$
where the infimum is over all $\DD \in \dR_+^{\cZ^\delta}$. Now by Proposition \ref{prop:SIGMA1}, $ \Sigma^0_{\DD} (\mu,h) = \infty$ unless $\dE_{\mu} \DEG(o) = \DD$. We deduce that 
\begin{equation}\label{eq:fSd}
\Sigma^\delta_d (\mu,h) =   \Sigma^0_{\DD} (\mu,h) + I_d^\delta(\DD),
\end{equation}
with $\DD = \dE_{\mu} \DEG(o)$. 
%Similarly, for integer $h \geq 1$, $U(G_n^\delta)_h$ satisfies a large deviation  principle  in $\cP( \cGr_h(\cZ^\delta))$ with good rate function on $\cP(\cGr_h(\cZ^\delta))$,
%$$
%\Sigma^\delta_d (\mu,h) =  \vec \Sigma^0_{\DD} (\mu,h) + I_d^\delta(\DD).
%$$

\paragraph{Dawson-G\"artner's Theorem. } It follows from Dawson-G\"artner's Theorem \cite[Theorem 4.6.1]{De-Ze} that for every integer number $h \geq 1$, $U(G_n)_h$ satisfies a large deviation principle with good rate function 
%$$
%\Sigma_d(\mu) = \lim_{\delta \to 0} \Sigma_d^\delta(\mu^\delta).
%$$
%Similarly, $U(G_n)_h$ satisfies a large deviation principle with good rate function 
{{
\begin{equation}\label{eq:Sigmalim}
\Sigma_{\gamma,d}(\mu,h) = \sup_{\delta \in\Delta } \Sigma_d^\delta(\mu^\delta,h)= \lim_{\delta \rightarrow 0\atop \delta\in \Delta } \Sigma_d^\delta(\mu^\delta,h)
\end{equation}
where the last equality comes from the fact that $\mu\rightarrow  \Sigma_d^\delta(\mu^\delta,h)$ increases as $\delta\in\Delta$ goes to zero because our choice of discretization is strictly increasing. Indeed, the topology on $\cZ^{\delta}$ gets finer as $\delta$ goes to zero in $\Delta$ so that the volume of balls decreases with $\delta$, and hence the rate functions $\Sigma_d^\delta(\mu^\delta,h)$, which measure the volume of a small ball around $\mu$ after projection in $\cZ^{\delta}$, increase with $\delta$ going to zero.}}
We have for $\mu \in \cP(\cGr)$, $ \Sigma_{\gamma,d}(\mu) = \lim_{h \to \infty}  \Sigma_{\gamma,d}(\mu,h)$ by another application of Dawson-G\"artner's Theorem.
This proves the existence of the LDP for $U(G_n)$ in Theorem \ref{th:SIGMA1R}.

The same argument carries over for $\vec U(G_n)$. For $h \geq 1$, we find that $\vec U(G_n^\delta)_h$ satisfies a LDP with good rate function on $\cP(\cGe_h(\cZ^\delta))$ given by
\begin{equation}\label{eq:fSd2}
\vec \Sigma^{\delta}_{\DD}(\nu,h)  = \vec  \Sigma^0_{\DD} (\nu,h) + I_d^\delta(\DD),
\end{equation}
where under $\nu$, $\xi(\rho)$ has law  $\DD (\cdot) / \bar d$.

\subsection{Proof of Theorem \ref{th:SIGMA1R} and Lemma \ref{le:casinfini} : the rate function}

Let $h \geq 1$ and $\delta \in \Delta$ be an integer. Our goal is to compute,  for $\mu \in \cP(\cGr_h(\cZ))$ and $\nu \in \cP(\cGe_h(\cZ))$,
$$
\Sigma_{\gamma,d}(\mu,h) = \lim_{\delta \rightarrow  0}  \Sigma^\delta_d (\mu^\delta,h) \AND \vec \Sigma_{\gamma,d} (\nu,h) = \lim_{\delta \rightarrow 0}  \vec \Sigma^\delta_d (\nu^\delta,h) 
$$
where $\Sigma^\delta_d$ and $\vec \Sigma^\delta_d$ are defined in \eqref{eq:fSd} and \eqref{eq:fSd2}.
We may assume that $\mu$ is $d$-admissible. Otherwise $\Sigma_d(\mu,h) = \infty$ by using Proposition \ref{prop:SIGMA1} and Lemma \ref{le:casinfini}
for  $\delta$ small enough. Similarly, we may assume that $\nu$  is $d$-admissible.

\paragraph{Formula for $\Sigma^\delta_d (\mu^\delta,h)$ and $\vec \Sigma^\delta_d (\nu^\delta,h)$. }

%We denote in this subsection by $\cGr = \cGr(\cB)$ the set of unlabeled locally finite rooted graphs with marks in $\cB$. 
Recall that if $G = (V,E,\xi)$ is a marked graph on the mark set $\cZ = \cB \times \dR^k$, we write $\xi = (\xi^0,\xi^1)$, where $\xi^0$ is the mark on $\cB$ and $\xi^1$ is the mark on $\dR^k$, that is $\xi^0(u,v) = b$ if $\xi(u,v) = (b,x)$. Recall that we denote by $G^0 = (V,E,\xi^0)$, the corresponding $\cB$-marked graph. Beware that this notation should not be mistaken with the quantized marked graph  $G^{\delta}$ for $\delta \in (0,1)$ defined in the previous subsection. %Let $\mu^0 \in \cP (\cGr(\cB))$ be the law of the $\cB$-marked graph associated to $\mu$. We define  $\vec  \mu ^0 \in \cP( \vec \cGr(\cB))$ similarly with $\vec \mu$.

 Let $G$ and $\vec G$ be the uniform labeling of $\mu$ and $\vec \mu$ respectively. In this subsection, the distributions of $G^0$ and $\vec G^0$ on labeled $\cB$ marked graphs is denoted by $\mu^0$ and $\vec \mu^0$. We shall write $\dE_{\mu} [f(G)]$ and $\dE_{\vec \mu} [f(\vec G)]$ to insist on which graph, $G$ or $\vec G$, we take expectation. For $g$ in the support of  $\mu^0$ with vertex set $V_g $ and edge set $E_g$, let $\mu^1_g$ be the probability distribution of $(\xi^1(u,v))_{\{u,v\} \in E_g} \in (\dR^{k})^{E_g}$ given $G^0  = g$. The choice of the orientation $(u,v)$ of the edge $\{u,v\}$ is such that $u$ is the parent of $v$ (any other deterministic convention is possible). We define $\vec \mu^1_g$ for $g \in \vec{\cGr_h}$ in the support of $\vec{\mu}^0$ analogously.

Note that $d = \dE_{\mu^0} \DEG(o) = \dE_\mu \DEG (o) (\cdot \times \dR^k)$ by assumption. We set $\DD^\delta = \dE_{\mu^\delta} \DEG (o) \in (0,\infty)^{\cZ^\delta}$. By construction, for $(b,l) \in \cZ^\delta$, $\DD^\delta(b,l) = \dE_{\mu} \DEG (o) (\{b\} \times B_l^{\delta})$. Since $\sum_{l} \DD(b,l) = d(b)$, we find, with $\sigma_d$ as in \eqref{eq:defJdmuh},  
\begin{eqnarray}
\sigma_{\DD^\delta} (\deg(o))  + I_d^\delta( \DD^\delta)& =& \sigma_d(\deg(o)) - \frac{1}{2} \sum_{(b,l) \in \cZ^\delta} \DD^\delta(b,l) \ln \gamma_b^{\delta} (l) \nonumber \\
& = &\sigma_d(\deg(o)) - \frac{1}{2} \dE_\mu \sum_{e \in E_G : o \in e} \ln  \gamma^\delta_{b_e} (l_e),\label{eq:TIII}
\end{eqnarray}
where the last sum is over all edges $e = \{o,v\}$ adjacent to the root $o$ so that $\xi_{G} (o,v) = (b_e,x_e)$ with $x_e \in B_{l_e}^\delta$.
On the other hand, we have 
\begin{equation}\label{eq:dhiedhe}
H(G^\delta) = H(G^0) - \sum_{g} \mu^0_h(g) \sum_{l \in (L^\delta)^{E_g}} \mu^1_g(B^\delta_l) \ln \mu^1_g(B^\delta_l),
\end{equation}
where the first sum is over all $g$ in the support of  $\mu^0_h$,  and for $l = (l_e)_{e \in E} \in (L^\delta)^{E}$, we have set 
$$
B^\delta_l = \bigtimes_{ l \in E} B^\delta_{l_e} \subset (\dR^k)^E.
$$
We next consider the random graph $G_\gamma$ and use that $G^0_\gamma = G^0$. We obtain the identity
\begin{eqnarray*}
H(G^\delta) - H(G^0) +  \DKL(G^\delta | G_\gamma^\delta) &=& - \sum_{g} \mu^0_h(g) \sum_{l \in (L^\delta)^{E_g}} \mu^1_g(B^\delta_l) \ln \prod_{e\in E_g} \gamma_{b_e}^\delta(l_e) \\
& =& - \sum_{g,l \in L^\delta,e \in E_g} \dP( G^0 = g , x_e \in B^\delta_{l}) \ln \gamma_{b_e}^\delta(l)\\
& = & - \dE_{\mu}\sum_{e \in E_G} \ln \gamma_{b_e}^\delta(l_e),
\end{eqnarray*}
where at the last line, as in \eqref{eq:TIII}, for $e = \{u,v\}$ oriented as $(u,v)$, we have set $\xi_{G} (u,v) = (b_e,x_e)$ with $x_e \in B_{l_e}^\delta$.
If $ e = \{u,v\} \in E_G$, let us denote by $G^{(uv)}$ the marked graph rooted at $v$ obtained by considering the connected component of $v$ in the graph $G$ when the edge $\{u,v\}$ is removed. We can rewrite the above expression as 
\begin{align}
& H(G^\delta) - H(G^0) +  \DKL(G^\delta | G_\gamma^\delta)  \label{eq:TI} \\
& \quad =  - \dE_{\mu}  \sum_{e \in E_G : o \in e}  \ln \gamma_{b_e}^\delta(l_e) - \dE_{\mu}  \sum_{v \in V_G :  v \sim o } \sum_{e \in E_{G^{(ov)}}} \ln \gamma_{b_e}^\delta(l_e).  \nonumber
\end{align}

Similarly, 
$$
H(\vec G^\delta ) = H(\vec G^0 ) - \sum_{g} \vec \mu^0_h(g) \sum_{l \in (L^\delta)^{E_g}} \vec \mu^1_g(B^\delta_l) \ln \vec \mu^1_g(B^\delta_l).
$$
It follows that 
$$
H(\vec G^\delta) - H(\vec G^0)+ \DKL(\vec G^\delta  | \vec G_\gamma^\delta) = - \dE_{\vec \mu}   \sum_{e \in E_{\vec G}} \ln \gamma_{b_e}^\delta(l_e).
$$
Using the definition of $\vec \mu$ in \eqref{eq:defvecmu}, we deduce that 
\begin{align*}
& \frac{\bar d}{2} \left (H(\vec G^\delta) - H(\vec G^0)+ \DKL(\vec G^\delta | \vec G_\gamma^\delta ) \right)  \; = - \frac{1}{2} \dE_{\mu} \sum_{v  \in V_G : v \sim o} \sum_{e : e \in (G,o,v)_h}\ln \gamma_{b_e}^\delta(l_e) \\
& =  -  \frac{1}{2} \dE_{\mu}  \sum_{e  \in E_G: o \in e}  \ln \gamma_{b_e}^\delta(l_e) - \frac{1}{2} \dE_{\mu} \sum_{v  \in V_G : v \sim o} \left( \sum_{e \in E_{G^{(ov)}}} \ln \gamma_{b_e}^\delta(l_e)  +  \sum_{e \in E_{G^{(vo)}}}\ln \gamma_{b_e}^\delta(l_e)   \right) \\
& =  -  \frac{1}{2} \dE_{\mu}  \sum_{e \in E_G: o \in e}  \ln \gamma_{b_e}^\delta(l_e) - \dE_{\mu} \sum_{v \in V_G  : v \sim o}  \sum_{e \in E_{G^{(ov)}}}\ln \gamma_{b_e}^\delta(l_e) ,
\end{align*}
where at the last line, we have used the invariance of $\mu$ (property (C1)). 

Consequently, if we combine this last expression with \eqref{eq:TIII} and \eqref{eq:TI}, we deduce that 
\begin{equation}\label{eq:sigmadeltad}
\Sigma^\delta_{d} (\mu^\delta,h) =   \DKL(G^\delta | G_\gamma^\delta ) - \frac{\bar d}{2} \DKL(\vec G^\delta | G_\gamma^\delta)  + \Sigma^0_d(\mu^0,h),
\end{equation}
where we have used that $\Sigma^0_d(\mu^0,h) = J_d(\mu^0,h)$ defined in \eqref{eq:defJdmuh} in Proposition \ref{prop:SIGMA1}.

The same argument gives 
\begin{equation}\label{eq:sigmadeltad2}
\vec \Sigma^\delta_{d} (\nu^\delta,h) =   \frac {\bar d} 2 \DKL ( \vec G^\delta  | \vec G^\delta_\gamma ) - \bar d \DKL ( \check G^\delta  |  \check G_\gamma^\delta ) + \DKL ( G^\delta | G^\delta_\gamma )  + \vec \Sigma^0_d(\nu^0,h).
\end{equation}
where $G,\check G,\vec G$ are uniform labeling of $\dot \nu,\check{\dot\nu}$ and $\nu$.

\paragraph{Convergence of relative entropy. } 
We next claim that 
\begin{equation}\label{eq:smoothedDKLG}
\lim_{\delta \to 0} \DKL(G^{\delta} | G_\gamma^{\delta}) = \DKL(G  | G_\gamma),
\end{equation}
\begin{equation}\label{eq:smoothedDKLvecG}
\lim_{\delta \to 0} \DKL(\vec G^{\delta} | \vec G_\gamma^{\delta}) = \DKL(\vec G  | \vec G_\gamma),
\end{equation}
\begin{equation}\label{eq:smoothedDKLvecGo}
\lim_{\delta \to 0}  \DKL((\vec G^{(o)})^{\delta} | (\vec G^{(o)}_\gamma)^{\delta}) = \DKL(\vec G^{(o)}  | \vec G^{(o)}_\gamma).
\end{equation}
The proofs of \eqref{eq:smoothedDKLG}-\eqref{eq:smoothedDKLvecG}-\eqref{eq:smoothedDKLvecG}  are identical and are consequences of a general result. Note that the above limits could be infinite. We only prove \eqref{eq:smoothedDKLG}. Using the notation of \eqref{eq:dhiedhe}, we write
\begin{align*}
 & \DKL(G^{\delta} | G_\gamma^{\delta}) = \sum_g \mu^0(g) \sum_{l \in (L_\delta)^{E(g)} } \mu^1_g (B_l^\delta) \ln \frac{ \mu^1_g (B_l^\delta)  }{ \gamma_g (B_l^\delta) } =  \sum_g \mu^0(g)  \DKL(\mu^{1,\delta}_g | \gamma_{g}^\delta) .
\end{align*}
where $\gamma_g = \otimes_{g \in E(g)} \gamma_{\xi^0(e)}$ and $\mu^{1,\delta}_g$, $\gamma^{\delta}_g$ are the distributions on $(L_\delta)^{E(g)}$ defined by $\mu^{1,\delta}_g (l)  =  \mu^1_g (B_l^\delta)$ and $\gamma^{\delta}_g = \gamma_g(B_l^\delta)$.   Lemma \ref{le:discretization}  implies that for any fixed $g$, 
 $$
\lim_{\delta \to 0}   \DKL(\mu^{1,\delta}_g | \gamma_{g}^\delta) = \DKL(\mu^1_g | \gamma_g).
 $$
Moreover, on $\Delta = \{ 2^{-i} : i \in \dN\}$, $\delta \to   \DKL(\mu^{1,\delta}_g | \gamma_{g}^\delta) $ is monotone (by the variational formula \eqref{eq:varDKL} below). Thus, by monotone convergence, the claim  \eqref{eq:smoothedDKLG} follows.

\paragraph{Proof of Theorem \ref{th:SIGMA1R}. } All ingredients are gathered to conclude the proof of Theorem \ref{th:SIGMA1R}. It remains to show that if $\mu \in \cP(\cGr_h)$ is $(\gamma,d)$-admissible then $\Sigma_{\gamma,d}(\mu,h)$ is given by \eqref{eq:defSigmad} and if $\nu \in \cP(\cGe_h)$ is $(\gamma,d)$-admissible then $\vec \Sigma_{\gamma,d}(\nu,h)$ is given by \eqref{eq:defvecSigmad}. We prove the latter. Recall that $\Sigma^0_d(\mu^0,h) < \infty$. In particular, since $\Sigma_{\gamma,d}(\mu,h)  = \lim_{\delta} \Sigma_d(\mu^\delta,h) \geq 0$, the condition (C5), $\DKL(G|G_\gamma) < \infty$ and \eqref{eq:smoothedDKLG}-\eqref{eq:smoothedDKLvecG} imply that $\DKL(\vec G| \vec G_\gamma) < \infty$ (as otherwise $\Sigma_{\gamma,d}(\mu,h) = -\infty$  from \eqref{eq:sigmadeltad}). From \eqref{eq:sigmadeltad}-\eqref{eq:smoothedDKLG}-\eqref{eq:smoothedDKLvecG}, it gives the claimed formula.  The argument is identical for $\vec \Sigma(\nu,h)$. \qed

\paragraph{Proof of Lemma \ref{le:casinfini}. } We have already checked that if $\mu$ is not $d$-admissible then $\Sigma_{\gamma,d}(\mu,h) = \infty$. We should prove that if $\mu$ is $d$-admissible, well-behaved and not $(\gamma,d)$-admissible then $\Sigma_{\gamma,d}(\mu,h) = \infty$. Let $G,\check G,\vec G$ be uniform labeling of $\mu,\check \mu,\vec \mu$ respectively. We have $\DKL(G|G_\gamma) = \infty$ by assumption. Let $1 \leq k \leq h$ be the largest $k$ such that either $\DKL((G)_k|(G_\gamma)_k) = \infty$ or $\DKL((\vec G)_k|(\vec G_\gamma)_k) = \infty$. If $\DKL((\vec G)_k|(\vec G_\gamma)_k) < \infty$ then \eqref{eq:sigmadeltad}-\eqref{eq:smoothedDKLG}-\eqref{eq:smoothedDKLvecG} applied to $k$ gives $\Sigma_{\gamma,d}(\mu_k,k) = \infty$. However, since the map $\mu \to \mu_k$ is continuous for the weak topology, the contraction principle gives 
$$
\Sigma_{\gamma,d}(\mu,h) \geq \Sigma_{\gamma,d}(\mu_k,k) =\infty,
$$
as requested. Otherwise $\DKL((\vec G)_k|(\vec G_\gamma)_k) = \infty$ and, by construction $\DKL((G)_{k-1}|(G_\gamma)_{k-1}) < \infty$. Since $\mu$ is well-behaved, we have $\DKL((\check G)_{k-1}|(\check G_\gamma)_{k-1}) < \infty$. Then \eqref{eq:sigmadeltad2}-\eqref{eq:smoothedDKLG}-\eqref{eq:smoothedDKLvecG}-\eqref{eq:smoothedDKLvecGo} applied to $k$ gives $\vec \Sigma_{\gamma,d}(\vec \mu_k,k) = \infty$. By contraction, it implies $\vec \Sigma_{\gamma,d}(\vec \mu,h) = \infty$. By Lemma \ref{le:SIGMAe}, we further get $\Sigma_{\gamma,d}(\mu,h) \geq \vec \Sigma_{\gamma,d}(\vec \mu,h) = \infty$. It concludes the proof that $\Sigma_{\gamma,d}(\mu,h) = \infty$. The same proof gives $\vec \Sigma_{\gamma,d}(\nu,h) = \infty$ if $\nu$ is well-behaved and not $(\gamma,d)$-admissible. \qed

\begin{remark}[Alternative expression for $\Sigma_{\gamma,d}(\mu,h)$]\label{rk:kavita}
 The recent work \cite{ramanan2023large} gives an alternative formula of $\Sigma_d(\mu^0,h)$ as the sum of $2h$ relative entropies. Using their expression, our discretization procedure and the general convergence of discretized relative entropies given by Lemma \ref{le:discretization}, it should be possible to extend their formula to our setting with real marks. 
\end{remark}

\subsection{Proof of Theorem \ref{th:SIGMADR}}

It follows from \cite[Subsection 5.4]{DeAn} and \cite[Section 6.1]{BoCa15} that  $U(H^0_n)_h$  satisfies a LDP with good rate function on $\cP( \cGr_h(\cB))$, if $\mu$ is admissible and $\DEG(o)$ has law $\pi$ under $\mu$,
$$
\Sigma^0_\pi (\mu,h) = \Sigma^0_{d} (\mu,h) - \Sigma^0_d (\pi,1),
$$
and $\Sigma^\delta_\pi (\mu,h) = \infty$ otherwise. We may then reproduce the above argument. We find that 
$U(H^\delta_n)_h$  satisfies a LDP with good rate function on $\cP( \cGr_h(\cZ^\delta))$, if $\mu$ is admissible and $\DEG^0(o)$ has law $\pi$ under $\mu$,
$$
\Sigma^\delta_\pi (\mu,h) = \Sigma^\delta_{d} (\mu,h) - \Sigma^0_d (\pi,1).
$$
 Using Dawson-G\"artner's Theorem,  $U(H_n)_h$ satisfies a LDP with good rate function:
$$
\Sigma_{\gamma,\pi} (\mu,h) = \lim_{\delta \to 0}\Sigma^\delta_{\pi} (\mu,h).
$$
Similarly, $\vec U(H^\delta_n)_h$ satisfies a LDP with good rate function
$$
\vec \Sigma^\delta_{\pi} (\nu,h) = \vec \Sigma^\delta_{d} (\nu,h) - \vec \Sigma^0_d (\pi,1).
$$
if $\nu = \vec \mu$ for some $\mu \in \cP(\cGr(\cZ^{\delta})$ which is $(\gamma,\pi)$-admissible. Otherwise, $\vec \Sigma^\delta_{\gamma,\pi} (\nu,h) = \infty$. We may then conclude as in the proof of Theorem \ref{th:SIGMADR}. Note that in this case, since $\pi$ has bounded support, for any integer number $1 \leq k \leq h$, $\DKL((G)_k | (G_\gamma )_k < \infty$ implies that $\DKL((\check G)_k | (\check G_\gamma )_k < \infty$ where $G,\check G$ are uniform labeling of $\mu$ and $\check \mu$. \qed

%\begin{remark}
%We may rewrite $\Sigma_{\pi}(\mu,h)$ in terms of relative entropies.  Indeed, as already pointed in the proof of Theorem \ref{th:SIGMAD},
%$$
%\Sigma(\mu^0,h) - \Sigma(\mu^0,1) = \Sigma_{\pi} (\mu^0,h),
%$$
%where $\Sigma_{\pi} (\mu^0,h)$ was defined in Theorem \ref{th:SIGMAD}. Moreover,  
%$$
%\DKL ((G,o)_{h} | (G_\gamma,o)_{h} ) = \sum_{g} \mu^0(g) \int \mu^1_g \ln \left( \frac{d\mu^1_{g |g_1} )}{d\mu_\gamma^1_{g |g_1))}}\right)
%$$
%\end{remark}

\subsection{Proof of Theorem \ref{th:SIGMAER}}

Theorem \ref{th:SIGMAER} is an easy corollary of Theorem \ref{th:SIGMA1R}. The argument is given in \cite[Section 6.3]{BoCa15}. It suffices to check that $j_d$ is the rate function for $2 |E_n| / n$, the average degree in $G'_n$ and invoke the large deviation principles for mixtures, Biggins \cite[Theorem 5(b)]{MR2081460}.\qed

\subsection{Proof of Lemma \ref{le:minimizer}}

We first assume that $\gamma$ has a finite support. Then the lemma is a  consequence of \cite[Lemma 4.10]{MR4492967} (see also Remark 4.12 there and  \cite{ramanan2023large}). In the general case, this follows from \eqref{eq:Sigmalim}. Indeed, using the notation from \eqref{eq:Sigmalim}, if $\mu \ne \UGW(\gamma,\pi)_h$ then there exists $\delta >0$ such that $\mu^{\delta} \ne \UGW(\gamma,\pi)_h^{\delta}$. Hence $\Sigma_d^{\delta} (\mu^{\delta},h) > \Sigma_d^{\delta} (\UGW(\gamma,\pi)^{\delta}_h,h) = 0$ from what precedes. Since the limit in \eqref{eq:Sigmalim} in non-decreasing, the claim follows. \qed

\section{Microstate entropy for networks}

\label{sec:networks}
In this section, we extend the previous results on entropy for locally finite marked graphs to the locally finite networks defined in \cite{MR2023650,BordCap}. In short, networks are marked graphs where the definition of neighborhoods depend on the marks.

\subsection{Benjamini-Schramm convergence}

%\paragraph{Networks} In short, networks are marked graphs which are projective limits of locally finite  marked graphs. More precisely, as above, we consider marked graphs $G = (V,E,\xi)$ with weights on  $\mathcal Z$ where $\xi(u,v) = \xi(v,u)^*$ for all edges $\{u,v\} \in E$ (and $*$ is an involution as in Subsection \ref{susbec:involution}). We assume that we are given a collection of projection operators on marked graphs $G \to G^\veps$ indexed by $\veps >0$ such that for $0 < \veps' < \veps$, 
%$$
%(G^{\veps})^{\veps'} = G^{\veps'},  
%$$
%and such that on locally marked 
%We then say that the marked graph $G$ is a {\em network} if for all $\eps > 0$, $G^\eps$ is locally finite (as a marked graph). We denote by $\cNr = \cNr ( \mathcal Z)$, the set of connected rooted networks up to isomorphisms. An element of $\cNr$ will be called an unlabeled rooted network.

\paragraph{Euclidean network}  In short, networks are marked graphs which are projective limits of locally finite  marked graphs. We will only consider here networks on the marked set $\mathcal Z = \dR^k$  and $*$ is an involution as in Subsection \ref{susbec:involution}. As above, we consider marked graphs $G = (V,E,\xi)$ with weights on  $\mathcal Z$ where $\xi(u,v) = \xi(v,u)^*$ for all edges $\{u,v\} \in E$. For $z \in \cZ$, $|z|$ denotes the Euclidean norm.{{
For $ \varepsilon > 0$, we denote by $G^\eps = (V,E^\eps,\xi) $ the marked graph  spanned by all edges $\{u,v\} \in E$ such that $|\xi (u,v)| \geq \eps$.  We assume that for all edges $\{u,v\} \in E$,}}
\begin{equation}\label{eq:nonzero}
|\xi(u,v)| >0.
\end{equation}
We then say that the marked graph $G$ is a {\em (Euclidean) network} if for all $\eps > 0$, $G^\eps$ is locally finite (as a marked graph). We denote by $\cNr = \cNr ( \mathcal Z)$, the set of connected rooted networks up to isomorphisms. An element of $\cNr$ will be called an unlabeled rooted network.

\paragraph{Local weak topology} If $ g= (G,o)$ is a rooted network, for real $\eps > 0$ and integer number $r \geq 0$, we denote by $g^\eps_r = (g^\eps)_r$ the marked $r$-neighborhood of $g^\eps$ for the graph metric, we call $g^\eps_r$ the $(\eps,r)$-neighborhood of $g$. The local topology on $\cN_r$  is the product topology inherited from projections around the root. More precisely, let $g = (G,o)$, $g' = (G',o') \in \cNr$, we say that the triple $(r,\delta,\eps)$ is good for $(g,g')$ if there exists an isomorphism $\psi$ between  $\bar g^{\eps}_{r}$ and $\bar g'^{\eps}_{r}$ such that  
$$\sup_{u,v\in V(G_{r})}d ( \xi(u,v), \xi'_{(\psi(u),\psi(v))}) \leq \delta,$$
where $d((b,x),(b',x')) = \IND( b \ne b') + |x -x'|$.  
We then endow $\cNr$ with the local topology, which is compatible with the distance
$$\DLOC(g,g')=\inf\left\{ \frac{1}{1+r} +\delta + \eps :(r,\delta,\eps)  \mbox{ is  good for  }(g,g')\right\}.$$

Note that condition \eqref{eq:nonzero} is here to guarantee that $\DLOC$ is indeed a distance.  The set $\cNr$ equipped with $\DLOC$ is a complete separable metric space. We denote by $\cP( \cNr)$ the set of probability measures on $\cNr$ equipped with the topology of weak convergence. %It is compatible with the Lévy-Prohorov distance associated to $\DLOC$ which we denote by $\DWLOC$. 

\paragraph{Neighborhood distribution, unimodularity and invariance}
For a finite network $G = (V,E,\xi)$, we define $U(G)$ as in \eqref{eq:defU} and 
the Benjamini-Schramm convergence is defined as for marked graphs. For integer number $h \geq 1$, $U(G_n)_h$ is the distribution of the depth-$h$ neighborhood. Unimodularity is defined as in \eqref{eq:defunimod} except that $f$ should now be a non-negative measurable function on $\cNe$, the edge-rooted networks up to isomorphisms. The notion of invariance for the law of depth-$h$ neighborhood extends also.

% inutile en l'état...
%\paragraph{Standard labeling}
%Since a vertex in a network may have infinite degree, the random labeling used to randomly label marked graphs is not well-defined anymore. We define instead  a new labeling defined as follows which we call standard labeling. Recall that $\dN^f$ defined in \eqref{eq:defNf} denote the set of finite integer sequences. 
%
%Let $(G,o)$ be a random rooted network. We perform the breadth-first search tree on the network $G$ started at the root $o$ by non-increasing value of $|\xi(u,v)|$ where ties between vertices are broken uniformly at random and independently of $(G,o)$.  This defines a random marked graph $(G',o)$ on a subset of $\dN^f$ whose law depends only the equivalence class of $(G,o)$. A vertex $v$ at distance $k$ from the root receives a label in $\dN^k$, if ${\bm i} = (i_1,\ldots,i_{k-1})$ is the label of its parent in the search tree, $v$ gets the label $(i_1,\ldots,i_{k-1},j) = ({\bm i},j)$ if it is the $j$-th offspring of its parent in the random ordering, $|\xi(u,1)| \geq |\xi(u,2)| \geq \ldots$. We call this random rooted marked graph, the {\em standard labeling}. We note that the standard labeling is preserved by the projection $(G,o) \to (G^{\eps},o)$, $\eps >0$ (in distribution if random tie-breaking are used).
%
%
%
%In the sequel, when the choice of the root is clear, we will often omit the root and write $G$ in place of $(G,o)$ of $(G,\rho)$. 

\subsection{Entropy for networks}

\paragraph{Random network with a given intensity}  
For integer number $n \geq 1$, let $\gamma_n$ be a probability measure on $\dR^k$ which is $*$-invariant in the sense that for all Borel $A \subset \dR^k$, $ \gamma_{n}(A) = \gamma_{n}(A^*)$. We assume that as $n \to \infty$, for all Borel $A \subset \dR^k$, $0 \notin A$, 
\begin{equation}\label{eq:gamman2N0}
\lim_{n \to \infty} n \gamma_{n}(A) = \Lambda (A),
\end{equation}
where $\Lambda$ is a non trivial Radon measure on $\dR^k \backslash \{0\}$ which we call the intensity measure of the model. We assume that $\Lambda$ is finite at infinity, that is, if $B_1$ is the unit ball of $\dR^k$, we assume that
\begin{equation}\label{eq:Lambdabdd0}
\Lambda(\dR^k \backslash B_1) < \infty.
\end{equation}

Let $V_n = \{1,\ldots,n\}$ and $K_n = \{\{u,v\} : 1 \leq u < v \leq n\}$. We define a random network  $G_n = (V_n,E_n,\xi_n) $ by sampling $\xi(u,v)$ for each $\{u,v\} \in K_n$, $u < v$, independently according to the probability distribution $\gamma_{n}$. Finally, we set $\xi_n(v,u)= \xi_n(u,v)^*$ and $E_n = \{ \{ u, v\} \in K_n : \xi_n(u,v) \ne 0 \}$.

For example, if $k = 1$ and  $\gamma_n  = (1-d/n) \delta_0 + (d/n ) \delta_1$ for some $d >0$, then the network $G_n$ is an Erd\H{o}s-Rényi random graph with edge probability $d/n$. In this case, we have 
$\Lambda = d \delta_1.$
 If $k = 1$ and $\gamma_n ( [0,t] ) = e^{-1 / ( nt^\alpha )}$ for all $t > 0$ and some $\alpha >0$, then $G_n$ is the mean-field model of distance defined in \cite{MR2023650} when $\alpha$ is an integer number (the weights  in \cite{MR2023650} are the inverse of our weights). Or more generally, $\gamma_{n}$ is the distribution of $X/n^{\alpha}$ where $X$ is a non-negative  random variable so that $P(|X|\ge t)\simeq t^{-\alpha}$ when $t$ goes to infinity.
  In this case, the intensity measure $\Lambda$ is the measure on $\dR_+$, 
\begin{equation*}
\Lambda (dt) = \alpha t^{-\alpha -1}  dt.
\end{equation*}
Observe that in this case $\Lambda(\mathbb C\backslash\{0\})=+\infty$.

\paragraph{LDP for random networks} We now state our result on the large deviation  principle  for the sequence of random networks $(G_n)$ with weight distribution $\gamma_n$ and asymptotic intensity $\Lambda$.

\begin{theorem}[Entropy for networks]\label{th:SIGMA1N}
For any integer number $h \geq 1$, $U(G_n)_h$ satisfies a LDP on $\cP(\cNr_h)$ with rate $n$ and good rate function $\Sigma_{\Lambda} (\mu,h)$ defined below.  Moreover $U(G_n)$ satisfies a LDP on $\cP(\cNr)$  with rate $n$ and good rate function $\Sigma_{\Lambda} (\mu) = \lim_{h} \Sigma_{\Lambda}(\mu_h,h)$ 
\end{theorem}

We now define $\Sigma_{\Lambda}(\mu,h)$ for $\mu \in \cP ( \cNr_h)$. For $\mu \in \cP( \cNr)$ and for $\eps >0$, we define $\mu^{\eps}$ as the push-forward of $\mu$ by the $\cNr \to \cGr$ map :  $G \mapsto G^\eps$. The map $\mu \to \mu^{\veps}$ is not continuous. However, it is continuous at $\mu$ if $\dP_{\mu} ( \exists \{u,v\} \in E : |\xi(u,v)| =  \veps )  = 0$. We denote by $\mathcal E(\mu)$ the set of $\varepsilon$ such that $\nu \to \nu^{\veps}$ is continuous at $\mu$. From what precedes, $0$ is an accumulation point of $\mathcal E(\mu)$. We set
\begin{equation}\label{eq:defdeps}
d_\eps = \int \IND ( |z|  \geq \eps) d \Lambda( z).
\end{equation}
Note that $d_\eps < 0$ by Assumption \eqref{eq:Lambdabdd0} and since we assumed that $\Lambda$ is Radon. For $\eps$ small enough, $d_\veps >0$ and we can define the probability measure: for $A \subset \dR^k$,
\begin{equation}\label{eq:defLambdaeps}
\Lambda_\eps (A)  = \frac{ \Lambda (A \cap \{ z : |z| \geq \veps\})}{d_\veps}.
\end{equation}
Recall the rate $\Sigma^{\tiny {\rm ER}}_{\gamma,d}$ appearing in Theorem \ref{th:SIGMAER}. Finally, the rate function $\Sigma_{\Lambda}(\mu,h)$ is defined for $\mu \in \cP ( \cNr_h)$ as
\begin{equation}\label{eq:SigmaLambda}
\Sigma_{\Lambda}  (\mu,h) =  \lim_{\eps \to 0, \eps \in \mathcal E(\mu)} \Sigma^{\tiny {\rm ER}}_{\Lambda_\eps,d_\veps}  (\mu^\veps,h),
\end{equation}
where the limit is taken along continuity points at $\mu$ of the map $\nu \to \nu^{\veps}$.

\paragraph{Global minimizer of the entropy}

The unique global minimizer of the entropy $\Sigma_{\Lambda}$ is a natural generalization of the {\em Poisson Weighted Infinite Tree} introduced by Aldous in his resolution of the $\zeta(2)$-conjecture of Parisi and M\'ezard for the random assignment problem \cite{zbMATH00123709,zbMATH01655645,MR2023650}. 

For $\Lambda \ne 0$, this random rooted network $g = (V,E,\xi,o)$ is defined as follows. The vertex set is a subset of $\dN^f$ defined in \eqref{eq:defNf}. For each $v \in \dN^f$, we sample an independent Poisson point process $\Xi_v$ on $\dR^k \backslash \{0\}$ with intensity $\Lambda$. From assumption \eqref{eq:Lambdabdd0}, the support of $\Xi_v$ is almost surely bounded and we may order the points of the process as 
$$
\Xi_v = \{ \xi_{v,k} : | \xi_{v,1} | \geq | \xi_{v,2} | \geq  \ldots \},
$$ 
where we break ties uniformly at random. If $\Lambda(\dR^k) = d$ is finite then $d(v) = \Xi_v( \dR^k)$ is a Poisson variable with mean $d$. If $d = \infty$, then $\Xi_v$ is a.s. infinite. Now, we consider the random network on $\dN^f$ where between $v = (i_1,i_2,\ldots,i_h)$ to $u = (i_1,\ldots,\ldots,i_{h},k) = (v,j)$ with $1 \leq j \leq d(v)$, we place an edge $\{u,v\}$ with associated mark 
$$
\xi(u,v) = \xi(v,u)^* = \xi_{v,k}.
$$ 
If $d = \infty$, then this network is a.s connected. If $d < \infty$, we define $ V \subset \dN^f$ as the connected component of the root $o$. We denote by $\PWIT(\Lambda) \in \cP(\cNr)$ the distribution of this random rooted network. For example if $\Lambda = d \gamma$ where $\gamma$ is a $*$-invariant probability measure on $\dR^k$, then $\PWIT(\Lambda)$ coincides with $\UGW(\gamma,\Poi(d))$.  As announced we have the following result.

\begin{lemma}[Minimizer of the network entropy]\label{le:minimizerPWIT}
For any $h \geq 1$ and $\mu \in \cP(\cNr_h)$, we have  $\Sigma_{\Lambda} (\mu,h) = 0$ if and only if  $\mu = \PWIT(\Lambda)_h$.\end{lemma}

\subsection{Proof of Theorem \ref{th:SIGMA1N}}

Fix $\veps >0$ and consider the marked random graph $G_n^\veps$. By construction, $G_n^\veps$ has distribution $G_n'$ defined in Subsection \ref{subsec:LDPER} with parameter $(\gamma_{n,\veps},d_{n,\veps})$ 
$$
d_{n,\veps} = n \gamma_n( |z| \geq \eps),
$$
and $\gamma_{n,\veps} (\cdot) = \gamma_n (\cdot \cap \{  |z| \geq \eps \}) / \gamma_n( |z| \geq \eps)$ (if $ \gamma_n( |z| \geq \eps) = 0$ then $\gamma_{n,\veps}$ is arbitrary). Since $\Lambda \ne 0$, $d(\veps) > 0$ for all $\veps$ small enough. Also, as $n\to \infty$, $d_{n,\veps}$ converges to $d_\veps$ and $\gamma_{n,\veps}$ toward $\Lambda_{\veps}$ because we assumed that $\eps\in\mathcal E(\mu)$. It follows that from Theorem \ref{th:SIGMAER} that $U(G_n)^\eps_h$ satisfies a LDP with good rate function $\Sigma^{\tiny {\rm ER}}_{\Lambda_\eps,d_\veps} $. Theorem \ref{th:SIGMA1N} is then a consequence of  Dawson-G\"artner's Theorem \cite[Theorem 4.6.1]{De-Ze}. \qed

\subsection{Proof of Lemma \ref{le:minimizerPWIT}}

Recall the definition of truncation map $\mu \to \mu^{\veps}$ below Theorem \ref{th:SIGMA1N}. For any $\veps >0$, let $d_\veps$ and $\Lambda_\veps$ be as in \eqref{eq:defdeps}-\eqref{eq:defLambdaeps}. The Poisson thinning property implies that $\PWIT(\Lambda)^\veps$ is equal to $\UGW(\Lambda_\veps, \Poi(d_\veps))$. Thus Lemma \ref{le:minimizer} implies that $\Sigma^{\tiny {\rm ER}}_{\Lambda_\veps,d_\veps}(\PWIT(\Lambda)^\veps_h,h) = 0$. Hence from \eqref{eq:SigmaLambda}, $\Sigma_{\Lambda} ( \PWIT(\Lambda )_h,h ) = 0$. Now, if $\mu \ne \PWIT(\Lambda)_h \in \cP (\cNr_h)$, there exists $\veps >0$ such that $\mu^{\veps} \ne \PWIT(\Lambda)^\veps_h$. Again, Lemma \ref{le:minimizer} implies that $\Sigma^{\tiny {\rm ER}}_{\Lambda_\veps,d_\veps} (\mu^\veps,h) > 0$. Since the limit in \eqref{eq:SigmaLambda} in non-decreasing (Dawson-G\"artner's Theorem \cite[Theorem 4.6.1]{De-Ze}), the conclusion follows.

\section{Large deviations for  the ESD of heavy-tailed matrices}

\label{sec:LDPHT}

 In this section, we explain how a large deviation  principle  for the ESD of heavy-tailed matrices can be deduced by contraction from the large deviation  principle   for marked graphs that we studied in the previous section.  To this end, we first discuss how to associate an operator to a marked graph and discuss the smooth dependence of its 
 ESD in terms of the neighborhood distribution of these graphs. We then deduce by the contraction principle the large deviation  principle  for the ESD stated in Theorem \ref{thA}, Theorem \ref{thB} and Theorem \ref{LDPHT}.

\subsection{ESD of sparse matrices and associated graphs}

The set $\dC \simeq \dR^2$ is equipped with the involution $z \to \bar z$.  We associate to an Hermitian matrix $A \in M_n(\dC)$ a marked graph  $G(A)$ on vertex set $V_n = \{1 , \ldots, n \}$, edge set $E_A = \{ (i,j) : |A_{ij}| \ne 0 \}$ and weight $\xi(i,j) = \bar \xi(j,i) = A_{ij}$.

Conversely, if $G = (V,E,\xi)$ is a marked graph on $\dC$ such that for all $u \in V$, 
\begin{equation}\label{eq:hypAG}
\sum_{v : \{u,v \} \in E} |\xi(u,v)|^2 < \infty,
\end{equation}
we may associate an operator $A_G$ defined on $\ell^0(V)$, the set of compactly supported functions of $\ell^2(V)$, by for $\psi \in \ell^0(V)$,
\begin{equation}\label{eq:degAG}
A_G \psi (u)  = \sum_{v : \{u,v \} \in E} \xi(u,v) \psi(v) .
\end{equation}
Note that if $G = G(A)$ then $A_G = A$. Note also that condition \eqref{eq:hypAG} is satisfied if $G$ is a locally finite marked graph. Condition \eqref{eq:hypAG} implies that $G$ is a network as defined in Section \ref{sec:networks}.  If $A_G$ has a unique self-adjoint extension, we may define unambiguously its spectral measure at vector $\psi \in \ell^0 (V)$ $L^\psi_{A_G} \in \cP(\dR)$ by  $L^{\psi}_{A_G}(f)=\langle \psi, f(A_{G})\psi \rangle$ for all $f$ bounded continuous. 
If $G = G(A)$ is the marked graph graph associated to an Hermitian $n\times n$ matrix $A$, then the ESD defined in \eqref{def:ESD} is the average of spectral measures: 
\begin{equation}\label{eq:asmf}
L_A = \frac{1}{n} \sum_{v = 1}^n L^{e_v}_{A} = \dE_{U(G(A))} [  \mu^{e_o}_{A_G}],  
\end{equation}
where $o$ is the uniformly distributed root and $e_v$ is the canonical vector $e_v (u) = \IND(v = u)$.

\subsection{Continuity of mean spectral measure}

If $Y$ is an Hermitian matrix, the measure $L_{Y}$ belongs to the set $\mathcal P(\mathbb R)$ of probability measures on the real line. We equip $\mathcal P(\mathbb R)$ with the weak topology where a sequence $\mu_{n}$ converges towards $\mu$ iff 
   $$\lim_{n\rightarrow\infty}\int f(x)d\mu_{n}(x)=\int f(x)d\mu(x)$$ for every bounded continuous function $f$. 
 This space is Polish. We equip it with the following distance
 $$d (\mu,\nu)=\sup_{\|f\|_{L}\le 1, \|f\|_{\TV}\le 1} \int f d(\mu-\nu) $$
 where 
 $$\|f\|_{L}=\sup_{x\neq y}\frac{|f(x)-f(y)|}{|x-y|}\, \quad\mbox{ and }\quad  \|f\|_{\TV}=\sup \left\{\sum_{k}|f(x_{k+1}-f(x_{k})|\right\}.$$
 where the supremum is taken on increasing sequences $(x_{k})_{k\in\mathbb N}$. 
  We discuss in the next lemma the continuity of the spectral measure of the eigenvalues of a  self-adjoint matrix in terms of the local weak limit of its associated graph $U(G(Y_{n}))$. Such a continuity was already studied in depth, in particular in \cite{MR2724665,BordCap,AbertThomVirag,BSV,MR3791802} in various settings. For the sake of completeness, we recall the main arguments as well as how to define the spectral measure associated to a general unimodular $\nu\in \mathcal P(\cGr)$.

 \begin{lemma}\label{smooth} Let $\nu \in \cP(\cGr)$ be a unimodular measure. There exists a probability measure $L_\nu \in \cP(\dR)$ such that the following holds.
  Let $Y_{n}, n\ge 0,$ be a sequence of $n\times n $ self-adjoint matrices with associated marked graphs $G(Y_n)$. Assume that $U(G(Y_n))$ converges toward $\nu \in \cP(\cGr)$ for the local weak topology.  
 \begin{enumerate}

 \item  If the degrees and marks are uniformly bounded, namely $$\sup_n\max_{i,j}|Y_{n}(i,j)|<\infty \AND \sup_{n}\max_{i}\sum_{j} 1_{Y_{n}(i,j)\neq 0}<\infty,$$ 
 then $L_{Y_n}\in \mathcal P(\mathbb R)$ converges  in moments to $L_{\nu}$. 
 
% \item In general, we have $L_{Y_n}$ converges weakly to the probability measure $\mu_{\nu}$.
%  \item 
% Let $\psi : [0,\infty) \to [0,\infty)$ be a non-decreasing function with $\lim_{x \to \infty} \psi(x) = \infty$. 
% \begin{itemize} \item If $\nu=U(G(Y))$ is the law of a finite marked graph with bounded marks, then $\mu_{\nu}$ is the expected empirical measure of the eigenvalues of $Y$.
% \item Let $M$ be a finite real number  The closure of set  
% $$K_{M}=\{\nu\in\mathcal P({\mathcal G}^{\bullet}): \nu(\deg_{G}(o)\log \deg_{G}(o))\le M\}$$
%  is compact.
% Moreover, the application $\nu\in K_{M}\rightarrow \mu_{\nu}\in \mathcal P(\mathbb R)$ is continuous.
%  \item If $U(G(Y_{n}))\in K_{M}$, and if $U(G(Y_{n}))\in  \cP( \cGr)$  converges  to a measure $\nu$ (for the local weak convergence), $L^{Y}_{n}\in \mathcal P(\mathbb R)$ converges weakly to the probability measure $\mu_{\nu}$.
% \end{itemize}

 \item In general,  $L_{Y_n}$ converges weakly to the probability measure $L_{\nu}$.

 \item We have that $A_G$ is $\nu$-a.s. essentially self-adjoint and $L_{\nu} = \dE_{\nu}[L_{A_G}^{e_o}]$. Moreover, the map   $\nu \rightarrow L_{\nu}$ is continuous on the set of unimodular measures.

 \end{enumerate}
 \end{lemma}
 
Note that if $\nu$ is supported on finite marked graphs, the formula $L_{\nu} = \dE_{\nu}[L_{A_G}^{e_o}]$ boils down to the identity \eqref{eq:asmf}.  
 
  \begin{proof}[Proof of Lemma \ref{smooth}]
  $\bullet$  For the first point, it is enough to prove the convergence of moments. Because the degrees are uniformly bounded by $D$, $Y_{n}$ is uniformly bounded in operator norm by $D\sup_{n\in\mathbb N}\sup_{i,j}|Y_{n}(i,j)|<\infty. $ Hence, the moments of $L_{Y_{n}}$ are uniformly bounded. Moreover, the $k$th moment of $L_{Y_{n}}$ is easily seen to be a smooth function of $U(G(Y_{n}))$ since
  $$L_{Y_{n}}(x^{k})=\frac{1}{n}\tr (Y_{n}^{k})=\frac{1}{n}\sum_{i=1}^{n} f_{k }((G(Y_{n}),i))= \int f_{k } dU(G(Y_{n}))$$ where
  $f_{k}((G,o))$ is the local function on rooted graphs given by
  $$f_{k}((G,o))=\sum _{w:o\rightarrow o}\prod_{e\in w} \xi_{e}$$
  where we sum over all loops of length $k$ going through $o$ and take the products over the marks at the edges of these loops.
  Since we assume the degrees uniformly bounded, $f_{k}$ are finite degree polynomials in the marks, and hence continuous for the local weak convergence. We conclude that $L_{Y_{n}}(x^{k})$ is a continuous function of $U(G(Y_{n}))$ which converges towards $\int f_{k}d\nu$ under our hypothesis. 
  
$\bullet$  We now prove the second point. When the degrees or weights are unbounded, and more specifically when they do not have all the moments finite, it is not possible to argue by using the moments which may be unbounded and ill-defined. The idea is to approximate our distributions by ones with bounded degrees and weights.
  The central point is to recall that if $A,B$ are two self-adjoint operators with spectral measures $L_{A}$ and $L_{B}$,
  \begin{equation}\label{cont}
  |L_{A}(f)-L_{B}(f)|\le \|f\|_{\TV} \mbox{Rank}(A-B)\end{equation}
  where $\|f\|_{\TV}$ denotes the total variation norm of the function $f$.
 We can use this result to approximate  the empirical measures $L_{Y_n}$ of the eigenvalues of  the matrices $Y_{n}$ by the empirical measures $L_{Y_{n,k}}$ of the matrices $Y_{n,k}$ defined for some fixed integer number $k$ as follows: 
  the entries in columns or rows with either (i) more than $k$ non-zero entries or (ii) an entry larger that $k$ are put to zero.
  This corresponds, at the level of the underlying graph $G(Y_{n})$, to remove 
 edges connected to vertices with degree higher than $k$  or connected to a weight larger than $k$ to get $G(Y_{n,k})$. 
 We deduce from \eqref{cont} that
\begin{align*} 
 |L_{Y_{n}}(f)-L_{Y_{n,k}}(f)|& \le \|f\|_{\TV}\frac{2}{n}\sum_{i=1}^{n}\IND_{\sum_j  \IND_{Y_{ji}\neq 0}\ge k \hbox{ or } \max_j |Y_{ij} | \ge k} \\
 & = 2 \|f\|_{\TV}U(G(Y_{n}))[\IND_{\deg_{G}(o)\ge k  \hbox{ or } \max_{v \sim o} |\xi (o,v) | \ge k}] \\
 & = 2 \|f\|_{\TV} P_n( [k,\infty)),%\le \frac{\|f\|_{\TV}}{k} U(G(Y_{n}))[\deg_{G}(o)],
 \end{align*}
 where $P_n \in \cP ( \dZ_+)$ is the distribution under $U(G_n)$ of $\max(\deg(o),\max_{v \sim o} |\xi(o,v)| ) $.
% where we finally used Tchebychev's inequality. 
 By assumption $U(G(Y_n))$ converges toward $\nu$, in particular, $P_n$ converges weakly toward $P$ the law under $\nu$ of $\max(\deg(o),\max_{v \sim o} |\xi(o,v)| ) $. Hence $P_n$  is a tight family of probability measures and for some function $\delta(k)$ with $\lim_k \delta(k) = 0$, we have for every integer number $k$
 \begin{equation}\label{est}
 \sup_n P_n( [k,\infty)) \leq \delta(k).
\end{equation}
 Taking the supremum over the test functions $f$ yields for $k \leq k'$
 \begin{equation}
 \label{co1}d(L_{Y_{n}},L_{Y_{n,k}})\le 2 \delta(k)\AND d(L_{Y_{n,k'}},L_{Y_{n,k}})\le 2 \delta(k).  \end{equation}

Let $G \mapsto G^{(k)}$ be the map on marked graphs where all vertices with degree larger than $k$ or which are adjacent to an edge with weight larger than $k$ have been isolated from the graph and have no edges connected to them.  We extend this map to rooted marked graphs $(G,o)$ by defining $(G,o)^{(k)}$ as the connected component rooted at $o$ of $G^{k}$. For each $k$, since elements in $\cGr$ are locally finite, the map $g \to g^{(k)}$ is continuous on $\cGr$ for the local topology and
\begin{equation}\label{eq:convgk}
\lim_k g^{(k)} = g,
\end{equation}
 (indeed, for any $h \geq 1$ and $g \in \cGr$, $D_{g,h}  = \max_{v \in V((g)_h)} \deg(v) < \infty$).  For $\nu \in \cP( \cGr)$, we denote by $\nu^{(k)}$ the push-forward of $\nu$ by the map $g \to g^{(k)}$.  By construction $G(Y_{n,k}) = G(Y_n)^{(k)}$ and by continuity, $U(G(Y_{n,k}))$ converges toward $\nu^{(k)} \in \cP(\cGr)$.
Also, \eqref{eq:convgk} implies that for the local weak topology,
$$
\lim_{k \to \infty} \nu^{(k)} = \nu.
$$

This shows that $ U(G(Y_{n,k}))$ and $U(G(Y_{n}))$ are close for the local weak convergence, this guarantees that $ L_{Y_{n}}$ depends continuously on $U(G(Y_{n}))$. More precisely, from the first point applied to $Y_{n,k}$ for fixed $k$, weakly
$$
\lim_{n \to \infty} L_{Y_{n,k}} = L_{\nu^{(k)}}.
$$
 
From \eqref{co1}, for $k \leq k'$, $d(L_{\nu^{(k)}},L_{\nu^{(k')}})\le 2 \delta(k)$. Hence $(L_{\nu^{(k)}})_{k \geq 0}$ is a Cauchy sequence which converges toward some probability measure $L_{\nu}$. Finally, using  the \eqref{co1} again and $
\lim_{n \to \infty} L_{Y_{n,k}} = L_{\nu^{(k)}},
$ we get for any $k \geq 1$, 
$$\limsup_n d(L_{Y_n}, L_{\nu} ) \leq \limsup_n d(L_{Y_n}, \mu_{\nu^{(k)}} ) + d ( L_{\nu^{(k)}},L_{\nu}) \leq   4 \delta(k).$$
 Letting $k$ going to infinity, we obtain the second point.
 
% It is enough  to show that a large proportion  of the $h$ neighborhoods of $U(G(Y_{n}))$ have all vertices of degree bounded by $k$ for $k$ large enough, which we will see is guaranteed by  $\sup_{n} U(G(Y_{n}))[\deg_{G}(0)\log \deg_{G}(0))] <\infty$.  Moreover, when we can restrict ourselves to graphs with bounded degree, we still need to require that the marks are tight: this is guaranteed by the definition of $K_{M}$.
%  Once this is done, we will be able to deduce that if $U(G(Y_{n}))$ converges, $L_{Y_{n}}$ converges towards some limit $\mu_{\nu}$. We first describe this limit.

$\bullet$  We now explain the third point which ultimately comes from the fact that we can associate to $\nu \in \mathcal P(\cGr)$ unimodular a self-adjoint  affiliated operator, see \cite[Proposition 1.4]{BSV}, \cite[pp 103-107]{MR3791802}. This point will play no role in the present paper but for the sake of completeness we sketch this argument below. 
  If $G = (V,E)$ is locally finite, we can always define the operator $A_{G}$ by \eqref{eq:degAG} on finitely supported functions $\ell^0(V)\subset \ell^2(V)$.
 $A_{G}$ is symmetric since we assume that the marks are. If $G$ has uniformly bounded degrees and marks then $A_G$ is self-adjoint on $\ell^2 (V)$. Hence, we can define the spectral measure of $A_{G}$ by $\mu^{\psi}_{G}(f)=\langle \psi, f(A_{G})\psi \rangle$ for all $\psi \in \ell^2(V)$ and $f$ continuous. We set if $\nu\in  \mathcal P(\cGr)$,
 \begin{equation}\label{defmunu} L_{\nu}(f)=\E_{\nu}[L^{e_{0}}_{A_G}(f)]\,.\end{equation}
Following the same type of arguments as above (see also the proof of   \cite[Proposition 1.4]{BSV}, \cite[pp 103-107]{MR3791802}), $A_{G}$ can be more generally defined  as an affiliated operator of the  von Neumann algebra associated to the unimodular measure $\nu$. This  von Neumann algebra is introduced in \cite{MR2354165,MR2593624}. In fact, assuming without loss of generality that the vertex set $V$ is deterministic, we set $H=\ell^{2}(V)$ and  denote $\mathcal B(H)$ the set of bounded linear operators on $H$. For a given unimodular measure $\nu$, we associate the $*$-algebra of bounded operators $\mathcal M \subset L^{\infty}(\cGr, \mathcal B(H),\nu)$ which commute with the operators $\lambda_{\sigma}$ defined for all $v\in V$ by $\lambda_{\sigma}(e_{v})=e_{\sigma(u)}$ for any bijection $\sigma:V\rightarrow V$. The $*$-algebra $\mathcal M$ is endowed with the faithful normal trace
 $$\tau(B)=\mathbb E_{\nu}\langle e_{o}, B e_{o}\rangle\,.$$
The operator $A_{G}$ is uniquely defined as a self-adjoint  operator affiliated to $\mathcal M$ (see \cite[Definition 5.2.28]{AGZ}) with spectral measure defined by \eqref{defmunu}. One of the nice property of an affiliated operator is that if $A,B$ are two self-adjoint affiliated operators, and $p$ self-adjoint projection  so that $pAp=pBp$ and $\tau(p)\ge 1-\epsilon$, then the Kolmogorov-Smirnov distance of the spectral measures of $A$ and $B$ differ at most by $\epsilon$, see \cite[Lemma 5.2.34]{AGZ}. As an application, if we let 
$\pi_{k}$ be the projection onto vertices with degree and adjacent weights bounded by $k$, the distribution of $\pi_{k}A_{G}\pi_{k}$ and $A_{G}$ are at distance at most $1-\tau(\pi_{k})=\nu(\deg_{G}(o)\ge k)\le \delta(k)$, which generalizes \eqref{cont}.% Moreover, clearly $\mu_{U(G(Y_{n}))}=L_{n}$ by \eqref{defmunu}. 
   \end{proof}

We have the following variant of Lemma \ref{smooth} to networks which is a slight improvement on \cite{BordCap}. In the following, we say that a measure $\nu \in \cP(\cNr)$ is {\em sofic} if there exists a sequence of finite marked graphs $\{G_n\}$ such that $U(G_n)$ converges to $\nu$ for the local weak topology on networks.  The set of sofic measures is a closed subset of the set of unimodular measures.

%$ \dE_{\nu} \left[\sum_{v \sim o} |\xi (o,v)|^2 \IND (|\xi (o,v)| \leq 1) \right] < \infty$ and $\nu$-a.s. $\sum_{v \sim o} |\xi (o,v)|^2 < \infty$.

 \begin{lemma}\label{smoothN} Let $0 < \beta < 2$, $\eps >0$, and let $K_{\beta,\eps}$ be the set of $\nu \in \cP(\cNr)$ sofic such that $\dE_{\nu} \left[\sum_{v \sim o} |\xi (o,v)|^\beta \IND_{ |\xi (o,v)| \leq \eps}\right] \leq 1$. Then, there exists a probability measure $L_\nu \in \cP(\dR)$ such that the following holds.
  Let $Y_{n}, n\ge 0,$ be a sequence of $n\times n $ self-adjoint matrices with associated marked graphs $G(Y_n)$. Assume that $U(G(Y_n)) \in K_{\beta,\eps}$ and that $U(G(Y_n)) $ converges toward $\nu \in \cP(\cNr)$ for the local weak topology on networks.  
 \begin{enumerate}

\item The  $K_{\beta,\eps} \to \cP(\dR) $ map  $\nu \mapsto L_{\nu}$ is continuous.
 
 \item We have that $A_G$ is $\nu$-a.s. essentially self-adjoint and $L_{\nu} = \dE_{\nu}[L_{A_G}^{e_o}]$.

 \end{enumerate}
 \end{lemma}

\begin{proof}
$\bullet$ For $G  = (V,E,\xi)$ and $u \in V$, let $\mathcal E_{G}(u) = \sum_{v \sim u} |\xi(u,v)|^2$. We first claim that 
\begin{equation}\label{eq:Efini}
\mbox{$\nu$-a.s. } \mathcal E_{G}(u) < \infty.
\end{equation} Indeed, since $(\sum_i x_i )^r \leq \sum_i x_i^r$ for $x_i \geq 0$ and $0 \leq r \leq 1$, we have
$$
(\mathcal E_{G}(u))^{\beta/2} \leq \sum_{v \sim o} |\xi (o,v)|^\beta \IND_{ |\xi (o,v)| \leq \eps} + \sum_{v \sim o} |\xi (o,v)|^\beta \IND_{ |\xi (o,v)| > \eps}.
$$
The first term has a finite expectation, it is thus a.s.\ finite. The second term is a.s.\ a finite sum since $\nu \in \cP(\cNr)$, in particular, the second term is also a.s.\ finite. This implies \eqref{eq:Efini}

The definition of $L_{\nu}$ is given in \cite[Lemma 3.12]{BordCap}. It is as follows. The idea is again to truncate the network to have a bounded operator. For $0 < \theta < 1$, we consider the network $G^{(\theta)}$ where we have removed all edges $\{u,v\}$ such that $\xi(u,v) \ne 0$ and $|\xi(u,v)| \leq \theta$, and removed all edges adjacent to a vertex $u$ such that {{$\mathcal E_G(u)  \geq \theta^{-2}$.}} The push-forward of $\nu$ by the map $g \mapsto g^{(\theta)}$ on $\cNr$ is denoted by $\nu^{(\theta)}$. By construction, the operator $A_{G^{(\theta)}}$ is bounded with operator norm at most $\theta^{-2}$. Hence, $A_{G^{(\theta)}}$ is self-adjoint on $\ell^2(V)$. We thus define the probability measure $L_{\nu^{(\theta)}} = \dE_{\nu} [\mu^{e_o}_{A_{G^{(\theta)}}}] =  \dE_{\nu^{(\theta)}} [\mu^{e_o}_{A_{G}}]$. Then \cite[Lemma 3.12]{BordCap} asserts that $L_{\nu^{(\theta)}}$ converges weakly to a measure $L_{\nu}$ as $\theta \to 0$ under the stronger condition  $\dE_{\nu} \left[\sum_{v \sim o} |\xi (o,v)|^\beta \right] \leq M$. We may obtain the same claim under the conditions of the present lemma by the following minor modification. The proof of \cite[Lemma 3.12]{BordCap} implies that for $0 < \theta' < \theta$
\begin{align*}
d(L_{\nu^{(\theta)}}, L_{\nu^{(\theta')}} )  & \leq  4 \dP_{\nu} ( \mathcal E_G(o) \geq \theta^{-2} ) + 2 \left(\dE_{\nu} \sum_{v \sim o} |\xi(o,v)|^2 \IND ( |\xi(o,v)|\leq \theta  ) \right)^{1/2} \\
& \leq 4 \dP_{\nu} ( \mathcal E_G(o) \geq \theta^{-2} ) + 2  \theta^{1 - \beta/2} \left(\dE_{\nu} \sum_{v \sim o} |\xi(o,v)|^\beta \IND ( |\xi(o,v)|\leq \theta  ) \right)^{1/2} 
\end{align*}
For $0 < \theta < \eps$, the later is upper bounded by  {{$4 \dP_{\nu} ( \mathcal E_G(o) \geq \theta^{-2} ) + 2  \theta^{1 - \beta/2}$. The second  term  goes to zero since $\theta$ does as $\beta<2$.}}  Also, by  \eqref{eq:Efini}, we have $4 \dP_{\nu} ( \mathcal E_G(o) \geq \theta^{-2} ) =  \delta (\theta)$ with $\delta(\theta) \to 0$ as $\theta \to 0$. It proves that the sequence $(L_{\nu^{(\theta)}})$ is Cauchy and converges toward  some probability measure $L_{\nu}$ as $\theta \to 0$. The first point of the lemma is the content of \cite[Lemma 3.13]{BordCap} up to the minor change on the upper bound for $\dP_\nu (\mathcal E_G(o) \geq \theta^{-2})$ explained above and the fact that $\limsup_{n} \dP_{U(G(Y_n))} (\mathcal E_G(o) \geq \theta^{-2}) \leq \dP_\nu (\mathcal E_G(o) \geq \theta^{-2})$ for  all $\theta$.

$\bullet$  Then, the second point follows from the same argument than the last point of the proof of Lemma \ref{smooth}.
\end{proof}

\begin{remark}
The assumption $\dE_{\nu} \left[\sum_{v \sim o} |\xi (o,v)|^\beta  \IND( |\xi(o,v)| \leq 1/M)\right] \leq 1$ is not optimal for the conclusion of Lemma \ref{smoothN} to hold. For example, it is not hard to show that the spectral measure $L_{\nu}$ can be defined as soon as  $\dE_{\nu} \sum_{v \sim o} |\xi(o,v)|^2 \IND (|\xi(o,v)| \leq \theta ) < \infty$ for $\theta$ small enough. The continuity of the spectral measure seems to require however some uniform integrability.
\end{remark}

\subsection{Proofs of Theorem \ref{thA} and Theorem \ref{thB}}\label{secproofAB}

We start with the proof of Theorem \ref{thA}. As an immediate consequence of Theorem \ref{th:SIGMAER},  Lemma \ref{smooth} and the contraction principle \cite{DZ},  we deduce that if $\Sigma^{\tiny {\rm ER}}_{\gamma,d}$ is the rate function described in Theorem \ref{th:SIGMAER}
  $L_{Y_n}$ satisfies a LDP with  speed $n$ and good rate function on $\cP(\dR)$,
  $$J_{\gamma,d}(p)=\inf\{ \Sigma^{\tiny {\rm ER}}_{\gamma,d}(\mu): L_{\mu}=p \}.$$

 The proof of Theorem \ref{thB} is identical. It suffices to use Theorem \ref{th:SIGMADR} instead and replace  $\Sigma^{\tiny {\rm ER}}_{\gamma,d}$ by $\Sigma_{\gamma,\pi}$. \qed

\subsection{Proof of Theorem \ref{LDPHT}}

Set $Y = Y_n$ be as in Theorem \ref{LDPHT}. We start by proving that $L_{Y_n}$  belongs to the set $K_{\beta,\eps}$ defined in Lemma \ref{smoothN} with exponential probability.

\begin{lemma}\label{app} Assume \eqref{eq:tailalpha2} and fix $\beta > \alpha$. Then, for $\eps ,\delta>0$ we have
\begin{equation}\label{res1}
\mathbb P\left( \frac{1}{n} \sum_{i\leq j}  |Y_{ij}|^\beta \IND_{|Y_{ij}|\le \eps   } \ge \delta\right)\le e^{-n f(\eps ,\delta)},\end{equation}
%\begin{equation}\label{res2}\mathbb P\left( \frac{1}{n} \sum_{i\leq j}  \IND_{|y_{ij}|\ge \epsilon^{-1} n^{\frac{1}{\alpha}}} \ge \delta\right)\le e^{-n f(\epsilon,\delta)},\end{equation}
with $f(\eps,\delta)$ going to $+\infty$ when $\delta \eps^{\alpha-\beta}$ does.  As a consequence, for any $\beta<\alpha$, for any $T >0$, there exists $\eps >0$ such that 
$$
\limsup_{n \to \infty} \dP ( U(G(Y_n)) \notin K_{\beta,\eps} )  \leq -T.
$$
%As a consequence, for any $\delta\in (0,1]$,
%$$\limsup_{n\rightarrow\infty}\frac{1}{n}\log \mathbb P\left( d(L_{n}^{X^{(\epsilon)}},L_{n}^{X})\ge \delta\right)\le -f(\epsilon,\delta^{2})\,.$$
\end{lemma}

\begin{proof}
We first check that \eqref{res1} implies the final statement of the lemma. By definition, we have
$$
\dE_{U(G(Y_n))} \sum_{v \sim o} |\xi(o,v)|^{\beta} \IND_{ |\xi(o,v) |\leq \veps} = \frac{1}{n} \sum_{i, j}  |Y_{ij}|^\beta \IND_{|Y_{ij}|\le \eps } \leq  \frac{2}{n} \sum_{i \leq j}  |Y_{ij}|^\beta \IND_{|Y_{ij}|\le \eps }.
$$
Hence, from \eqref{res1}, we get, for $0 < \veps \leq 1/2$,
$$
\dP (  U(G(Y_n)) \notin K_{\beta,\eps} )  \leq \mathbb P\left( \frac{1}{n} \sum_{i\leq j}  |Y_{ij}|^\beta \IND_{|Y_{ij}|\le \eps   } \ge 1\right)
\le e^{- n f (\veps,1/2)} 
$$
with $ f (\veps,1/2)$ going to $+\infty$ if $\eps^{-\beta+\alpha}$ goes to infinity, namely $\eps$ goes to zero when $\beta>\alpha$. 
The conclusion follows from the properties of $f$.
We now prove \eqref{res1}.
By Chernoff bound, for every $\lambda>0$
\begin{eqnarray}\label{eq:chernoff}
\mathbb P\left( \frac{1}{n}\sum_{i\leq j} |Y_{ij}|^{\beta} \IND_{|Y_{ij}| \leq \eps}\ge \delta\right)&\le &e^{-\lambda \delta n}\prod_{i \leq j}\mathbb E \left[ e^{\lambda   |Y_{ij}|^{\beta} \IND_{|Y_{ij}|\leq \eps }}\right].
\end{eqnarray}
To estimate the above expectation, recall the integration by parts formula, for $\mu$ probability measure on $\mathbb R$ and $f$ in $C^1$, $a < b$,  
$$
\int_a^b f(t) d \mu (t) = f(a) \mu([a,\infty)) - f(b) \mu([b,\infty)) + \int_a ^b f'(t) \mu([t,\infty)) dt .
$$
Using this equality with $\mu$ the law of $|Y_{ij}|^{\beta}$, $f(x)= ( e^{\lambda  x } - 1) $, $a=0$ and $b= \eps^{\beta} $,
we see that
\begin{align*} \mathbb E  \left[\IND_{|Y_{ij}|\leq  \eps }   ( e^{\lambda |Y_{ij}|^{\beta} } - 1) \right] & =-(e^{\lambda \eps^{\beta}}-1)\mu([ \eps^{\beta} ,+\infty))
+ \lambda   \int_{0}^{\eps^{\beta}}e^{\lambda  x } \mu([x,\infty)) dx \\
& \leq c_{0}\lambda  n^{-\beta/\alpha} \int_{0}^{\eps^{\beta}n^{\beta/\alpha}}e^{\lambda  n^{-\beta/\alpha} x } x^{-\alpha/\beta} dx,
\end{align*}
where at the second line, we have used \eqref{eq:tailalpha2}. We find 
\begin{eqnarray*}
\mathbb E \left[ e^{\lambda  |Y_{ij}|^{\beta} \IND_{|Y_{ij}|\leq  \eps }}\right] & =  & 1 + \mathbb E  \left[\IND_{|Y_{ij}|\leq  \eps  }   ( e^{\lambda  |Y_{ij}|^{\beta} } - 1) \right] \\
& \leq & 1 +  c_0 \lambda n^{-\beta/\alpha} \int_{0}^{\eps^\beta n^{\beta/\alpha}} t^{-\alpha/\beta} e^{\lambda n^{-\beta/\alpha} t} dt \\
& = & 1 + c_0 \lambda^{\alpha/\beta} n^{-1} \int_0^{\eps^\beta \lambda} x^{-\alpha /\beta} e^{-x} dx.
\end{eqnarray*}
Hence, since $\beta > \alpha$, for $\lambda \leq \eps^{-\beta}$, we find for some new constant $c_1>0$, 
$$
\mathbb E \left[ e^{\lambda   |Y_{ij}|^{\beta} \IND_{|Y_{ij}|\leq \eps }}\right]  \leq 1 + c_1 \lambda \eps^{\beta - \alpha} n ^{-1} \leq e^{c_1 \lambda \eps^{\beta - \alpha} n^{-1} }.
$$
From \eqref{eq:chernoff}, we get 
$$
\mathbb P\left( \frac{1}{n}\sum_{i\leq j} |Y_{ij}|^\beta \IND_{|Y_{ij}| \leq \eps }\ge \delta\right) \leq e^{-\lambda\delta n  + n c_1  \lambda \eps^{\beta - \alpha}}.
$$
This implies \eqref{res1} by taking $\lambda = \eps^{-\beta}$ with $f(\eps,\delta) = \delta \eps^{-\beta} - c_1 \eps^{-\alpha}$. The result follows. \end{proof}

\paragraph{Proof of Theorem \ref{LDPHT}}
We fix $\beta$ such that $\alpha < \beta < 2$. For $\mu \in \cP(\cNr) \backslash \bigcup_{\veps >0} K_{\beta,\eps}$, we set $L_\mu = \delta_0$. We have that $\mu\rightarrow L_{\mu}$ is continuous on $K_{\beta,\eps} \subset \cP(\cNr)$ by Lemma \ref{smoothN} and that by Lemma \ref{app}, for any $T > 0$, for some $\eps > 0$, $\limsup_{n \to \infty} \dP ( U(G(Y_n)) \notin K_{\beta,\eps} )  \leq -T$. The theorem is thus an application of Theorem \ref{th:SIGMA1N} and the contraction principle in its extended version, see  \cite[Theorem 4.2.23]{De-Ze}. We find that $L_{Y_n}$ satisfies an LDP with speed $n$ and good rate function,
$$
J_{\Lambda} (p) =  \inf\{ \Sigma_{\Lambda}(\mu): L_{\mu}=p \}.
$$
 To be complete, we note that in Theorem \ref{th:SIGMA1N}, the diagonal values of the networks were set to $0$. It is however easy to check that the diagonal entries of $Y_n$ play asymptotically no role on $L_{Y_n}$. Let us check this claim. We denote by $\tilde Y_n$ the matrix $Y_n$ where all diagonal entries have been set to $0$ and $Y_{n}^{\epsilon}$ those were  all diagonal entries with modulus greater than $\veps$ have been set to $0$. For $\veps >0$, we set 
$$
S = \frac{1}{n } \sum_{i}  \IND_{|Y_{ii}|\ge \eps  }.
$$
From \eqref{cont}, we have  $d(L_{Y_n^{\veps}},L_{Y_n}) \leq S$. Moreover, by Hoffman-Wielandt inequality \cite{AGZ}[Lemma 2.1.19],   $d(L_{Y_n^{\veps}},L_{\tilde Y_n}) \leq \veps$. Thus,   $d(L_{\tilde Y_n},L_{Y_n}) \leq S+\veps$. 
We also observe from \eqref{eq:tailalpha2} that 
$$
\mathbb E \left[  S \right]  =   \frac{1}{n } \sum_{i}  \mathbb P \left( |Y_{ii}|\ge \epsilon  \right)  \leq \frac{c_0 \eps ^{\alpha}}{ n},
$$
and 
$$
\mathrm{Var} (S) = \mathbb {E} \left[ ( S - \mathbb E [S] )^2 \right]  = \frac{1}{n^2} \sum_{i }  \mathrm{Var} (\IND_{|Y_{ii}|\ge \eps }) \leq \frac{c_0 \eps ^{\alpha}}{n^2}.
$$
Therefore, Bennett's inequality implies that 
$$\mathbb P\left(S \ge \delta +  \mathbb E \left[  S \right] \right)\le \exp \left(-  c_0 \eps^{\alpha}  h ( \frac{ n \delta}{c_0 \eps^{\alpha}} ) \right),$$
with $h(t) = (1+t) \log (1+t) - t$. We take $\eps^{\alpha} = \delta = ( \log n)^{-1/2}$. As $t \to \infty$, $h(t) \sim t \log t$, we deduce that, for some $c >0$, all $n \geq 1$,
$$
\mathbb P\left(S \ge 1 /  \sqrt{ c \log n} \right)\le  \exp ( - c n \sqrt{\log n}).
$$
Since $d(L_{\tilde Y_n},L_{Y_n}) \leq \eps + S$, we obtain, for some new $c >0$,
$$
\mathbb P\left(d(L_{\tilde Y_n},L_{Y_n})  \ge 1 /  \sqrt{  c \log n} \right)\le  \exp ( - c n \sqrt{\log n}).
$$
It follows that  $L_{\tilde Y_n}$ is an exponentially good approximation of $L_{ Y_n}$ \cite{DZ}[Definition 4.2.10] and therefore
 the LDP at rate $n$ of $L_{\tilde Y_n}$ implies the LDP at rate $n$ of $L_{ Y_n}$ with the same rate function \cite{DZ}[Theorem 4.2.13].
\qed

\section{Microstates entropy for the traffic distribution of marked graphs}\label{sec:traffic}

 \subsection{Traffic distributions : proof of Theorem \ref{LDPtraffic0}}
 \label{subsec:TF1}
 We first outline  the  proof of Theorem \ref{LDPtraffic0}, which is a direct consequence of Theorem \ref{th:SIGMADR} and the contraction principle.  We endow the space $\TRAF$ of traffic distributions with its weak topology. We consider the independent  matrices $\mbf Y=(Z_{j})_{j\in J}$  with $Z_{j}(k,\ell)=A_{j}(k,\ell) X_{j}(k,\ell)$ with $A_{j}$ the adjacency matrix of the random graph $G^n_{j}$ with vertices $V_{n}=\{1,\ldots,n\}$ uniformly sampled so that the degree of $v$ is $D_{n,j}(v)$. We consider the marked graph $G_{n}^{J}$ with vertices $V_{n}$ and with up to $|J|$ edges between $i$ and $j$, namely one edge with mark $X_{j}(k,\ell)$ as soon as $A_{j}(k,\ell)\neq 0$, for $j\in J$. It is uniformly sampled  on colored random graph
  so that the degree of edges of color $j$ of $v$ is $D_{n,j}(v)$  for each $j\in J$ and each of its edge  of color $j$ carry a mark $X_{j}(k,\ell)$.  The graph $G_{n}^{J}$ can also be seen as a particular example of randomly marked graphs considered in Section 
  \ref{sec:markedg} with $\mathcal B=J$ and distribution of marks $\gamma$. Hence, Theorem \ref{th:SIGMADR} states that $U(G_{n}^{J})$ satisfies a LDP with rate $n$ and good rate function $\Sigma_{\gamma,\pi}$. On the other hand, for every test graph $H$,
  $$\tau_{\mbf Y}[H]=\int f_{H} dU(G_{n}^{J}
 )$$
 where $f_{H}(G,o)=\sum_{  \phi:V\to \{1,\ldots,n\},\phi(o)=1}    \prod_{e = (v,w)\in E} Y_{\ell(e)}^{\varepsilon(e)}\big(\phi(w),\phi(v) \big)$. Thus, for any $\mu \in\mathcal P(\cGr)$ and any test graph $H$, we can set $\tau_{\mu}[H]:=\int f_{H} d\mu$. 
 Clearly $f_{H}$ is bounded continuous, uniformly on the set $\mathcal B_{\theta}^{R} $ of marked graphs with  degree bounded by $\theta$ and marks bounded by some $K$. As a consequence,  $\mu\mapsto \tau_{\mu}$  is continuous, from $\mathcal P(\cGr\cap \mathcal B_{\theta}^{R} )$ into the space of traffics equipped with its weak topology. 
Hence,  
 Theorem \ref{LDPtraffic0} is a direct consequence of Theorem \ref{th:SIGMADR} and the contraction principle: $\tau_{\mbf Y}$ satisfies a LDP with speed $n$ and good rate function
 $$\chi_{\pi,\gamma}(\tau)=\inf\{ \Sigma_{\gamma,\pi} (\mu), \tau_{\mu}=\tau\}\,.$$

 \subsection{Traffic distributions : definition of $\tilde \chi_{\Lambda} (\tau)$}
\label{subsec:tildechi}

In this subsection, we explain how to define properly $\tilde \chi_\Lambda (\tau)$ introduced in \eqref{eq:deftildechi}. Let $\theta_0$ be such that $\Lambda ( B_{\dC}(0,\theta_0) ) > 0$ where $B_{\dC}(0,r)$ is the ball in $\dC$ of radius $r$. As explained in the  introduction, it is sufficient to prove the following identity, for any fixed $\theta > \theta_0 $ and $\tau \in \TRAF$,
\begin{equation}
\label{eq:prechitilde}
\tilde \chi^-_{\Lambda,\theta}(\tau)  = \tilde \chi^+_{\Lambda,\theta}(\tau)  \end{equation}
where we have set, 
$$
\tilde \chi^-_{\Lambda,\theta}(\tau) = \lim_{\epsilon\downarrow 0}\liminf_{n\rightarrow\infty}\frac{1}{n}\ln \mathbb P\left( \mbf Y\in  B_{n}(\theta) ^{ J } \; ;  d( \tau_{\mbf Y}, \tau) \le\epsilon\right)
$$
and $\tilde \chi^+_{\Lambda,\theta}(\tau) $ is defined similarly with a $\limsup$. When \eqref{eq:prechitilde} holds, we define 
\begin{equation}\label{defchit} \tilde \chi_{\Lambda,\theta}(\tau)  := \tilde \chi^+_{\Lambda,\theta}(\tau) \, \quad \hbox{ and } \quad \tilde \chi_{\Lambda }(\tau)  = \lim_{\theta \to \infty} \tilde \chi_{\Lambda,\theta}(\tau)  \\.\end{equation}
Since $Y_j$ is self-adjoint for all $j \in J$, we may replace in the above expression, $B_n(\theta)$ by $\tilde B_n(\theta)$, the subset of self-adjoint matrices of $B_n(\theta)$. If $\mbf Y = (Y_j)_{j \in J} \in M^J_n(\dC)$ is a collection of self-adjoint matrices,  we denote by $G^J_n$ the graph defined in Subsection \ref{subsec:TF1} on the vertex set $V_n = \{1,\ldots ,n\}$ and with edge marks in $J$ associated to the non-zero entries of $Y_j, j \in J$. We also set $D_n^J = (D_{n} (v ))_{v \in V_n}$ the degree sequence of $G^J_n$ where $D_{n}(v)  = (D_{n,j}(v))_j \in \dZ_+^J$ counts the number of edges adjacent to $v$ with a  $j$-th mark, $j \in J$. The empirical degree distribution of $G_n^J$ is then 
$$
L_{D^J_n}  = \frac{1} n \sum_{v \in V_n} \delta_{D_{n} (v )} \in \cP( \dZ^J_+). 
$$
We denote by $\cP^J_{n,\theta} \subset \cP( \dZ^J_+)$ the set of empirical probability measures of the form
$$
\frac 1 n  \sum_{v=1}^n \delta_{d(v)},
$$
with  $ 0 \leq d_j(v) \leq \theta$  for all $j \in J$. We have $| \cP^J_{n,\theta} | \leq n^{ \theta^J}$ and therefore :  
\begin{equation}\label{eq:partition}
\lim_{n \to \infty} \frac 1 n \ln | \cP^J_{n,\theta} | = 0.
\end{equation}
 We may decompose $\tilde B_n(\theta)^J$ over the values of the empirical degree distributions:
$$
\tilde B_n(\theta)^J = \bigsqcup_{\mbf p \in \cP_{n,\theta}^J } \{ \mbf Y \in \tilde M_n(\dC) :  L_{D_n^J} = \mbf p,  \max_{j,u,v} |(Y_j)_{uv}| \leq \theta \} = \bigsqcup_{\mbf p \in \cP^J_{n,\theta}} \tilde B^J_n(\mbf p,\theta).
$$
From \eqref{eq:partition}, we deduce that 
 \begin{equation}\label{defq}
 \frac{1}{n}\ln \mathbb P\left( \mbf Y\in  \tilde B_{n}(\theta) ^{ J } \; ;  d( \tau_{\mbf Y}, \tau) \le\epsilon\right) = \max_{\mbf p \in \cP_{n,\theta}^J}  \frac{1}{n}\ln \mathbb P\left( \mbf Y\in  \tilde B^J_{n}(\mbf p,\theta) \; ;  d( \tau_{\mbf Y}, \tau) \le\epsilon\right) + o(1).
 \end{equation}
Now, for every fixed $\mbf p \in \cP_{n,\theta}^J$, we write that $\frac{1}{n}\ln \mathbb P\left( \mbf Y\in  \tilde B^J_{n}(\mbf p,\theta) \; ;  d( \tau_{\mbf Y}, \tau) \le\epsilon\right)$ equals
\begin{equation}\label{defq2}
\frac{1}{n}\ln \mathbb P\left( \mbf Y\in  \tilde B^J_{n}(\mbf p,\theta) \right) + \frac 1 n \ln \mathbb P \left(  d( \tau_{\mbf Y}, \tau) \le\epsilon | \mbf Y\in  \tilde B^J_{n}(\mbf p,\theta) \right).
\end{equation}
 Consider a sequence $\mbf p = \mbf p_n \in \cP_{n,\theta}^J$ which converges toward a probability measure $\pi \in \cP(\dZ_+^J)$. Since $\cP^J_{n,\theta}$ is relatively compact, the proof of \eqref{eq:prechitilde} follows from the following two claims: 
\begin{eqnarray}
& \lim_{n \to \infty} \frac{1}{n}\ln \mathbb P\left( \mbf Y\in  \tilde B^J_{n}(\mbf p,\theta) \right) = L_{\pi,\Lambda} (\theta). \label{eq:twocl1}\\
& \hbox{$\tau_{\mbf Y}$ given $\{ \mbf Y \in  \tilde B^J_{n}(\mbf p,\theta)\}$ satisfies a LDP with rate $n$ and rate function  $\tilde \chi_{\pi,\Lambda_{\theta}}$}\label{eq:twocl2}
\end{eqnarray}
where $\Lambda_{\theta}(.)=\Lambda(. \cap B_{\dC} (0,\theta) )/ \Lambda( B_{\dC} (0,\theta) )$.
It is indeed easy to deduce from \eqref{defq},\eqref{defq2}, \eqref{eq:twocl1} and  \eqref{eq:twocl2}, that \eqref{defchit} is satisfied and
$$\tilde\chi_{\Lambda,\theta}(\tau)=\inf_{\pi\in\mathcal P(\mathbb Z^J_{+})}\{\tilde \chi_{\pi,\Lambda_{\theta}}(\tau) -L_{\pi,\Lambda}(\theta) \}.$$
Let us check the two claims  \eqref{eq:twocl1} and  \eqref{eq:twocl2}. We start with \eqref{eq:twocl1}.  Let  $q_n =  \dP ( (Y_j)_{12}  > 0 )$,  and $r_n = \dP (  |(Y_j)_{12}|  \leq \theta | (Y_j)_{12} \ne  0 )$. From \eqref{eq:gamman2N}-\eqref{eq:Lambdabddt}, we have
$$
\lim_{n \to \infty} n q_n = \lambda = \Lambda ( \dC  \backslash \{0 \} ) >0 \quad \hbox{ and } \quad \lim_{n \to \infty} r_n = r(\theta)=  \frac{ \Lambda (B_{\dC} (0,\theta) \backslash \{0 \} ) }{ \Lambda ( \dC \backslash \{0 \} )} > 0.
$$
Note that the matrices $(\IND_{  (Y_j)_{uv} \ne 0})_{u,v}$, $j \in J,$ are the adjacency matrices of independent Erd\H{o}s-R\'enyi random graph with edge probability $q_n$. Note also that $\{\mbf Y\in  \tilde B^J_{n}(\mbf p,\theta)  \}$ is equivalent to all edges of $G_n^J$ have a corresponding value in $\mbf Y$ of absolute value at most $\theta$ and $G_n^J \in \cG_n(\mbf p)$ where $\cG_n(\mbf p)$  is the set of $J$-marked graphs on $V_n$ whose empirical degree sequence is $\mbf p$ and whose restriction $G_n^j$ to color $j$ is a simple graph for all $j \in J$.  It follows that 
$$
\dP \left( \mbf Y\in  \tilde B^J_{n}(\mbf p,\theta) \right) =  | \cG_n(\mbf p)| (q_n)^{n d_n /2} (1 - q_n)^{n (n -1) / 2 - n d_n /2} r_n^{n d_n /2}, 
$$
where, with   $D =(D_1,\ldots,D_j)$ having distribution $\mbf p$ under $\dE_{\mbf p}$, $n d_n = n \dE_{\mbf p} \sum_j D_j$ is twice the total number of edges in any $G_n^J\in \cG_n(\mbf p)$. We have $d_n \to d = \dE_{\pi} \sum_j D_j$ where $D =(D_1,\ldots,D_j)$ has distribution $\pi$ under $\dE_{\pi}$. We deduce that 
$$
\dP ( \mbf Y\in  \tilde B^J_{n}(\mbf p,\theta) )  =  | \cG_n(\mbf p)| e^{ - \frac{nd_n}{2} \ln n} e ^{ n  a (\theta)( 1+ o(1))}, 
$$
with $a (\theta)= (d/2)\ln \lambda - (\lambda /2) + (d/2)\ln r (\theta)$. 
On the other hand, the asymptotic equivalent of $| \cG_n(\mbf p)| $ is well-known. It follows for example from \cite[Corollary 4.6, (5.7)-(5.8)]{BoCa15} that, 
$$
| \cG_n(\mbf p)|  \sim e^{  \frac{nd_n}{2} \ln n} e^{ n b(\pi)  (1+ o(1))}
$$
with $b (\pi)= H(\pi) + (1/2) \sum_j d(j) \ln (d(j)/e) - \sum_j \dE_{\pi} \ln (D_j !)  $ and $d(j)  =  \dE_{\pi} D_j$.
This concludes the proof of \eqref{eq:twocl1} with $L_{\pi,\Lambda}(\theta)= b(\pi)+a(\theta)$.

We now prove the claim \eqref{eq:twocl2}.  The law of $G_n^J$ given $\{ \mbf Y \in  \tilde B^J_{n}(\mbf p,\theta)\}$ is, by exchangeability of vertices, a random marked graph as in Theorem \ref{th:SIGMADR} with $\Lambda$ replaced by $\Lambda_{\theta}$.  Therefore, we see that $U(G_{n}^{J})$  given $\{ \mbf Y \in  \tilde B^J_{n}(\mbf p,\theta)\}$  satisfies a 
LDP with good rate function $\Sigma_{\Lambda_{\theta},\pi}$.
We may thus repeat the proof of Theorem \ref{LDPtraffic0} as in Section \ref{subsec:TF1} to deduce that 
 the law of  $\tau_{\mbf Y}$ conditional to   $\{ \mbf Y \in  \tilde B^J_{n}(\mbf p,\theta)\}$  satisfies 
 a LDP for with a good rate function $\tilde \chi_{\pi,\Lambda_{\theta}}$ obtained by the contraction principle so that $\tilde \chi_{\pi,\Lambda_{\theta}}(\tau)=\inf\{ \Sigma_{\Lambda_{\theta},\pi}(\mu):\tau_{\mu}=\tau\}$. 
 
\subsection{Rooted traffic distributions}
\label{subsec:rootedtraff}

Recall that a test graph $H=(V_H,E_H,\ell, \varepsilon)$ is a finite graph $(V_H,E_H)$ and the maps $\ell:E_H \to J, \varepsilon:E_H\to \{1,*\}$, interpreted as edge labeling by indeterminates $y_j, y_j^*,j\in J$. As our approach is local, we define a \emph{rooted test graph} as a couple $(H,1)$ where $H$ is a test graph and $1\in V$ is a vertex called the root. When the context is clear, we write in short $H$ to refer to the couple $(H,1)$. We denote by $\cH^\bullet \langle J\rangle$ the set of test graphs up to root preserving isomorphisms and by ${\dC\cH}^\bullet\langle J \rangle$ the vector space they generate. An (algebraic) \emph{rooted-traffic distribution} is therefore a linear form on ${\dC\cH}^\bullet\langle J \rangle^*$.

We also extend the definition of Section \ref{traffic:int} to the locally finite marked graphs of Section \ref{Sec:MarkG}. Let $J$ be a fixed finite label set. Let $G = (V,E,\xi)$ be a marked graph on the mark set $\cZ = \dC^J$ equipped with the involution $(z_j)_{j\in J} ^* = (\bar z_j)_{j\in J}$. For $j \in J$, we set  $E_j\subset E$ the set of edges $e\in E$ such that the $j$-th coordinate of the mark is $\xi_j(e)\neq0$, and denote by $G_j = (V,E_j,\xi_j)$ the marked graph on the mark set $\cZ_j=\dC$. For $j \in J$, we consider the operator $A_{G_j}$ defined in \eqref{eq:degAG} under the assumption \eqref{eq:hypAG}. We strengthen assumption \eqref{eq:hypAG} by assuming that $G$ is locally finite.

\begin{definition}[Rooted traffic distribution for marked graphs]\label{Def:TrafficDistribution}Let $(G,o) \in  \cGr(\dC^J)$ be a marked graph, where $G= (V,E,\xi)$ and $G_j$ are as above. The (canonical) rooted-traffic distribution $\tau_{(G,o)}$ of $(G,o)$ is defined by:  for any $ H=(V_H,E_H,\ell, \varepsilon) \in {\dC\cH}^\bullet\langle J \rangle,$
	\begin{equation}\label{def:tauG}
		\tau_{(G,o) }[  H ] =  \sum_{\substack{ \phi:V_H\to V \\ \phi(1)=i}}    \prod_{e=(v,w)\in E_H} \xi^{\varepsilon(e)}_{\ell(e)}\big(\phi(w), \phi(v)\big) \in \dC
	\end{equation}
where the sum is over all maps $\phi$ from $V_H$ to $V$ such that $\phi(1)=o$. 
\end{definition}

To connect this definition with the traffic distribution of matrices, note that  if $V = \{1,\ldots,n\}$ and $\mbf A_G = (A_{G_j})_{j \in J}$ the adjacency operator of $\mbf G$, then $\frac 1 n \sum_i \tau_{\mbf G, i} = \tau_{\mbf A_{G}}$ with $\tau_{\mbf A_{G}}$ as in Definition \ref{Def:TrafficDistribution}. Note also that since $G$ is locally finite, the sum defining $\tau_{(G,o) }[  H ]$ can be reduced to a sum over a finite number of maps $\phi$, so it is well defined. Also, this definition depends only on the isomorphism class of the rooted graph $(\mbf G, o)$, so it well defines a function on the set of locally finite rooted marked graphs with marks in $\cZ = \dC^J$
	\begin{equation}\label{Eq:EquivalenceMap}
		\begin{array}{cccc}
			\hat \tau: & \cGr(\dC^J) & \to & \TRAFB\\
			& (G ,o) & \mapsto & \tau_{G ,o}.
		\end{array}
	\end{equation}

The definitions imply the following observation.

\begin{lemma} \label{continuous}Endowing $\cGr(\dC^J)$ with the weak local topology and $\TRAFB$ with the topology of pointwise convergence, the map $\hat \tau$ is continuous. 
\end{lemma}
\begin{proof}
Since $H$ is finite, all vertices of $H$ are within graph distance at most say $h \geq 1$ from $1 \in V_H$. In particular, $ \tau_{G,o}(H)$ is a function on $G$ restricted to $B_G(h)$, the ball of radius $h$ around the root $o$ of $G$.  The continuity follows immediately from the definition of the local topology on $\cGr$. 
\end{proof}

Lemma \ref{continuous} shows that the rooted-traffic distributions of a locally finite graph contains \emph{a certain information} about the graph. Actually, it contains \emph{all} its information, even in the context where of random graphs,  and this is the main argument to deduce a LDP of the traffic distribution from a LDP of the local law of the graph by the contraction principle.

More precisely, we endow the space ${\dC\cH}^\bullet\langle J \rangle^*$ with the  topology of pointwise convergence and denote by $\cP(\TRAFB)$ the set of probability measures on the space of rooted-traffic distributions. If $\mu$ is the law in  $\cP(\cGr(\dC^J))$, we define $\tau_{\mu}\in \cP(\TRAFB)$ as the push-forward of $\mu$ by $\hat \tau$, that is the law of the random rooted-traffic distribution $\tau_{(G ,o)}$ where $(G,o)$ has law $\mu$. Therefore $\tau_\mu$ is characterized by the collection of complex numbers given, for  every $k$-tuple of test graphs  $\mbf H\in \mathcal  H\langle J\rangle^{k}$ and  bounded continuous  function $f$ from $\mathbb C^{k} \to \mathbb C$, by
	$$\tau_{\mu}[\mbf H ]( f)= \E\Big[ f(\tau_{(G,o)}[\mbf H] )\Big].$$
In particular, we observe that if $\mbf Y= (Y_j)_{j \in J}$ is a collection of matrices in $M_n(\dC)$ and $G(\mbf Y)$ is the associated marked graph on $\dC^J$, we have 
\begin{equation}\label{eq:trafGA}
\Utau_{\mbf Y} = \tau_{U(G(\mbf Y))},
\end{equation}
where, as usual, $U(G)$ is as in \eqref{eq:defU}.

\begin{lemma}\label{Lem:Homomorphismfin} The map $\tilde \tau:  \cP(\cGr(\dC^J)) \mapsto   \cP(\TRAFB)$ induced by $\hat \tau$ is injective. Moreover, denoting by $ \cP(\mathrm{Traf}_\cG^\bullet\langle J \rangle)$ its image set, the inverse map $\tilde \tau^{-1} : \cP(\mathrm{Traf}_{\cG}^\bullet\langle J \rangle ) \to \mu\in \cP(\cGr(\dC^J)) $ is continuous.
\end{lemma}

This implies that the weak topology of traffics and the Benjamini Schramm topology are equivalent. The lemma is proved in the next section, we present the first argument.

As in previous sections, we use the approximation of rooted-graphs by graphs with bounded marks and degree in a neighborhood of the root. For any $k\geq 1$, we denote by $\cP_{k}( \cGr(\dC^J))$ the set of laws of weighted graphs with degree uniformly bounded by $k$ and marks uniformly bounded by $k$, and we set $\cP_{\infty}( \cGr(\dC^J))$ the union of all $\cP_{k}( \cGr(\dC^J))$'s. In the proof of Lemma \ref{smooth}, we have seen that  $\cP_{\infty}( \cGr(\dC^J))$ is dense in $\cP(\cGr(\dC^J))$. 

Let now $\mu$ in $\cP_{\infty}( \cGr(\dC^J))$ be the rooted-traffic distribution of a random graph $(G,o)$. Then for any tuple of rooted test graphs $\mbf H=(H_1 , \dots , H_n)$, the values of $\tau_{G,o}(H_1), \dots , \tau_{G,o}(H_n)$ belong almost surely to a uniformly bounded subset of $\dC^n$ where any continuous bounded function can be approximated by polynomials.  We can hence extend $\tau_\mu [\mathbf H]( f) $ for $f$ a polynomial in $n$ variables and $n$ conjugate variables. Let $f$ be a monomial of the form $f:(x_1,\bar x_1 \dots , x_n,\bar x_n) \mapsto x_1^{k_1} \bar x_1^{k'_1} \dots x_n^{k_n}\bar x_n^{k'_n}$ for some $k_1,k'_1, \dots , k_n,k'_n\geq 1$. Then the expression of $\tau_{G,o}$ implies  that
	$$f \big(\tau_{G,o}(H_1),  \dots , \tau_{G,o}(H_k) \big) = \tau_{G,o}(H_1)^{k_1} \overline{\tau_{G,o}(H_1)}^{k'_1} \dots \tau_{G,o}(H_n)^{k_n}\overline{\tau_{G,o}(H_n)}^{k'_n} =  \tau_{G,o}(H_0),$$
where $H_0 = H_0(\mbf k, \mbf H)  \in \cH^\bullet \langle J \rangle$ is the rooted test graph obtained as follow. For each $i=1, \dots ,n$ consider $k_i$ copies of each $H_i$, as well as $k'_i$ copies of the graph $H_i^*$ obtained by reversing the orientation of the edges of $H_i$ and replacing the label map $\varepsilon$ by $\varepsilon^*$. Then $H_0$ is obtained by identifying the vertex $1$ of all these graph copies for all $i$. 

As a conclusion, for $\mu$ in $\cP_{\infty}( \cGr(\dC^J))$, the moments of the traffic distribution read
	\begin{equation}\label{Annealing}
		\tau_\mu [\mathbf H]( f) = \E\big[ \tau_{G,o}(H_0) \big] := \bar \tau_\mu[H_0].
	\end{equation}
Note in particular that if  $G$ is a graph on the vertex set $\{1, \dots , n\}$ and $o$ is uniformly chosen at random, then $\tau_\mu [\mathbf H]( f)$ coincides with the definition of \cite{Camille}. Hence the traffic distribution $\tau_\mu \in \cP(\TRAFB)$ and the expectation of the rooted distribution $\E[\tau_{(G,v)}] \in \TRAFB$ define the same object when $\mu$ is in $\cP_{\infty}( \cGr(\dC^J))$ and we can use usual traffic techniques.

\subsection{Proof of Lemma \ref{Lem:Homomorphismfin}}\label{Sec:TraffAndBS}

In this section we prove the equivalence between the weak topology of traffics distributions and the Benjamini Schramm topology. We first note that, as usual for the product topology, we can consider a distance $d = d_J$ on $\cP(\TRAFB)$ by choosing, for each $n\geq 1$, a countable dense set  $(g_{n,k})_{k\in \mathbb N}$ of continuous bounded functions $\dC^n \to \dC$, and a exhaustive sequence of tuples $\mbf H_n,n\geq 1$ of rooted graph tests with $n$ edges, and setting 
$$d(\tau,\tau')=\sum_{n,k\in \mathbb N} \frac{1}{2^{n+k}}\mathrm{min} \big(1, |\tau(\mbf H_n)(g_{n,k})-\tau'(\mbf H_n)(g_{n,k})|\big).$$

We write in short $\cGr=\cGr(\dC^J)$ in this section. In order to prove the lemma it is sufficient to prove the injectivity and bi-continuity for the map $\nu\in \cP_{\infty}( \cGr) \to \tau_{\nu}\in \mathcal P_\infty(\TRAFB )$. Indeed, as seen in the proof of Lemma \ref{smooth}, we can approximate $\nu\in \cP( \cGr) $ by $\nu_{k}\in  \cP_{\infty}( \cGr)$ up to an error $\delta_{k}$ going to zero and moreover $\tau=\tau_{\nu}$ and $\tau_k=\tau_{\nu_k}$ are at distance going to zero with $\delta_k$ since the map $\nu\rightarrow \tau_{\nu}$ is continuous by Lemma \ref{continuous}.
 If we establish that there is a unique  $\nu_k=\nu_{\tau_k}$ pre-image of $\tau_k$, then it is as close as wished to $\nu$. It follows that if $\nu\in \cP_\infty( \cGr)\mapsto \tau_{\nu}\in \mathcal P_\infty(\TRAFB )$ is a homeomorphism on its image, then so is the map $\nu\in \cP( \cGr) \to \tau_{\nu}\in \mathcal P(\TRAFB)$. It is hence sufficient to prove the property for the map $(G,o) \mapsto \tau_{G,o}$ on the set of bounded locally finite rooted marked graphs.

Our main  tool is the so-called  injective transform $\tau^0_{G,o}$ of $\tau_{G,o}$. It is defined as the expression in the right hand side of \eqref{def:tauG} of $\tau_{G,o}$ where the sum is restricted to \emph{injective} maps $\phi$. For a partition $\pi\in \mathcal P(V_H)$ of the vertex set of a test graph $H$, we set $H^\pi$ the test graph obtained by identifying vertices in a same block of $\pi$. We therefore have the relations
\begin{eqnarray}
	\tau_{G,o}[H] & = & \sum_{\pi \in \mathcal P(V_H)} \tau^0_{G,o}[H^\pi],\label{MobiusOn}\\
	\tau^0_{G,o}[H]  & =  & \sum_{\pi \in \mathcal P(V_H)} \prod_{B\in \pi} (-1)^{|B|-1}(|B|-1)!\tau_{G,o}[H^\pi],\label{MobiusOff}
\end{eqnarray}
which exhibits a (linear) bijection $\tau_{G,o} \mapsto \tau^0_{G,o}$, bi-continuous for the product topology. It hence suffices to prove the injectivity and bi-continuity of the map $(G,o) \mapsto \tau^0_{G,o}$ on the set of rooted marked graph $(G,o)$ on the marked set $\dC^J$ with  degree uniformly bounded by $k$ and marks uniformly bounded by $k$, for $k$ arbitrary large.

We hence fix  such a graph $(G,o)$ and 
 prove that the injective traffic distribution of $(G,o)$ completely characterize the graph in a continuous way. The proof has 
 two steps. We first recover the graph $(G,o)$ up to the values of the non zero coordinate of the mark. 
More precisely, we call chromatic skeleton $\bar G$ of $G$ the marked graph obtained by replacing each mark  $\xi\in \mathbb C^{J}$ by $(1_{\xi_{j}\neq 0})_{j\in J}$. We first show in lemma \ref{LemmaSkeleton} below that the chromatic skeleton $(\bar G, o)$ is a continuous function of the injective traffic distribution of $(G,o)$. In a second step,  we prove that we can recover the  marks.

For the rest of the section, it is important to recall that a rooted graph is defined up to root preserving isomorphisms, which means that $(G,o)=(G',o')$ if and only if there is a root-preserving bijection $V_G\to V_{G'}$, such that an edge between two vertices of $G$ implies that there is an edge in $G'$ between the image of these vertices with same mark. For rooted test graphs, the bijection sends 1 to 1 and respects also the multiplicity and labels. Moreover, $G$ is a subgraph of $G'$ (we write $G\preccurlyeq G'$) if there is an  isomorphism from  $G$ to $G'$ which is an injective map $V_G\to V_{G'}$.

We define the skeleton  $\bar H $ of a test graph $H = (V,E,\gamma, \eps)$ as the triplet $\bar H = (V, \bar E, \bar \gamma)$ where $(V, \bar E)$ is the undirected graph obtained by forgetting the orientation of the edges of $E$, the $\varepsilon$ labels, as long as the multiplicity of the edges of $H$ labeled with a same color $j\in J$. For an edge $\bar e \in \bar E$ we associate the label $\bar \gamma(\bar e) = \gamma(e)$ where $e$ is any edge in the group forming $\bar e$. For any $p\geq 1$ let $\bar{\mathcal  H}^\bullet_p\langle J \rangle$ be  the set of skeleton test graphs $\bar H$ whose depth is not greater than $p$, up to isomorphisms preserving 1. For $H_0\in  \bar{\mathcal  H}^\bullet_p\langle J \rangle$, denote by $\mathcal  H^\bullet_{H_0}\langle J\rangle$ the set of test graphs $H \in\mathcal  H^\bullet\langle J\rangle$  such that $\bar H =H_0$. In this proof, we identify a chromatic skeleton with the skeleton test graph obtained by replacing an edge $e$ whose non zero coordinates are $j_1, \dots,  j_n$ by $n$ edges labeled so between the same endpoints. 

\begin{lemma}\label{LemmaSkeleton} For any $p\geq 1$, let $\mathcal X_p$ be the set of all $H'$ in $\bar{\mathcal  H}^\bullet_p\langle J \rangle$ such that there exists $H \in \mathcal  H_{H'}^\bullet\langle J\rangle$ satisfying 
$\tau^0_{  G,o}(H ) \neq 0$. Then $\mathcal X_p$ has a maximal element with respect to the partial order $\preccurlyeq$, which is the chromatic skeleton graph of $  G_p$. As a consequence, the latter is a continuous function of the injective distribution $\tau^{0}_{G,o}$.
\end{lemma}

\begin{proof}[Proof of Lemma \ref{LemmaSkeleton}.]
	Recall the definition of $\tau^0$: for all $H= (V,E,\gamma, \eps)$, 
	\begin{equation*}
		\tau^0_{G ,o}(H)= \sum_{\substack{ \phi:V_H \to V \\ \phi(1)=o\\\mathrm{injective}}} \underbrace{ \prod_{e=(u,v)\in E} \xi_{G_{\ell(e)}}^{\varepsilon(e)}\big(\phi(v), \phi(u)\big)}_{\delta^0(H,\phi)}.
\end{equation*}

	 Note first that for any test graph $H$, if $\bar H  \not \preccurlyeq \bar{ G}_p$ then the above expression vanishes since the sum is over the empty set. For any $p\geq 1$, we fix a test graph  $H_0$ whose skeleton is $\bar G_p$ (i.e. $H_0\in \mathcal  H^\bullet_{\bar{G}_p}\langle J\rangle$), whose edges are of multiplicity two in each color, such that for each pair $\bar e = \{e,e'\}$ of edge of multiplicity two, we have $\varepsilon(e) \neq \varepsilon(e')$ and the edges have opposite direction.

	 Let $\phi:V_{H_0} \to V$ be an injective map such that $\phi(1)=o$ and $\xi_{G_{\ell(e)}}\big(\phi(e)\big)\neq 0$ for all $e\in E_{H_0}$. Firstly, the latter condition applied to the edges of $H_0$ adjacent to its root 1 implies that $\phi$ must send the neighbors $i_1, \dots ,  i_L$ of 1 to the neighbors $v_1, \dots , v_{L'}$ of $o$ by preserving the adjacency and the color labels of the edges between $1$ and the $i_\ell$'s. Moreover, the skeleton of $H_0$ is the chromatic skeleton of $(\bar{ G}_p,o)$, so the numbers of root neighbors are the same in $H_0$ and in $\mathcal G$, i.e. $L=L'$.  Since $\phi$ is injective, it realizes a bijection between these two sets of vertices. Similarly, if two vertices $i_{k}$ and $i_{k'}$ of $H_0$ are linked by an edge with color $j$, then $\delta^0(H,\phi)\neq 0$ implies that there an edge in $\mathcal G$ with same label between the images of $\phi$. Altogether, this shows that, restricted to the ball of radius one, $\psi$ an isomorphism of rooted graphs from $\bar H_0$ to $\mathcal G_1$.
	 
	 We can reproduce the argument to the vertex adjacent to the neighbors of $o$, and by induction on the distance of the root, see that $\phi$ is an bijection between $H_0$ and the vertex set of $(\bar G_p,o)$, which must respect the root, the adjacency of the graph (with the above convention). Hence $\phi$ is a root preserving isomorphism between the skeleton of $H_0$ and the skeleton $\bar G$. In particular, for such a $\phi$, in the expression of $\delta^0(H_0,\phi)$ each mark appears exactly twice with opposite $\varepsilon$ label, 
	 	$$\delta^0(H_0,\phi) =  \prod_{e\in E_{H_0}} \prod_{\substack{j\in J | \xi_j(e) \neq 0}}  \big| \xi_j(e) \big|^2\neq 0.$$
This is valid for any automorphism $\phi$, so if we denote by $\mathrm{Aut(H_0)}$ the set of automorphisms of $H_0$ that preserve the root 1, then we get $\tau_{G ,o}(H) = \delta^0(H,\phi)  \times |\mathrm{Aut(H_0)}| \neq 0$. 
\end{proof}

Now that we have determined the chromatic skeleton of $(\mathcal G_p,o)$, we shall determine its marks. We first consider an orientation of the edges of $E_0$. Then we denote by $X_p = (\xi(e))_{e\in E_0}$ the collection of all non-zero marks given that orientation ($\xi(e)=\xi_j(e')$ whenever $e$ has color $j$ and corresponds to $e'$ via the identification $H_0\sim\bar G_p$). Since $(G,o)$ is determined up to automorphism, note that this collection is invariant by the action of $\mathrm{Aut}(H_0)$. The knowledge of  $X_p$ is equivalent to the knowledge of the empirical distribution of marks 
	$$\mu_{\xi(G_p)} = \frac 1 {|\mathrm{Aut}(H_0)|} \sum_{\phi \in \mathrm{Aut}(H_0)} \delta_{ \xi(\phi(e)), e\in E_{0}},$$
where we use the shortcut $ \phi(e) = (\phi(v), \phi(u))$ for $e=(u,v)$.

 On the other hand, for two collections $\mbf k, \mbf k'$  of positive integers indexed by the edges of $H_0$, denote the monomial $h( \mbf x)= \prod_{e\in E_0} x^{k(e)} \bar x^{k'(e)}$ and set  $H_{  h} \in  \mathcal  H_{\bar G_p}\langle J\rangle$ the test graph with skeleton $\bar G_p$, whose edges are oriented as for $E_0$ with $k(e)$ representants with label $\varepsilon$ equal to 1 and $k'(e)$ representants with label $k'(e)$. The argument of Lemma \ref{LemmaSkeleton} holds for $H_h$ instead of $H_0$:  that the injections $\phi:V_{H_{0}}\rightarrow V$ such that $\phi(1)=o$ and $\delta^0(H(h),\phi)\neq 0$ are the root preserving isomorphisms of the skeleton of $H_{h}$ in $\bar G_p$, in which case
 	$$\delta^0(H(h),\phi) = \prod_{e\in E_0} ( \xi( \phi(e)))^{k(e)}( \bar \xi( \phi(e)))^{k'(e)}.$$
We then have, using the invariance by automorphism of the marks
	\begin{eqnarray*}
		\tau^0_{ G, o}[H(  h) ] & = & \sum_{\phi \in \mathrm{ Aut}(H_0)}\prod_{e\in E_0} ( \xi( \phi(e)))^{k(e)}( \bar \xi( \phi(e)))^{k'(e)} \\
		& = &  \sum_{\phi \in \mathrm{ Aut}(H_0)} \prod_{e\in E_0} ( \xi(e))^{k(\phi(e))}( \bar \xi( e))^{k'(\phi(e))} =|\mathrm{Aut}(H_0)| \int  h \mathrm{d}\mu_{\xi(G_p)},
	\end{eqnarray*}
where $\mu_{\xi(G_{p})}$ is as above.
First we extend the definition of $h\mapsto \tau^0_{ G, o}[H(  h) ]$ by linearity to polynomials. Since $\mu_{\xi(G_{p})}$ has compact support (recall that we bounded the marks by $k$), by Stone-Weierstrass we can continuously extend this definition to bounded continuous functions. We then have completely determined $\int f \mathrm{d}\mu_{\xi(G_p)}$ in a continuous way. By Riesz representation theorem, there is a unique positive measure $\mu$ such that $\int f d\mu = \int  h \mathrm{d}\mu_{\xi(G_p)}$. Hence $\mathrm{d}\mu_{\xi(G_p)}$ is the unique possible  distribution of weights given $\tau^0_{  G, o}$. 

This  proves  that the equivalence class of  a deterministic graph $(G,o)$ with bounded weights and degree
 is determined continuously by its injective traffic distribution $\tau_{(G,o)}$. It follows that $\nu\in \cP_\infty( \cGr)\mapsto \tau_{\nu}\in \mathcal P(\TRAFB )$ is a homeomorphism on its image, and so is $\nu\in \cP( \cGr)\mapsto \tau_{\nu}\in \mathcal P(\TRAFB )$ which completes the proof of Lemma \ref{Lem:Homomorphismfin}.

\subsection{Proof of Theorem \ref{LDPtraffic1}}

As before, we can therefore deduce a LDP for the rooted-traffic distribution from our LDP on marked graphs. We define a marked graph $G(\mbf Y) = (V,E,\xi)$ with marks in $\cZ = \dC^J$ on the vertex set $V= \{1,\ldots,n\}$ by setting $E = \{ \{k,l\} : Y_j (k,l) \ne 0 \hbox{ for some $j$}\}$ and $\xi ( k,l) = (Y_j(k,l))_{j \in J}$.

Let $p_n$ be the probability that $\{ 1,2\} \in E$ and let $\gamma'_n$ be the law of $\xi(1,2)$ conditioned on $\{1,2\} \in E$. We set $d_n =  n p_n$. The graph $G(\mbf Y)$ is a marked Erd\"os-R\'enyi random graph as defined in Subsection \ref{subsec:LDPER} with parameters $(d_n,\gamma'_n)$. Using independence and \eqref{eq:gamman2N}, it is straightforward to check that 
$$
\lim_n d_n =  d_{\Lambda} =  |J| \Lambda(\dC),
$$
(where for ease of notation, we have extended $\Lambda$ to $\dC$ by setting $\Lambda (\{0\}) = 0$).  We have $0 < d_{\Lambda} < \infty$ by \eqref{eq:Lambdabddt}. Similarly, for $j \in J$, let $\Lambda^{j}$ be the measure on $\dC^J$ defined as $\Lambda^j = \otimes_{i \neq j} \delta_0 \otimes \Lambda$. We have, for every Borel set $A \subset \dC^J$, 
$$
\lim_n \gamma'_n(A) = \gamma_{\Lambda}(A) =  \frac{1}{ |J| \Lambda(\dC)} \sum_{j \in J} \Lambda^j(A).
$$

As a consequence of Theorem  \ref{th:SIGMAER}, we therefore find that $U( G(\mbf Y))$ satisfies a LDP in $\cP (\cGr)$ with good rate function $\Sigma^{\tiny {\rm ER}}_{\gamma_{\Lambda},d_{\Lambda}}$. Also from \eqref{eq:trafGA}, we have $\hat \tau_{U( G(\mbf Y))} = \Utau_{\mbf Y}$. Therefore, the contraction principle and Lemma \ref{continuous} imply Theorem \ref{LDPtraffic1} with good rate function {{
\begin{equation}\label{defent} \chi^{*}_{\Lambda}(\tau)=\inf\{ \Sigma^{\tiny {\rm ER}}_{\gamma_{\Lambda},d_{\Lambda}}(\mu):\tau_\mu = \tau\},\end{equation}}}
(the notation $\chi^{*}_{\Lambda}$ is valid since both $d_{\Lambda}$ and $\gamma_{\Lambda}$ are functions of $\Lambda$, note however that $\chi^{*}_{\Lambda}$ depends implicitly on $J$). Since Theorem  \ref{th:SIGMAER} also proves that $ \Sigma^{\tiny {\rm ER}}_{\gamma_{\Lambda},d_{\Lambda}}$ has a unique minimizer $\mu_{\gamma_{\Lambda},d_\Lambda}$,  $\chi^{*}_{\Lambda}$ has a unique minimizer $\tau_{\Lambda}=\tau_{\mu_{\gamma_{\Lambda},d_\Lambda}}$. Moreover, by Lemma \ref{Lem:Homomorphismfin}, we see that the infimum defining $\chi^{*}_{\Lambda}(\tau)$ is achieved at $\nu_{\tau}$ so that
$$ \chi^{*}_{\Lambda}(\tau)=\Sigma^{\tiny {\rm ER}}_{\gamma_{\Lambda},d_{\Lambda}}(\nu_{\tau})\,.$$
Moreover, since this minimizer must also be the almost sure limit of $\tau_{\mbf Y}$, we deduce from \cite{Camille2} that it is the free product of the marginal limiting traffic distributions, described in \cite{Camille,Camille2} and in the next sections. 
Furthermore, by Lemma \ref{le:casinfini}, $\Sigma^{\tiny {\rm ER}}_{\gamma_{\Lambda},d_{\Lambda}}$ is infinite if $\nu$ is not admissible and therefore is supported on marked trees. Hence, $\chi^{*}_{\Lambda}(\tau)$ is infinite if $\tau$ is not the traffic associated to a random rooted tree.
\qed

\subsection{Definitions of the free products}
The definition of the free product of random graphs uses a step by step construction algorithm.

\begin{definition} Let $\mu_1 \in \cP\big(\cGr(\dC^{J_1})\big)$ and $\mu_2\in \cP\big(\cGr(\dC^{J_2})\big)$ be two distributions of rooted weighted random graphs. We call the free product of $\mu_1$ and $\mu_2$ and denote $\mu_1*\mu_2\in \cP\big(\cGr(\dC^{J_1\sqcup J_2})\big)$ the law of the random graph $(G,\rho)$ described as follows. 

Let first $(G^{(0)}_1,\rho_1^{(0)})$ and $(G^{(0)}_2,\rho_2^{(0)})$ be two independent realizations of $\mu_1$ and $\mu_2$ and denote by $G^{(0)}$ the graph obtained by identifying $\rho_1^{(0)}$ and $\rho_2^{(0)}$. For each edge $e$ of $G^{(0)}$, we associate the mark $\xi^{(0)}(e) = ( \xi(e), 0, \dots, 0)$ if $e$ comes from $G^{(0)}_1$ and $\xi^{(0)}(e) = ( 0, \dots, 0, \xi(e))$ if it comes from $G^{(0)}_2$ so that the marks belong to $\dC^{J_1\sqcup J_2}$. We say that we have "fused" $(G^{(0)}_2,\rho_2^{(0)})$ to $G^{(0)}_1$ at $\rho_1^{(0)}$. 

Let $v$  be a vertex of $G^{(0)}$ different from $\rho$, and let us consider an independent realization $G_v^{(1)}$, either distributed according to $\mu_1$ if $v$ belongs to $G_2^{(0)}$, or according to $\mu_2$ if $v$ is a vertex of $G_1^{(0)}$. Then we fuse $G_v^{(1)}$ to $G^{(0)}$ at $v$ and repeat this operation for for each vertex different from the root,  getting a graph $G^{(1)}$. We pursue this process to construct a sequence of graphs $G^{(n)}, n\geq 1$, rooted at $\rho_1^{(0)}= \rho_2^{(0)}$, obtained by fusing independent copies alternating from $\mu_1$ and $\mu_2$. This sequence converges in weak local topology since it describes neighborhood of arbitrary depth. We denote by $(G, \rho)$ the limit of this random rooted graph   and by $\mu_1*\mu_2$ its law.
\end{definition}

On the other hand, \cite{Camille} defines a notion of product for traffic distributions in the following terms. A colored component of a test graph $H \in \cH \langle J_1 \sqcup J_2 \rangle$ is a maximal connected subgraph of $H$ labeled in $J_1$ or in $J_2$. We denote by $\mathcal {CC}(H)$ the set of colored components of $H$. The graph of colored components $\mathcal {GCC}(H)$ is the undirected graph  whose vertex set is $\mathcal {CC}(H)$, and such that two colored components are linked by one edge for each vertex they have in common.

Moreover, if $H$ in $\mathcal H^\bullet\langle J_1\sqcup J_2\rangle$ is a rooted test graph such that $\mathcal {GCC}(H)$ is a tree, then the colored colored components $S$ of $H$ are rooted as follows. If $S$ contains the root $1$ of $H$, then it is rooted at $1$. Otherwise, $S$ is rooted at the closest vertex to the root 1 of $H$ for the graph distance, that we may call the cut-vertex of $S$.

\begin{definition} Let $\tau_1 \in \mathrm{Traf}\langle J_1 \rangle$ and $\tau_2 \in \mathrm{Traf}\langle J_2 \rangle$. The  free product $\tau_1*\tau_2 \in \mathrm{Traf}\langle J_1\sqcup J_2 \rangle$ is then defined via  the injective traffic distribution $(\tau_{1}*\tau_{2})^{0}$ thanks to \eqref{MobiusOn}, by the following formula: for any $H$ in  $\mathcal H\langle J_1\sqcup J_2\rangle$
	\begin{eqnarray}\label{TrafficFreeness}
		\big(\tau_1*\tau_2\big)^0(H) & = & \mathbf 1\Big( \mathcal {GCC}(H) \mathrm{\ is \ a \ tree}\Big) \prod_{ S \in \mathcal {CC}(H)} \tau_{i(S)}^0(H),
	\end{eqnarray}
where $i(S) =i$ whenever $S\in \mathcal H\langle J_i\rangle$.

\end{definition}

\begin{proposition} Let $\nu$, $\nu_1$ and $\nu_2$ be three random graphs with law respectively in $\cP_\infty\big(\cGr(\dC^{J_1\sqcup J_2})\big),$ $\cP_\infty\big(\cGr(\dC^{J_1})\big)$ and $\cP_\infty\big(\cGr(\dC^{J_2})\big)$. Then $\nu=\nu_1* \nu_2$ if and only if the map $\bar \tau_\nu$ defined in \eqref{Annealing} satisfies $\bar \tau_\nu = \bar \tau_{\nu_1} * \bar \tau_{\nu_2}$. 
\end{proposition}

Therefore, for $\tau_{\mu_1}\in \cP(\mathrm{Traf}_\cG^\bullet\langle J_1\rangle)$ 
and $\tau_{\mu_2}\in \cP(\mathrm{Traf}_\cG^\bullet\langle J_2 \rangle)$ we define the free product $ \tau_{\mu_1}* \tau_{\mu_2}:= \tau_{ \mu_1* \mu_2}\in \cP(\mathrm{Traf}_\cG^\bullet\langle J_1\sqcup J_2 \rangle)$ as the traffic distributions of the free product of the laws of marginal graphs, consistently with \cite{Camille}.

\begin{proof} We say that a test graph $H \in \mathcal H\langle J_1 \sqcup J_2\rangle$ is a \emph{free product} if it is constructed as in the definition of the free product of random graphs by considering arbitrary test graphs in $\mathcal H^\bullet\langle J_1\rangle$ and in $\mathcal H^\bullet\langle J_2\rangle$ instead of independent copies of the random graphs, and proceeding a finite number of steps. These copies in $\mathcal H\langle J_i\rangle$ clearly form the colored components of $H$ since at each step we fuse graphs with different labels. The graph of colored components of $H$ is a tree since the fusing process results in cut-vertex identifications.

Let us show that the  traffic distribution of $\nu=\nu_1* \nu_2$ satisfies \eqref{TrafficFreeness}. If the graph of colored components of $H$ is not a tree, its skeleton is not a subgraph of the skeleton of any realization $(G,\rho)$ of $\nu$, therefore $\big(\tau_1*\tau_2\big)^0(H)=0$. Assume that it is a tree and let $\phi : V_H \to V$ such that $\phi(1) = \rho$ and $ \xi^{\varepsilon(e)}_{\ell(e)}\big(\phi(w), \phi(v)\big)\neq 0$ for any $e =(v,w) \in E_H$. The latter condition implies that the vertices that are in a colored component of $H$ containing 1 are sent by $\phi$ to vertices in a colored component of the root in $G$.

Let  now $S$ be a colored component of $H$ that do not contain the root, but contains a cut-vertex $v_0$ in a colored component of the root. The condition $ \xi^{\varepsilon(e)}_{\ell(e)}\big(\phi(w), \phi(v)\big)\neq 0$ implies that the image of $S$ belong to a colored component of $G$ with the same property with a cut-vertex $w_0$ for which $\phi(v_0) = w_0$. By induction, the image by $\phi$ of a colored component of $H$ belongs to a colored component of $G$ that is determined by its cut-vertex in the same way.

Let now $S_0$ be a leaf in the tree of colored component of $H$ with cut vertex $v_0$. Denote $H\setminus S_0$ the graph obtained from $H$ by removing the edges of $S$ and the vertices that are not $v_0$. Consider $S_0$ as rooted in $v_0$. A map $\phi$ as above therefore decomposes as an injective map $\phi'$ on the vertices of $H\setminus S_0$, and denoting $w_0$ the image by $\phi'$ of $v_0$, an injective map $\phi_0$ on the vertices of $S_0$ with image in the colored component of $(G,o)$ with cut-vertex $w_0 = \phi'(v_0)$:
	\begin{eqnarray*}
		\bar \tau_\nu^0(H) & = &  \sum_{\substack{ \phi':V_{H\setminus S_0} \to V \\ \phi(1)=i}}    \E\Big[ \prod_{e=(v,w)\in E_{H\setminus S_0}} \xi^{\varepsilon(e)}_{\ell(e)}\big(\phi(w), \phi(v)\big)  \\
		&& \sum_{\substack{ \phi_0:V_{S_0} \to V\setminus{\phi'(V)} \\ \phi(v_0)=w_0}}\prod_{e=(v,w)\in E_{S_0}} \xi^{\varepsilon(e)}_{\ell(e)}\big(\phi(w), \phi(v)\big) \Big] 
	\end{eqnarray*}
Conditionally on $w_0$, the colored component rooted at $w_0$ is independent of the rest of the graph, which gives the formula
	\begin{eqnarray*}
		\bar \tau_\nu ^0(H) & = & \bar \tau_\nu^0(H\setminus S_0) \times \tau_{i(S_0)}[S_0].
	\end{eqnarray*}
By induction, this proves that $\bar \tau_\nu $ is the free product of the marginal distributions. 

Reciprocally, if $\mu$ is the law a random rooted graph such that $\bar \tau_\nu = \bar \tau_{\nu_1} * \bar \tau_{\nu_2}$, then it has the same traffic distribution as the free product of the marginal laws of random graphs $\mu_1$ and $\mu_2$. By injectivity of the map $\mu\mapsto \tau_\mu$ (Lemma \ref{Lem:Homomorphismfin}), necessarily $\mu = \mu_1*\mu_2$.
\end{proof}

\subsection{Proof of Corollary \ref{cor:trafficind} for the free product of random graphs}

In this section, all random rooted-traffic distributions belong  to $\cP(\mathrm{Traf}_\cG^\bullet\langle J \rangle)$, i.e. there are distributions of random rooted graphs. 

Let $\mbf Y = (Y_j)_{j \in J}$ be as in Theorem \ref{LDPtraffic1}. To stress on the dependence of $\chi_{\Lambda}$ on $J$, we write here $\chi_{\Lambda,J}$ in place of $\chi_{\Lambda}$. By Theorem \ref{LDPtraffic1}, we have the weak large deviation  principle , namely that for every $\tau\in\cP(\mathrm{Traf}_\cG^\bullet\langle J \rangle)$, the microstates entropy $\chi_{\Lambda,J}$ defined by
\begin{eqnarray}
-\chi_{\Lambda,J}(\tau)&:=&\lim_{\varepsilon\downarrow 0}\limsup_{n\rightarrow\infty}\frac{1}{n}\log \dP(d(\tau_{\mbf Y},\tau)\le \varepsilon)\label{defchi}
\end{eqnarray}
is unchanged if we replace the limsup by a liminf:
\begin{eqnarray}
-\chi_{\Lambda,J}(\tau)&=&\lim_{\varepsilon\downarrow 0}\liminf_{n\rightarrow\infty}\frac{1}{n}\log \dP(d(\tau_{\mbf Y},\tau)\le \varepsilon)\label{micros}
\end{eqnarray}

Theorem \ref{LDPtraffic1} shows that $\chi_{\Lambda,J}=\chi^{*}_{\Lambda,J}$ but we will not use the formula for $\chi^{*}_{\Lambda,J}$ in this proof.
Because the limsup is equal to the liminf, we can prove summability under traffic independence.  
If $J = J_1 \sqcup J_2$ and  $\tau\in \mathcal P(\cH^*\langle J\rangle )$, we denote by $\tau_{i}=\tau|_{J_{i}}$ the restriction of $\tau$ to test graphs in $\mathcal H\langle J_{i}\rangle$.  Then we always have
\begin{equation}\label{tr}\chi_{\Lambda ,J_{1}\sqcup J_{2}}(\tau)\ge \chi_{\Lambda,J_{1}}(\tau_{1})+\chi_{\Lambda,J_{2}}(\tau_{2})\end{equation}
Indeed, by definition, if $\mbf Y=(\mbf Y_{1},\mbf Y_{2})$ with $\mbf Y_{i}$ the matrices indexed by $J_{i}$, we always have for $i=1,2$,
$$ \{d(\tau_{\mbf Y},\tau)\le \epsilon\}\subset \{d(\tau_{\mbf Y_{i}},\tau_{i})\le \epsilon\}
$$
Consequently
\begin{eqnarray*}
\mathbb P\left( \{d(\tau_{ \mbf Y},\tau)\le \epsilon\}
\right)&\le& \mathbb P\left( \{d(\tau_{\mbf Y_{1}},\tau_{1})\le \epsilon\}\cap  \{d(\tau_{\mbf Y_{2}},\tau_{2})\le \epsilon\}
\right)\\
&=&\mathbb P\left( \{d(\tau_{\mbf Y_{1}},\tau_{1})\le \epsilon\}\right)\mathbb P\left( \{d(\tau_{\mbf Y_{2}},\tau_{2})\le \epsilon\}
\right)
\end{eqnarray*}
where we finally used the independence of $\mbf Y_{1}$ and $\mbf Y_{2}$. Therefore, we have
$$\lim_{\varepsilon\downarrow 0}\limsup_{n\rightarrow\infty}\frac{1}{n}\log \dP(d(\tau_{\mbf Y},\tau)\le \varepsilon)\qquad\qquad\qquad\qquad\qquad\qquad
$$
$$\qquad\qquad\quad\le
\lim_{\varepsilon\downarrow 0}\limsup_{n\rightarrow\infty}\frac{1}{n}\log \dP(d(\tau_{\mbf Y_{1}},\tau)\le \varepsilon)+\lim_{\varepsilon\downarrow 0}\limsup_{n\rightarrow\infty}\frac{1}{n}\log \dP(d(\tau_{\mbf Y_{2}},\tau)\le \varepsilon)$$
from which \eqref{tr} follows. By taking $\tau=\tau_{1}*\tau_{2}$ with the definition of the previous section, we deduce by the definition \eqref{defchi}  that

$$\chi_{\Lambda ,J_{1}\sqcup J_{2}}(\tau_{1}*\tau_{2})\ge \chi_{\Lambda,J_{1}}(\tau_{1})+\chi_{\Lambda,J_{2}}(\tau_{2})$$
To prove the converse bound,  
it is enough  to show thanks to \eqref{micros} that for every $\eps>0$ there exists $\delta(\eps)>0$ so that
\begin{equation}\label{top}\liminf_{n\rightarrow\infty}\frac{1}{n}\log\frac{\dP(d(\tau_{\mbf Y},\tau_{1}*\tau_{2})\le \varepsilon)}{\dP(d(\tau_{\mbf Y_{1}},\tau_1)\le \kappa(\varepsilon))\dP(d(\tau_{\mbf Y_{2}},\tau_2)\le \kappa(\varepsilon))}\ge 0\,.\end{equation}
 To prove this statement, the idea is that by independence of ${\mbf Y}_{1}$ and ${\mbf Y}_{2}$, we can conjugate the matrices of ${\mbf Y}_{2}$ by an independent permutation $S$ and that conditionally to $({\mbf Y}_{1}, {\mbf Y}_{2})$, ${\mbf Y}_{1}$ and $S{\mbf Y}_{2}S^{*}$ are asymptotically  traffic independent. A difficulty arises from the fact that we need this convergence to hold uniformly over the test graphs, or equivalently for the distance $d$. To do so, we first approximate the  graphs and their traffic distributions by ones with bounded degrees and bounded entries.

To this end, for  a fixed integer number $k>1$, we denote by $\mbf Y^{k}$ the sequence of matrices where the entries corresponding to vertices with degree larger than $k$ or entries with modulus greater than $k$ have been put to zero (namely we replace the entries by their pushforward by a smooth function $f_{k}$ which vanishes outside the ball of radius $k$  and is equal to one on the ball of radius $k-1$). We let $G(\mbf Y^k)$ be the corresponding marked graph with marks on $\dC^J$ which is defined in the proof of Theorem \ref{LDPtraffic1}. By definition, we have $\tau_{U( G(\mbf Y^k))} = \tau_{\mbf Y^k}$. 
We have already seen that $U( G(\mbf Y^k))$ converges towards $U(G({\mbf Y}))$ for the weak topology as $k$ goes to infinity. Therefore, by continuity of $\hat \tau$, $ \tau_{\mbf Y^k}$ converges towards $\tau_{\mbf Y}$ pointwise. To get a uniform convergence we restrict ourselves to a compact set. 
In fact, for any compact set of $K \subset \cP(\cGr)$, there exists a function $\delta_K(k)$ going to $0$ as $k \to \infty$, such that for all $\eta >0$, 
\begin{equation}\label{approx}
  d(\tau_{ \mbf Y }, \tau_{ \mbf Y^{k}}) \le \delta_K(k) +\IND ( (U(G(\mbf Y))  \notin K). 
\end{equation}

As explained in the proof of Theorem \ref{LDPtraffic1}, $G(\mbf Y)$ is a marked Erd\"os-R\'enyi random graph. By Theorem \ref{th:SIGMAER}, $U(G(\mbf Y))$ is exponentially tight since it satisfies a LDP with a good rate function: for any $M >0$, there exists a compact set $K$ such that $\dP ( U(G(\mbf Y))  \notin K) \leq  \exp ( - n M)$ for all $n$ large enough. 
As a consequence, by taking $k=k(\veps,M)$ large enough so that $\delta_K(k) \leq \veps /2$, we get
\begin{eqnarray*}
\dP(d(\tau_{\mbf Y},\tau_{1}*\tau_{2})\le \varepsilon) &\ge & \P(\{d(\tau_{\mbf Y},\tau_{1}*\tau_{2})\le \varepsilon\}\cap \{ U(G(\mbf Y))  \in K \} )\\
&\ge &\dP(d(\tau_{\mbf Y^k},\tau_{1}*\tau_{2})\le \varepsilon/2)-e^{-nM }\end{eqnarray*}
 We note $\tau_{i}^{k}$ the traffic distribution  of the operators $\mbf A_{G}^{i}$ with $G$ truncated to have degree smaller than $k$ and entries given by the pushforward by $f_{k}$ of the entries of $G$. Similarly, $\tau^{k}_{i}$ converges weakly towards $\tau^{i}$ when $k$ goes to infinity so that $d(\tau^{k}_{i},\tau_{i})\le \eps/8$ for $k$ large enough. Moreover, it is easy to see from the definition of independence that we also have
 that $\tau_{1}^{k}*\tau_{2}^{k}[\mbf H](f)$ goes to $\tau_{1}*\tau_{2}[\mbf H](f)$ as $\eta$ goes to zero so that $
 d(\tau^{k}_{1}*\tau^{k}_{2},\tau_{1}*\tau_{2})$ goes to zero as $k$ goes to infinity. We assume hereafter $k$ large enough so that this is smaller than $\eps/4$.

Now, the matrices $\mbf Y_{1}^{k}$ and $\mbf Y_{2}^{k}$ are independent and the distributions of $\mbf Y_{i}^{k}$ is permutation invariant (that is $\mbf Y_{i}^{k}$ and $ S \mbf Y_{i}^{k} S^*$ have the same distribution for any permutation matrix $S$ of size $n$). In particular, we can  replace $\mbf Y_{2}^{k}$ by $ S \mbf Y_{2}^{k} S^*$ where $S$ is a uniform permutation matrix and set $S.\mbf Y^k = (\mbf Y^k_1, S \mbf Y^k_{2}  S^*)$. By construction $\mbf Y^k $ and $S.\mbf Y^k $ have the same distribution. It was shown in \cite[Theorem 1.8]{Camille} that, given $(\mbf Y^k_1, \mbf Y^k_2)$ so that $d( \tau_{\mbf Y^k_i},\tau_{i}^{k})\le \eps/8$, $i=1,2$
 $\tau_{ S. \mbf Y^k}(\mbf H, P)-\tau_{1}^{k}*\tau_{2}^{k}(\mbf H, P)$  is of order  $O(\eps)$  in probability (with respect to the randomness of the permutation $S$) for every test graph $\mbf H=(H_{1},\ldots,H_{k})$ and every polynomial $P$. Importantly,  the errors are uniform on $(\mbf Y^k_1, \mbf Y^k_2)$ (but depend on $k$). Because everything is bounded, we can approximate any bounded continuous function $f$ by a polynomial $P$.  This implies that uniformly  $\tau_{S. \mbf \mbf Y^k}(\mbf H, f)-\tau_{ 1}^{k}*\tau_{2}^{k}(\mbf H, f)$ is of order $\eps$, and therefore, given $(\mbf Y^k_1, \mbf Y^k_2)$,
 $d(\tau_{S. \mbf Y^k}, \tau_{1}^{k}*\tau_{2}^{k})$
 goes to zero in probability uniformly in $(\mbf Y^k_1, \mbf Y^k_2)$. Notably,  for every $\delta>0$, there exists $\kappa(\delta)=\kappa_{k}(\delta)>0$ going to zero with $\delta$, so that  for $n $ large enough,
 \begin{equation}\label{tu}
 \dP (  d(\tau_{S. \mbf Y^k}, \tau_{1}^{k}*\tau_{2}^{k}) \leq \delta | \max_{i=1,2}d( \tau_{\mbf Y^k_i},\tau_{i}^{k})\le \kappa(\delta)) \geq 1/2.
 \end{equation}
 We take $\delta<\eps/4$, remember that $k$ was chosen large enough so that $d(\tau_{1}^{k} *\tau_{2}^{k},\tau_{1}*\tau_{2})\le \eps/8$ and write

\begin{eqnarray*}
 \dP(d(\tau_{\mbf Y^k},\tau_{1}*\tau_{2})\le \varepsilon/2) &=&  \dP(d(\tau_{S. \mbf Y^k},\tau_{1}*\tau_{2})\le \varepsilon/2)
\\ 
&\ge&  \dP(d(\tau_{S. \mbf Y^k}, \tau_{1}^{k}*\tau_{2}^{k}) \leq \delta,  \max_{i=1,2}d(\tau_{ \mbf Y_{i}^{k}}, \tau_{i}^{k})\le\kappa(\delta) )\\
&\ge& \frac{1}{2} \dP( d(\tau_{ \mbf Y_{1}^{k}}, \tau_{1}^{k})\le\kappa(\delta) ) \dP( d(\tau_{ \mbf Y_{2}^{k}}, \tau_{2}^{k})\le\kappa(\delta) )\\
&\ge&  \frac{1}{2} \dP( d(\tau_{ \mbf Y_{1}}, \tau_{1})\le\kappa'(\delta) ) \dP( d(\tau_{ \mbf Y_{2}}, \tau_{2})\le\kappa'(\delta) )  - 2e^{-nM}
\end{eqnarray*}
where we finally used \eqref{approx} to find $\kappa'(\delta)>0$ so that for $k$ large enough $ \{d(\tau_{ \mbf Y_{1}}, \tau_{1})\le\kappa'(\delta)\} \subset \{d(\tau_{ \mbf Y_{1}^{k}}, \tau_{1}^{k})\le\kappa(\delta) \}$ when $U(G(\mbf Y))$ belongs to $K$. Hence, gathering the above estimates we deduce

$$
\dP(d(\tau_{\mbf Y},\tau_{1}*\tau_{2})\le \varepsilon) \geq \frac{1}{2} \dP(d(\tau_{\mbf Y_{1}},\tau_1)\le\kappa( \delta) )\dP(d(\tau_{\mbf Y_{2}},\tau_2)\le \kappa( \delta) )  - 4e^{-nM}.
$$
Since $M$ can be taken arbitrarily large so that the last term in the above right hand side is negligible, \eqref{top}  follows. This proves Corollary \ref{cor:trafficind} when the random rooted traffic distributions are distributions of random graphs. \qed

\section{Random traffics and Corollary \ref{cor:trafficind}}\label{trafind}

\subsection{Motivation}
At first glance, the previous section  concludes the proof of Corollary \ref{cor:trafficind}. For all $\tau_i \in \cP(\mathrm{Traf}^\bullet\langle J_i \rangle)$, if  $\tau_1$ is not in $\cP(\mathrm{Traf}_\cG^\bullet\langle J \rangle)$ then the additivity condition holds trivially since
\begin{equation}\label{trivid}+\infty = \chi_{\Lambda}(\tau_{1}*\tau_{2})=\chi_{\Lambda}(\tau_{1})+\chi_{\Lambda}(\tau_{2})=+ \infty+\chi_{\Lambda}(\tau_{2})\end{equation}
since $\chi_{\Lambda}(\tau_{2})\ge 0$ by definition of micro-state entropy. The problem, which is not related to LDP, is the ill-definiteness of the free product for random rooted-traffic distributions, when theyr are not the distribution of random rooted graphs. In general, a random rooted-traffic distribution in $ \cP(\mathrm{Traf}^\bullet\langle J\rangle)$ cannot be encoded into a canonical  traffic distribution in the classical sense as in \eqref{Annealing}, so the definition of traffic independence from \cite{Camille} cannot be directly applied. 

This section hence introduce more sophisticated notions, new in the context of traffic probability, whose aim is to extend  trivially the identity \eqref{trivid} by defining $\tau_{1}*\tau_{2}$ in greater generality.  Hence this part is less relevant outside the perspectives of traffic probability. The problem is that  seeing a random traffic as a traffic whose distribution is random is the classical sense \emph{omits} notions of dependence that have no analogue in classical probability. 
For sparse Wigner matrices, traffic probability nevertheless indicates that amalgamation is an ingredient that cannot be bypassed for stating the additivity property of heavy traffic entropy, thanks to a result of independent interest stated in the next section.

\subsection{Semicircularity over the diagonal of sparse Wigner matrices}\label{Subsection:Semicircularity}

Voiculescu's central notion of entropy is not defined in terms of sparse Wigner matrices, but of independent GUE matrices. Such as, the analogue of the heavy traffic entropy is, for Voiculescu, the semicircular entropy. The fact that the sparse traffic entropy is related to freeness with amalgamation, a notion defined by Voiculescu more than 50 years old, may be justified heuristically by the following result. 

\begin{proposition} Let $\mathcal X$ be a collection of independent sparse Wigner matrices as in \eqref{eq:SW} where the moments of $\gamma$ are finite. Then $\mathcal X$ converges to a semi-circular system over the diagonal, namely with amalgamation with respect to the operator $\Delta: A \mapsto \mathrm{Diag}_i(A_{i,i})$. 
\end{proposition}

\begin{proof} We sketch the proof, which is in fact valid for the slightly more general of Ryan-Zakharevich Wigner matrices with exploding moments considered in \cite{Camille2} with variance profiles. In the context of \cite{MaleDeltaFluct}, asymptotically semi-circularity of a collection $\mbf X$ of matrices  means the convergence to the same limit of the complex numbers
	\begin{equation}\label{FalseFreeness}
		\langle D_0 , X_1D_1 \dots X_LD_L\rangle =o(1)+  \langle D_0 , \sum_{ \sigma\in \mathrm{NC}_2(L) }  \kappa_\pi( X_1D_1, \dots ,X_LD_L) \rangle,
	\end{equation}
where, $\langle A , B\rangle  =  \E[\frac 1 N \Tr( AB^*)]$, $\kappa_\pi$ is the $L$-th cumulant-function over the diagonal and $NC_2(L)$ the set of non-crossing pair partitions. On the other hand, the so-called \emph{false freeness property} of \cite{Camille2}, that relates the distribution of heavy Wigner matrices with an enumeration of $\mathrm{NC}_2(L)$, is nothing else than a proof of \eqref{FalseFreeness}. 
\end{proof}

The interest is that it gives a new characterization of the macroscopic asymptotic properties of heavy Wigner matrices, in particular a Pastur's equation over the diagonal. Nevertheless the important consideration in the above proposition is that the matrices are not asymptotically semi-circular with respect to $\E[\Delta]$. Therefore using in practice this equation is delicate since we shall consider operator-valued free probability theory in a context where the diagonal coefficients are \emph{random}. Although our strategy do not take benefit of the above proposition, it serves us to guide the definition in the next subsection.

\subsection{Definition of random traffics and their independence}
 
Let $\mbf H = (H_1 , \dots,  H_k)$ in ${\dC\cH}^\bullet\langle J \rangle^{   k}$ be a $k$-tuple of rooted test graphs, where $H_s= (V_s,E_s, \gamma_s, \varepsilon_s)$ for any $s=1, \dots , k$. 
 Denote by $H_0$ the test graph obtained by  identifying the roots of the $H_s$'s. 
 We call \emph{amalgamation} of $\mbf H$ a partition $\rho$ of the vertex set $V_0$ of $H_0$ such that each block of $\rho$ contains at most one vertex from each $H_s$. We denote by $\cH^{\mathrm{am}}_k\langle J\rangle$ the set couples $(\mbf H, \rho)$, where $\mbf H$ is a $k$-tuple of rooted test graphs and $\rho$ an amalgamation of $\mbf H$. We say in short that $(\mbf H, \rho)$ are am.-test graphs. We set $\dC\cH^{\mathrm{am}} \langle J\rangle= \oplus_{k\geq 1}\dC \cH^{\mathrm{am}}_k\langle J\rangle $ the space of am.-test graphs.

\begin{definition}[Amalgamated traffic distribution of marked graphs]\label{Def:AmalgTrafficDistribution}Let $(G,o) \in  \cGr(\dC^J)$ be a marked graph, where $G= (V,E,\xi)$. The \emph{amalgamated rooted-traffic distribution} $\tau_{(G,o)}$ of $(G,o)$ is the linear map $\tau_{(G,o) } : \dC\cH^{\mathrm{am}} \langle J\rangle\to  \oplus_{k\geq 1} \dC^k$ 
such that, for any  $(\mbf H, \rho)\in \cH_k^{\mathrm{am}} \langle J\rangle$, with above notations,
	\begin{equation}\label{def:tauAmalgG}
		\tau_{(G,o) }[ \mbf  H, \rho ] =  \sum_{\substack{ \phi:V_{H_0}\to V \\ \mathrm{s.t. \ } \phi(1)=i \\ v\overset{\rho} \sim v' \Rightarrow \phi(v) = \phi(v') }}  \Big(  \prod_{e=(v,w)\in E_{H_s}} \xi^{\varepsilon_s(e)}_{\ell_s(e)}\big(\phi_s(w), \phi_s(v)\big)  \Big) \in \dC^k
	\end{equation}
where the sum is over all root preserving maps $\phi:V_{H_0}\to V$ such that $\phi(v) = \phi(v') $ whenever $v$ and $v'$ are in a same block of $\rho$, and where $\phi_s$ denotes the restriction of $\phi$ on $V_s$  for any $s=1, \dots , k$.
\end{definition}

One sees that an obvious analogue of property \eqref{Annealing} still applies. As previously it shows that for random rooted graphs with bounded marks and degree, the amalgamated traffic distribution does not contain more information than the traffic distribution. Our proofs, up to the definition of traffic independence, are always based on a computation on traffic distributions, so their consequence remains valid with amalgamated traffic distributions replacing random rooted-traffic distributions. 

The amalgamated rooted traffic distribution of a random graph forms a consistant sequence, in the sense that if $\rho$ splits as two partitions of disjoint subsets of graphs, it is given by the product of the distributions for the induced am.-test graphs in an obvious way. We call (algebraic, rooted) \emph{amalgamated traffic distribution} a linear map  $\dC\cH^{\mathrm{am}} \langle J\rangle \to  \oplus_{k\geq 1} \dC^k$ that satisfies these consistance properties. Endowing this space with the product topology, we can hence consider the space of random traffic distribution $\cP\big( \mathrm{Lin}( \dC\cH^{\mathrm{am}} \langle J\rangle,  \oplus_{k\geq 1} \dC^k)\big)$ which is the true setting for Corollary \ref{cor:trafficind}.

Firstly, we set $\tau^0_{(G,o) }$ as $\tau_{(G,o) }$ but with summation over all injective maps. This injective version is equivalent to the plain one thanks to the same relations as before. Let $\mbf H = H_1\otimes  \cdots  \otimes H_k\in \mathbb C \cH^\bullet\langle J_1\sqcup J_2 \rangle^{\otimes k}$ be a tensor product of rooted test graphs, where $H_s= (V_s,E_s, \gamma_s, \varepsilon_s)$ for any $s=1, \dots , k$ (a tensor notation is introduced to ease the formulation below). Denote as before by $H_0=(V,E,\gamma, \varepsilon)$ the test graph obtained by  identifying the roots of the $H_s$'s and $\rho \in \cP(V)$ forming an amalgamation of the test graphs.

Assume that the graph $\mathcal {GCC}(H_0)$ of colored components of $H_0$ is a tree. Denote by $\mathfrak V(H_0^\rho)$ the set of roots of the connected components of $\mathcal {GCC}(H_0^\rho)$. For each $i\in \mathfrak V(H_0^\rho)$ we denote by $H^\rho(i)$ the corresponding colored component, and for each $s=1, \dots , k$, we denote by  $H_s(i)$  the subgraph of $H_s$ whose edges are in $H_0^\rho(i)$ (if it has no edge,  $H_s(i)$ is the test graph rooted in $i$ with no edge), and we set $\mbf H(i) = H_1(i) \otimes \cdots \otimes H_k(i)$. We also denote by $\rho(i)$ the restriction of $\rho$ on the vertices of $H_0(i)$.

\begin{definition} Let $\mu_1, \mu_2$ and $\mu$ the law of random rooted traffic distributions $\tau_1, \tau_2$ and $\tau$ labeled in $J_1, J_2$ and $J_1\sqcup J_2$ respectively. We say that $\mu$ is the free product of $\mu_1$ and $\mu_2$ whenever,  for any $\mbf H = H_1\otimes  \cdots  \otimes H_k\in \mathbb C \cH^\bullet\langle J_1\sqcup J_2 \rangle^{\otimes k}$ and $\rho$ as above,
	$$ \tau^0(\mbf H; \rho) \overset{d.} = \mathbf 1\Big( \mathcal{ GCC}(H_0^\pi) \mathrm{ \ is \ a \ tree} \Big) \prod_{ i \in \mathfrak V(H_0^\rho)} \tau^0_i\big( \mbf H(i) ; \rho(i) \big),$$
where  $(\tau^0_i)_{i \in \mathfrak V(H_0)}$ is an \emph{independent} collection of random rooted traffic distributions (in the sense of classical probability), distributed as $\mu_1$ if the labels of $S(i)$ are in $J_1$ and distributed as $\mu_2$ otherwise. 
\end{definition}

\section{Discretization  and relative entropy}

\label{sec:discretization}

In this section, we give a continuity result for the relative entropy with respect to our discretization procedure of the Euclidean space.

We recall the discretization map  defined in the proof of Theorem \ref{th:SIGMA1R}. Fix a mesh size $\delta >0$ and a window size $\kappa > 0$ with $\kappa/\delta$ integer.  The quantization is a map $\{ \cdot \}_\delta^\kappa : \dR^d \to  \left(\delta \dZ \cap [-\kappa,\kappa)\right)^k \cup \{\omega\}$ with $\omega \in \dR^d \backslash (\delta \dZ)^d$ is any default value such that $\omega^* =\omega$. Assume first that $d=1$. For $|x| \geq \kappa$, we set $\{x\}^\kappa_\delta = \omega$. For $0 \leq x < \kappa$, we set $\{x\}^\kappa_\delta = \lfloor  x / \delta \rfloor \delta$. If $-\kappa < x< 0$, we set $\{x\}^\kappa_\delta = - \{-x\}^\kappa_\delta$. On $\dR^d$, we set $\{(x_1,\ldots,x_d)\}^\kappa_\delta = (\{x_1\}^\kappa_\delta,\ldots,\{x_d\}^\kappa_\delta)$. Finally, if $\kappa/\delta$ is not an integer, we set $\{x\}^\kappa_\delta = \{x\}^{\kappa'}_{\delta}$ with $\kappa' = \delta \lceil \kappa/\delta \rceil$.

\begin{lemma}\label{le:discretization}
Let  $X,Y $ be  random variables on $\dR^d$. Then, 
$$
\lim_{\kappa \to \infty} \lim_{\delta \to 0} \DKL(\{X\}^\kappa_\delta | \{Y\}^\kappa_\delta  ) = \DKL(X|Y).
$$
\end{lemma}

\begin{proof}
 The proof is based on the variational formula, for probability measures $p,q$ on $\dR^d$,
\begin{equation}\label{eq:varDKL}
\DKL(p|q)=\sup_{\phi\in \mathcal C_{b}(\mathbb R^{d})}\left\{\int \phi(x) dp(x)-\ln\int e^{\phi(x)} dq(x)\right\}.
\end{equation}

By construction $\{X\}^\kappa_\delta$ and $\{Y\}^\kappa_\delta$ are random variables on $L^\kappa_\delta = \left(\delta \dZ \cap [-\kappa,\kappa)\right)^d \sqcup \{\omega\}$. We may omit the explicit dependence of $L$ and other parameters on $(\kappa,\delta)$ for ease of notation. For $l \in L$, let $B_l = \{ x \in \dR^d : \{x \}^\kappa_\delta = l\}$ be the bin associated to $l$. By construction, if $l \in L \backslash \{\omega\}$, we have $\bar B_l = [l_1+\delta]\times \cdots \times [l_d + \delta] $ and $B_{\omega} = \{ x \in \dR^d: |x|_{\infty} \geq \kappa \}$, where $|x|_\infty = \max_i |x_i|$.

%We first notice that in \eqref{eq:varDKL}, we may replace $\cC_b$ by $\cC_0$ the compactly supported functions. Also

Let $p,q$ be the distributions of $X $ and $Y$. Then \eqref{eq:varDKL} implies that 
$$
\DKL(\{X\}^\kappa_\delta | \{Y\}^\kappa_\delta  ) = \sup_{\phi \in S^\kappa_\delta} \left\{\int \phi(x) dp(x)-\ln\int e^{\phi(x)} dq(x)\right\}.
$$
where $\cS^\kappa_\delta$ is the set of functions on $\dR^d$ such that $\phi$ is constant on $B_l$ for all $l \in L$. We note also that \eqref{eq:varDKL} implies that 
$$
\DKL(X|Y) = \sup_{\kappa > 0} \sup_{\phi \in \cC^\kappa}  \left\{\int \phi(x) dp(x)-\ln\int e^{\phi(x)} dq(x)\right\},
$$
where $\cC^\kappa$ is the set of continuous functions with support in $[-\kappa,\kappa)^d$. To prove the lemma, it is thus sufficient to check that for any $\kappa >0$,  
\begin{equation}\label{eq:DDKL1}
\sup_{\phi \in S^\kappa_\delta} \left\{\int \phi(x) dp(x)-\ln\int e^{\phi(x)} dq(x)\right\} \leq \sup_{\phi \in \cC_b}  \left\{\int \phi(x) dp(x)-\ln\int e^{\phi(x)} dq(x)\right\}.
\end{equation}
and 
\begin{equation}\label{eq:DDKL2}\sup_{\phi \in \cC^\kappa}  \left\{\int \phi(x) dp(x)-\ln\int e^{\phi(x)} dq(x)\right\} \leq \liminf_{\delta \to 0} \sup_{\phi \in S^\kappa_\delta} \left\{\int \phi(x) dp(x)-\ln\int e^{\phi(x)} dq(x)\right\}.
\end{equation}
Fix $\veps >0$ and take $\phi \in S^\kappa_\delta$ such that the left-hand side of \eqref{eq:DDKL1} is finite, that is $\phi$ positive on some $B_{l_\phi}$ which intersects the support of $q$. By taking a convolution with a compactly supported smooth kernel, we find that there exists a bounded continuous function $\phi_{\veps} \in \cC_b$ such that $| \phi - \phi_\veps |_{\infty} \leq \veps$. Then, for all $\veps >0$ small enough, $\phi_{\veps}$ positive on $B_{l_\phi}$. It follows by dominated convergence that 
$$
\lim_{\veps \to 0} \left\{\int \phi_\veps(x) dp(x)-\ln\int e^{\phi_\veps(x)} dq(x)\right\} = \int \phi(x) dp(x)-\ln\int e^{\phi} dq(x).
$$
This proves \eqref{eq:DDKL1}. Similarly, take $\phi \in \cC_\kappa$  such  that the left-hand side of \eqref{eq:DDKL2} is finite, that is support $\phi$ intersects the support of $q$. Let $\phi_\delta \in \cS^\kappa_\delta$ be equal to $0$ on $B_\omega$ and, on $B_l$ with $l \in L \backslash \{\omega\}$, to $y_l$ with,
$$
\int_{B_l} e^{\phi(x)} dq(x) = \int_{B_l} e^{y_l} dq(x)
$$
Also, since $\phi$ is uniformly continuous, $\phi_\delta$ converges uniformly toward $\phi$ along any sequence $\delta \to 0$. Again, by dominated convergence, this proves \eqref{eq:DDKL2}.
%$L^\delta = \left(\delta \dZ \cap [-\kappa(\delta),\kappa(\delta))\right)^k \sqcup \{\omega\} $
\end{proof}

\bibliographystyle{plain}

\bibliography{bib}

\end{document}